\documentclass[10pt]{amsart}

\usepackage[text={420pt,660pt},centering]{geometry}
\usepackage[T1]{fontenc}
\usepackage{lmodern}
\usepackage{microtype}
\usepackage{mathtools}
\usepackage{amssymb}
\usepackage{bm}
\usepackage{enumitem}
\usepackage{xcolor}
\usepackage[colorlinks=true,pdfstartview=FitV,linkcolor=blue,citecolor=blue,urlcolor=blue,pagebackref=false]{hyperref}
\usepackage{aliascnt}

\hypersetup{
pdftitle={FRSB in the SK spin glass: convergence to full-interval support at zero temperature},
pdfauthor={Hong-Bin Chen},
pdfsubject={Full replica symmetry breaking in the Sherrington--Kirkpatrick model},
pdfkeywords={Sherrington--Kirkpatrick model, Parisi measure, full replica symmetry breaking, zero temperature, endpoint atom}
}

\allowdisplaybreaks
\newtheorem{theorem}{Theorem}[section]

\newaliascnt{proposition}{theorem}
\newtheorem{proposition}[proposition]{Proposition}
\aliascntresetthe{proposition}

\newaliascnt{lemma}{theorem}
\newtheorem{lemma}[lemma]{Lemma}
\aliascntresetthe{lemma}

\newaliascnt{corollary}{theorem}

\aliascntresetthe{corollary}

\newaliascnt{claim}{theorem}

\aliascntresetthe{claim}

\theoremstyle{definition}
\newaliascnt{definition}{theorem}

\aliascntresetthe{definition}

\theoremstyle{remark}
\newaliascnt{remark}{theorem}
\newtheorem{remark}[remark]{Remark}
\aliascntresetthe{remark}

\numberwithin{equation}{section}

\newcommand{\R}{\mathbb R}
\newcommand{\E}{\mathbb E}
\newcommand{\PP}{\mathbb P}

\newcommand{\1}{\mathbf 1}
\newcommand{\supp}{\operatorname{supp}}
\newcommand{\Var}{\operatorname{Var}}

\newcommand{\sech}{\operatorname{sech}}
\newcommand{\cU}{\mathcal U}
\newcommand{\cP}{\mathcal P}
\newcommand{\cI}{\mathcal I}
\newcommand{\cG}{\mathcal G}
\newcommand{\dd}{\,\mathrm d}
\newcommand{\stackref}[2]{\stackrel{#2}{#1}}
\DeclareRobustCommand{\authorcontactmark}{\hyperlink{author-contact}{\raisebox{0.7ex}{\scriptsize\ensuremath{\dagger}}}}

\title[FRSB in the SK spin glass]{FRSB in the SK spin glass: convergence to full-interval support at zero temperature}
\author[Hong-Bin Chen]{Hong-Bin Chen\authorcontactmark}
\date{}

\begin{document}

\begin{abstract}
We prove full replica symmetry breaking for the zero-field Sherrington--Kirkpatrick model at zero temperature: the Parisi minimizer is absolutely continuous, has a smooth density, and has support $[0,1)$. At inverse temperature $\beta>1$, \cite{Lopatto2026} recently proved that the Parisi measure has support $[0,q_\beta]$. Here, we show $q_\beta$ converges to $1$ as $\beta\to\infty$.
\end{abstract}

\renewcommand{\thefootnote}{\fnsymbol{footnote}}
\maketitle
\footnotetext[2]{\hypertarget{author-contact}{}Email: \href{mailto:hbc236@nyu.edu}{\nolinkurl{hbc236@nyu.edu}}; affiliation: New York University Shanghai, Shanghai, China.}
\renewcommand{\thefootnote}{\arabic{footnote}}
\setcounter{footnote}{0}

\noindent\textbf{Authorship and AI-use disclosure.} This manuscript is a successor to \cite{AIxivPreprint}, which was written entirely by ChatGPT~5.6 from prompts supplied by Hong-Bin Chen and another person who prefers to remain anonymous. The proof arguments and original prose of the present manuscript were generated by the same model from prompts supplied by Hong-Bin Chen. The substantive division of labor may therefore be summarized as follows.
\begin{center}
\textit{ChatGPT~5.6 is the author of the manuscript.}
\end{center}
Since \href{https://info.arxiv.org/help/moderation/index.html\#policy-for-authors-use-of-generative-ai-language-tools}{arXiv's policy on generative AI language tools} does not permit such a tool to be listed as an author, Hong-Bin Chen is only \textit{formally} listed as the author for submission purposes and assumes full responsibility for the submitted text. This formal attribution reflects arXiv's policy rather than the division of labor in producing the manuscript: Hong-Bin Chen's role was limited to prompting, editing, proofreading, and verifying the arguments; in particular, he did not construct the proof arguments. He has read and verified the proofs, although errors or oversights may remain.

\begingroup
\renewcommand{\thefootnote}{\fnsymbol{footnote}}
\setcounter{footnote}{2}
A conditional Lean~4 formalization\footnote{Also written by ChatGPT~5.6, it treats seven explicitly identified analytic inputs as assumptions. They encapsulate Parisi variational theory, PDE and stochastic analysis, convergence arguments, and the fact from \cite{ChenHandschyLerman} that $0$ lies in the support. This is not an assumption-free verification of the entire paper; the repository documents the precise trust boundary and appendix coverage.} machine-checks Theorems~\ref{thm:finite-temperature} and~\ref{thm:zero-temperature} and is available at
\begin{center}
\url{https://github.com/hbchen-math/FRSB_zero_temp_SK}.
\end{center}
\setcounter{footnote}{0}
\endgroup

\enlargethispage{8pt}
\tableofcontents

\section{Introduction}
\label{sec:introduction}

The Sherrington--Kirkpatrick model was introduced in \cite{SherringtonKirkpatrick1975} as a mean-field model of a spin glass. Parisi subsequently proposed the hierarchical order parameter and the variational formula that underlie replica symmetry breaking; see \cite{Parisi1979}. Guerra established the replica-symmetry-breaking interpolation bound in \cite[Theorem 3]{Guerra2003}, and Talagrand proved the matching Parisi formula for the SK and mixed even $p$-spin models in \cite[Theorem 1.1]{Talagrand2006}. Panchenko later extended the formula to general mixed $p$-spin models in \cite[Theorem 1]{Panchenko2014}. The present paper concerns the structure of the unique variational minimizer rather than the value of the variational formula itself.

For completeness, let $\Sigma_N:=\{-1,1\}^N$, let $(g_{ij})_{1\leq i<j\leq N}$ be independent standard Gaussian random variables, and define the zero-field SK Hamiltonian by
\begin{equation*}
 H_N(\sigma):=\frac1{\sqrt N}\sum_{1\leq i<j\leq N}g_{ij}\sigma_i\sigma_j,
 \qquad \forall\sigma\in\Sigma_N.
\end{equation*}
For $\beta>0$, define
\begin{equation*}
 F_N(\beta):=\frac1N\E\log\sum_{\sigma\in\Sigma_N}e^{\beta H_N(\sigma)}
 \quad\text{and}\quad
 G_N:=\frac1N\E\max_{\sigma\in\Sigma_N}H_N(\sigma).
\end{equation*}
Thus $F_N(\beta)$ is the normalized quenched free energy at inverse temperature $\beta$, and $G_N$ is the normalized ground-state energy.

For a Borel probability measure $\mu$ on $[0,1]$, set
\begin{equation*}
\alpha_\mu(s):=\mu([0,s]),
\qquad
0\leq s\leq1.
\end{equation*}
In the normalization used throughout the positive-temperature part, the Parisi PDE is
\begin{equation*}
\partial_su_\mu(s,x)
=-\frac{\beta^2}{2}\left(u_{\mu,xx}(s,x)+\alpha_\mu(s)u_{\mu,x}(s,x)^2\right)
\quad\text{and}\quad
u_\mu(1,x)=\log\cosh x,
\end{equation*}
and the Parisi functional is
\begin{equation}
\mathcal P_\beta(\mu)
:=\log2+u_\mu(0,0)-\frac{\beta^2}{2}\int_0^1s\alpha_\mu(s)\dd s.
\label{eq:intro-finite-functional}
\end{equation}
The positive-temperature Parisi formula states that
\begin{equation*}
 \lim_{N\to\infty}F_N(\beta)
 =\inf_{\mu\in\operatorname{Prob}([0,1])}\mathcal P_\beta(\mu);
\end{equation*}
see \cite[Theorem~1.1]{Talagrand2006}. By \cite[Theorem 1]{AuffingerChenUnique}, the functional in \eqref{eq:intro-finite-functional} has a unique minimizer $\mu_\beta$.

We first record the complete positive-temperature structure theorem from \cite[Theorem~1.1]{Lopatto2026}. Here and below, $f\in C^\infty([0,a))$ means that $f$ is smooth on every compact interval $[0,T]\subset[0,a)$, with derivatives at zero understood from the right.

\begin{theorem}[Positive-temperature structure \cite{Lopatto2026}]
\label{thm:finite-temperature-structure}
For every $\beta>1$, there exist $q_\beta\in(0,1)$, $c_\beta\in(0,1)$, and a nonnegative function $\rho_\beta\in C^\infty([0,q_\beta))$ such that
\begin{equation}
 \mu_\beta(\mathrm d s)=\rho_\beta(s)\dd s+c_\beta\delta_{q_\beta}(\mathrm d s).
 \label{eq:intro-finite-decomposition}
\end{equation}
Moreover,
\begin{equation*}
 \supp\mu_\beta=[0,q_\beta]
 \quad\text{and}\quad
 \overline{\supp_{[0,q_\beta)}(\rho_\beta(s)\dd s)}^{\,[0,q_\beta]}=[0,q_\beta],
\end{equation*}
and the atom at $q_\beta$ is the only atom of $\mu_\beta$.
\end{theorem}

The closure assertion for the density follows from \eqref{eq:intro-finite-decomposition} and $\supp\mu_\beta=[0,q_\beta]$: if $\rho_\beta(s)\dd s$ vanished on a nonempty relatively open subset of $[0,q_\beta)$, then that subset would be disjoint from $\supp\mu_\beta$.

The positive-temperature contribution of the present paper is the following quantitative refinement at the right endpoint.

\begin{theorem}[Quantitative endpoint control]
\label{thm:finite-temperature}
Let $q_\beta$ and $c_\beta$ be as in Theorem~\ref{thm:finite-temperature-structure}.  Then
\begin{equation}
 c_\beta>\frac13
 \quad\text{and}\quad
 \frac{3-c_\beta}{2\beta^2}<1-q_\beta<\frac2{\beta^2}.
 \label{eq:intro-finite-endpoint-bounds}
\end{equation}
In particular, we have $\frac1{\beta^2}<1-q_\beta<\frac2{\beta^2}$ and $\lim_{\beta\to\infty}q_\beta=1$.
\end{theorem}

This shows that the support of the Parisi measure at $\beta$ converges to $[0,1]$ as $\beta\to\infty$. We next describe the structure of the Parisi measure at $\beta=\infty$.

The canonical zero-temperature order parameter is not the ordinary weak limit of the probability measures $\mu_\beta$. Define
\begin{equation}
 \cU:=\left\{\gamma:[0,1)\to[0,\infty):\gamma\text{ is nondecreasing and right-continuous and }\int_0^1\gamma(t)\dd t<\infty\right\}.
 \label{eq:intro-zero-class}
\end{equation}
For $\gamma\in\cU$, our Stieltjes convention is
\begin{equation*}
 \gamma(0-):=0,
 \qquad
 \mathrm d\gamma(\{0\}):=\gamma(0)
 \quad\text{and}\quad
 \mathrm d\gamma([0,t]):=\gamma(t).
\end{equation*}
Thus $\nu:=\mathrm d\gamma$ is locally finite on $[0,1)$ and Tonelli's theorem gives the equivalent weighted finiteness relation
\begin{equation*}
 \int_0^1\gamma(t)\dd t
 =\int_{[0,1)}(1-s)\,\nu(\mathrm d s)<\infty.
\end{equation*}
In the SK normalization $\xi(t):=t^2/2$, let $u_\gamma$ solve
\begin{equation}
 \partial_tu_\gamma+\frac12\left(u_{\gamma,xx}+\gamma(t)u_{\gamma,x}^2\right)=0
 \quad\text{and}\quad
 u_\gamma(1,x)=|x|.
 \label{eq:intro-zero-PDE}
\end{equation}
The zero-temperature Parisi functional is
\begin{equation}
 \cP(\gamma):=u_\gamma(0,0)-\frac12\int_0^1t\gamma(t)\dd t.
\label{eq:intro-zero-functional}
\end{equation}
The zero-temperature Parisi formula states that
\begin{equation*}
 \lim_{N\to\infty}G_N
 =\inf_{\gamma\in\cU}\cP(\gamma);
\end{equation*}
see \cite[Theorem 1]{AuffingerChenGroundState}. By \cite[Theorem 4]{ChenHandschyLerman}, the functional in \eqref{eq:intro-zero-functional} has a unique minimizer $\gamma_\star$. Set $\nu_\star:=\mathrm d\gamma_\star$.

\begin{theorem}[Zero-temperature structure]
\label{thm:zero-temperature}
There exists a nonnegative function $\rho_\infty\in C^\infty([0,1))$ such that
\begin{equation}
 \gamma_\star(0)=0,
 \qquad
 \nu_\star(\mathrm d t)=\mathrm d\gamma_\star(t)=\rho_\infty(t)\dd t
 \quad\text{and}\quad
 \rho_\infty=\gamma_\star'.
 \label{eq:zero-density-decomposition}
\end{equation}
The measure in \eqref{eq:zero-density-decomposition} has full support on $[0,1)$:
\begin{equation}
 \supp_{[0,1)}\nu_\star=[0,1)
 \quad\text{and}\quad
 \overline{\supp_{[0,1)}\nu_\star}^{\,[0,1]}=[0,1].
\label{eq:zero-main}
\end{equation}
In particular, $\nu_\star$ has neither an atom nor a singular continuous component on $[0,1)$.
\end{theorem}

\begin{remark}[The right endpoint at zero temperature]
\label{rem:zero-endpoint-convention}
The half-open interval in Theorem~\ref{thm:zero-temperature} is essential.  The zero-temperature PDE and functional depend only on $\gamma_\star(t)$ for $t<1$.  Therefore adjoining an endpoint atom,
\begin{equation*}
 \overline\nu_{\star,c}
 :=\rho_\infty(t)\dd t+c\delta_1,
 \qquad c\geq0,
\end{equation*}
does not change either object and cannot be detected by the zero-temperature variational problem.  Thus there is no canonical zero-temperature endpoint mass to which $c_\beta$ could converge.  If one insists on a closed-interval convention with an atom at the right endpoint, its coefficient must be selected by an additional limiting prescription.  The positive-temperature boundary layer studied in Proposition~\ref{prop:ft-endpoint-laws} is one such prescription.
\end{remark}

\begin{remark}[Positive- and zero-temperature normalizations]
\label{rem:normalizations}
The scaled measures $\beta\mu_\beta$ converge vaguely on $[0,1)$ to $\nu_\star$ in the zero-temperature theory; see \cite[proof of Theorem 2]{AuffingerChenZeng}. The unscaled probability measures instead satisfy
\begin{equation*}
\lim_{\beta\to\infty}\mu_\beta=\delta_1
\quad\text{weakly on }[0,1].
\end{equation*}
Indeed, for $T<1$, choose $f\in C_c([0,1))$ with $f\geq1$ on $[0,T]$. Vague convergence makes $\int f\,\dd(\beta\mu_\beta)$ bounded, and therefore $\mu_\beta([0,T])=O(\beta^{-1})$. Thus Theorems~\ref{thm:finite-temperature-structure} and~\ref{thm:finite-temperature} concern the positive-temperature probability measures and their endpoint layer, whereas Theorem~\ref{thm:zero-temperature} concerns the scaled zero-temperature Stieltjes measure.
\end{remark}

Vague convergence of finite measures means convergence against every continuous compactly supported function.

The support problem has developed through several complementary results. Auffinger and Chen proved that the origin belongs to the support of every positive-temperature Parisi measure and that the distribution function is smooth on every interval contained in the support; see \cite[Theorems 1 and 2]{AuffingerChenProperties}. Jagannath and Tobasco obtained the variational optimality conditions used below; see \cite[Corollary 3.6 and Proposition 1.1]{JagannathTobasco2017}. At zero temperature, Auffinger, Chen, and Zeng proved that the Parisi measure has infinitely many support points in \cite[Theorem 1]{AuffingerChenZeng}. At positive temperature, Zhou proved the complete interval-and-atom description immediately below the critical temperature in \cite[Theorem 1]{Zhou2025}. Lopatto subsequently proved the complete positive-temperature characterization for every $\beta>1$; this is Theorem~\ref{thm:finite-temperature-structure}.

The support describes more than the number of levels of replica symmetry breaking. In the ultrametric description of the asymptotic Gibbs measure, proved under the Ghirlanda--Guerra identities in \cite[Theorem 1]{PanchenkoUltrametricity}, the points in the support of the Parisi measure correspond to levels in the hierarchy; see \cite[Section~1]{AuffingerChenZeng}. At zero temperature, the same support has a direct interpretation in terms of near-ground states. For every $u\in\supp\nu_\star$ and every $\varepsilon,\eta>0$, there exists $K=K(u,\varepsilon,\eta)>0$ such that, with probability at least $1-Ke^{-N/K}$, there are $\sigma^1,\sigma^2\in\Sigma_N$ satisfying
\begin{equation*}
 |R(\sigma^1,\sigma^2)-u|<\varepsilon
 \quad\text{and}\quad
 \min_{\ell\in\{1,2\}}\frac{H_N(\sigma^\ell)}N\geq\lim_{M\to\infty}G_M-\eta,
\end{equation*}
where $R(\sigma^1,\sigma^2):=N^{-1}\sum_{i=1}^N\sigma_i^1\sigma_i^2$; see \cite[Equation~(6)]{AuffingerChenZeng}. Thus Theorem~\ref{thm:zero-temperature} shows that every $u\in[0,1)$ is approximated by overlaps of pairs of near-ground states.

Full support is also relevant to optimization algorithms. Montanari's incremental approximate-message-passing algorithm assumes that the positive-temperature Parisi distribution function is strictly increasing on its support; under this assumption, it finds a $(1-\varepsilon)$-optimal SK configuration in time $C(\varepsilon)N^2$; see \cite[Assumption~1 and Theorem~2]{MontanariSKOptimization}. Theorem~\ref{thm:finite-temperature-structure} verifies this assumption for every $\beta>1$. At zero temperature, El Alaoui, Montanari, and Sellke formulate the corresponding no-overlap-gap assumption as the existence of a strictly increasing minimizer $\gamma$ of \eqref{eq:intro-zero-functional}; see \cite[Assumption~2]{ElAlaouiMontanariSellke}. Under this assumption, their message-passing algorithm reaches energy within any fixed $\varepsilon>0$ of the limiting optimum; see \cite[Corollary~2.2]{ElAlaouiMontanariSellke}. Since Theorem~\ref{thm:zero-temperature} gives $\supp\mathrm d\gamma_\star=[0,1)$, it implies
\begin{equation*}
 \gamma_\star(t)-\gamma_\star(s)
 =\nu_\star((s,t])>0,
 \qquad \forall\,0\leq s<t<1,
\end{equation*}
and therefore verifies this zero-temperature no-overlap-gap assumption for the SK model.

\medskip

\noindent
\textbf{Organization of the paper.}
We do not reprove any part of the positive-temperature characterization in Theorem~\ref{thm:finite-temperature-structure}. Section~\ref{sec:finite-temperature-proof} uses it as an input and proves only the endpoint identities, the lower bound $c_\beta>1/3$, the two-sided estimate for $1-q_\beta$, and the compactness statement for the endpoint laws. Sections~\ref{sec:zero-temperature-structure} and~\ref{sec:zero-temperature-support} prove Theorem~\ref{thm:zero-temperature}. Appendix~\ref{app:finite-KJ} proves the finite Cole--Hopf inequalities used in the zero-temperature argument, and Appendix~\ref{app:technical-lemmas} collects the proofs of the remaining technical lemmas. In particular, Appendix~\ref{app:ft-endpoint-density} proves the transformed-density monotonicity used in Section~\ref{sec:finite-temperature-proof}.

Finally, Appendix~\ref{app:development} clarifies the relation of the present manuscript to the aiXiv preprint \cite{AIxivPreprint} and to the second and third arXiv versions of Lopatto's recent preprint \cite{Lopatto2026v2,Lopatto2026}. Except for Theorem~\ref{thm:finite-temperature-structure}, which is taken from \cite{Lopatto2026}, the arguments developed here expand those of \cite{AIxivPreprint}, which in turn used \cite{Lopatto2026v2} as an input and source of ideas. Apart from Theorem~\ref{thm:finite-temperature-structure} and well-established results cited in the text, the manuscript is self-contained.

\section{Positive-temperature endpoint estimates}
\label{sec:finite-temperature-proof}

Throughout this section, $q_\beta$, $c_\beta$, and $\rho_\beta$ are the objects supplied by Theorem~\ref{thm:finite-temperature-structure}. Define
\begin{equation*}
 u:=u_{\mu_\beta}
 \quad\text{and}\quad
 \alpha(s):=\mu_\beta([0,s]),
\end{equation*}
and let the optimal diffusion solve
\begin{equation*}
 \dd X_s=\beta^2\alpha(s)u_x(s,X_s)\dd s+\beta\dd W_s
 \quad\text{and}\quad X_0=0.
\end{equation*}
Define
\begin{equation*}
 \Gamma(s):=\E[u_x(s,X_s)^2]
 \quad\text{and}\quad
 C_\beta:=\sech^2(X_{q_\beta}).
\end{equation*}
The variational self-consistency theorem for Parisi minimizers, written in the diffusion representation used here, states that
\begin{equation}
 \Gamma(s)=s,
 \qquad
 s\in\supp\mu_\beta;
 \label{eq:ft-self-consistency}
\end{equation}
see \cite[Theorem~5]{AuffingerChenProperties}. Since Theorem~\ref{thm:finite-temperature-structure} gives $\supp\mu_\beta=[0,q_\beta]$, \eqref{eq:ft-self-consistency} holds for every $s\in[0,q_\beta]$. This identity is the variational input for the endpoint moment computation below.

\begin{proposition}[Endpoint moment identities]
\label{prop:ft-endpoint-identities}
One has
\begin{align}
 1-q_\beta&=\E C_\beta\quad\text{and}\quad \beta^2\E C_\beta^2=1,
 \label{eq:ft-endpoint-first-two}\\
 \text{and}\quad \frac{\E C_\beta^3}{\E C_\beta^2}
 &=\frac2{3-c_\beta}.
 \label{eq:ft-endpoint-atom-ratio}
\end{align}
\end{proposition}

\begin{proof}
Since $\alpha=1$ on $[q_\beta,1]$, the Cole--Hopf formula gives
\begin{equation}
 u(q_\beta,x)=\log\cosh x
 +\frac{\beta^2}{2}(1-q_\beta).
 \label{eq:ft-terminal-value-at-q}
\end{equation}
Consequently,
\begin{equation*}
 u_x(q_\beta,x)=\tanh x
 \quad\text{and}\quad
 u_{xx}(q_\beta,x)=\sech^2x.
\end{equation*}
The first endpoint identity is
\begin{equation*}
 1-q_\beta
 \stackref{=}{\eqref{eq:ft-self-consistency}}
 \E\left[1-u_x(q_\beta,X_{q_\beta})^2\right]
 \stackref{=}{\eqref{eq:ft-terminal-value-at-q}}
 \E C_\beta.
\end{equation*}

For completeness, differentiate the two quantities needed at the endpoint. Put $\mathsf B:=u_x$, $\mathsf C:=u_{xx}$, and $\mathsf D:=u_{xxx}$. It\^o's formula along $X$ gives, for $0<s<q_\beta$,
\begin{align}
 \Gamma'(s)&=\beta^2\E\,\mathsf C(s,X_s)^2,
 \label{eq:ft-unscaled-Gamma-prime}\\
 \frac{\dd}{\dd s}\E\,\mathsf C(s,X_s)^2
 &=\beta^2\E\left[\mathsf D(s,X_s)^2
       -2\alpha(s)\mathsf C(s,X_s)^3\right].
 \label{eq:ft-unscaled-C-moment}
\end{align}
By \cite[Proposition~1(i)]{AuffingerChenProperties}, every positive-order spatial derivative of $u$ is bounded and jointly continuous on $[0,1]\times\R$. Hence the stochastic integrals arising in the two It\^o formulas above are square-integrable martingales and have mean zero. By \eqref{eq:ft-self-consistency}, $\Gamma'(s)=1$ for $0<s<q_\beta$, so the first identity gives
\begin{equation*}
 \E\,\mathsf C(s,X_s)^2=\frac1{\beta^2},
 \qquad 0<s<q_\beta.
\end{equation*}
Continuity as $s\uparrow q_\beta$ gives
\begin{equation*}
 \E C_\beta^2
 =\lim_{s\uparrow q_\beta}\E\,\mathsf C(s,X_s)^2
 \stackref{=}{\eqref{eq:ft-unscaled-Gamma-prime},\,\eqref{eq:ft-self-consistency}}
 \frac1{\beta^2},
\end{equation*}
which proves the second identity in \eqref{eq:ft-endpoint-first-two}. The second line above also gives
\begin{equation*}
 \E\left[\mathsf D(s,X_s)^2-2\alpha(s)\mathsf C(s,X_s)^3\right]=0,
 \qquad 0<s<q_\beta.
\end{equation*}
Now
\begin{equation*}
 \lim_{s\uparrow q_\beta}\alpha(s)=1-c_\beta,
\end{equation*}
while \eqref{eq:ft-terminal-value-at-q} yields
\begin{equation*}
 \mathsf C(q_\beta,X_{q_\beta})=C_\beta
 \quad\text{and}\quad
 \mathsf D(q_\beta,X_{q_\beta})
 =-2\tanh(X_{q_\beta})C_\beta.
\end{equation*}
Passing to the limit gives
\begin{equation*}
 0\stackref{=}{\eqref{eq:ft-unscaled-C-moment}}
 4\E(C_\beta^2-C_\beta^3)
 -2(1-c_\beta)\E C_\beta^3,
\end{equation*}
which rearranges to \eqref{eq:ft-endpoint-atom-ratio}.
\end{proof}

Let $p_\beta$ denote the density of $X_{q_\beta}$ and set
\begin{equation}
 f_\beta(x):=p_\beta(x)e^{-u(q_\beta,x)}.
 \label{eq:ft-endpoint-f}
\end{equation}

\begin{lemma}[Endpoint transformed density]
\label{lem:ft-endpoint-density}
The function $f_\beta$ is continuous, strictly positive, even, integrable, and nonincreasing on $[0,\infty)$.  It is not constant on that half-line.
\end{lemma}

\begin{proof}
See Appendix~\ref{app:ft-endpoint-density}.
\end{proof}

We shall use the following elementary consequence twice. If $f:[0,\infty)\to(0,\infty)$ is continuous, nonconstant, and nonincreasing, $g:[0,\infty)\to\R$ is bounded and strictly decreasing, and $\nu$ is a probability measure assigning positive mass to every nonempty open interval, then
\begin{equation}
 \frac{\E_\nu[f g]}{\E_\nu f}>\E_\nu g.
 \label{eq:ft-strict-covariance}
\end{equation}
Indeed, the difference after multiplication by $\E_\nu f$ is $\operatorname{Cov}_\nu(f,g)$, and
\begin{equation*}
 2\operatorname{Cov}_\nu(f,g)
 =\iint[f(x)-f(y)][g(x)-g(y)]\,\nu(\mathrm d x)\nu(\mathrm d y)>0.
\end{equation*}
The strict inequality follows because continuity and nonconstancy of $f$ give nonempty open intervals $I<J$ on which $f(x)>f(y)$ for every $(x,y)\in I\times J$, while strict decrease of $g$ gives $g(x)>g(y)$ there and $\nu(I)\nu(J)>0$.

\begin{proposition}[Endpoint atom and endpoint scale]
\label{prop:ft-endpoint-quantitative}
For every $\beta>1$,
\begin{equation}
 c_\beta>\frac13
 \quad\text{and}\quad
 \frac{3-c_\beta}{2\beta^2}<1-q_\beta<\frac2{\beta^2}.
 \label{eq:ft-endpoint-quantitative}
\end{equation}
\end{proposition}

\begin{proof}
Constants in \eqref{eq:ft-terminal-value-at-q} cancel from moment ratios, and therefore
\begin{equation}
 \frac{\E C_\beta^3}{\E C_\beta^2}
 \stackref{=}{\eqref{eq:ft-endpoint-f},\,\eqref{eq:ft-terminal-value-at-q}}
 \frac{\int_0^\infty f_\beta(x)\sech^5x\dd x}
        {\int_0^\infty f_\beta(x)\sech^3x\dd x}.
 \label{eq:ft-endpoint-f-ratio}
\end{equation}
Apply \eqref{eq:ft-strict-covariance} under the probability measure with density proportional to $\sech^3x$ on $[0,\infty)$, taking $g(x):=\sech^2x$.  This gives
\begin{equation*}
 \frac{\E C_\beta^3}{\E C_\beta^2}
 \stackref{>}{\eqref{eq:ft-strict-covariance},\,\eqref{eq:ft-endpoint-f-ratio}}
 \frac{\int_0^\infty\sech^5x\dd x}
         {\int_0^\infty\sech^3x\dd x}
 =\frac34.
\end{equation*}
Thus
\begin{equation*}
 \frac2{3-c_\beta}
 \stackref{=}{\eqref{eq:ft-endpoint-atom-ratio}}
 \frac{\E C_\beta^3}{\E C_\beta^2}>\frac34,
\end{equation*}
which proves $c_\beta>1/3$.

The same covariance inequality under the probability measure with density proportional to $\sech x$ gives
\begin{equation*}
 \frac{\E C_\beta^2}{\E C_\beta}
 \stackref{>}{\eqref{eq:ft-strict-covariance}}
 \frac{\int_0^\infty\sech^3x\dd x}
         {\int_0^\infty\sech x\dd x}
 =\frac12.
\end{equation*}
Consequently,
\begin{equation*}
 1-q_\beta
 \stackref{=}{\eqref{eq:ft-endpoint-first-two}}
 \frac1{\beta^2}\frac{\E C_\beta}{\E C_\beta^2}
 <\frac2{\beta^2}.
\end{equation*}
Finally, Lemma~\ref{lem:ft-endpoint-density} gives $p_\beta(x)=e^{u(q_\beta,x)}f_\beta(x)>0$ for every $x\in\R$. Hence $X_{q_\beta}$ is nondegenerate and $C_\beta=\sech^2(X_{q_\beta})$ is positive and nonconstant. The strict Cauchy--Schwarz inequality applied to $C_\beta^{1/2}$ and $C_\beta^{3/2}$ gives $(\E C_\beta^2)^2<\E C_\beta\,\E C_\beta^3$; equality would force $C_\beta^{1/2}$ and $C_\beta^{3/2}$ to be proportional almost surely, and hence would force $C_\beta$ to be constant. Dividing the strict inequality by $\E C_\beta\,\E C_\beta^2>0$ gives
\begin{equation*}
 \frac{\E C_\beta^2}{\E C_\beta}
 <\frac{\E C_\beta^3}{\E C_\beta^2}
 \stackref{=}{\eqref{eq:ft-endpoint-atom-ratio}}
 \frac2{3-c_\beta}.
\end{equation*}
Therefore
\begin{equation*}
 1-q_\beta
 \stackref{=}{\eqref{eq:ft-endpoint-first-two}}
 \frac1{\beta^2}\frac{\E C_\beta}{\E C_\beta^2}
 >\frac{3-c_\beta}{2\beta^2},
\end{equation*}
which is the lower bound in \eqref{eq:ft-endpoint-quantitative}.
\end{proof}

\begin{proposition}[Endpoint laws and subsequential limits]
\label{prop:ft-endpoint-laws}
Define a probability measure on $[0,\infty)$ by
\begin{equation}
 \eta_\beta(\mathrm d x)
 :=\frac{f_\beta(x)\sech^3x\dd x}
        {\int_0^\infty f_\beta(y)\sech^3y\dd y}.
 \label{eq:ft-endpoint-law}
\end{equation}
Equivalently, $\eta_\beta$ is the law of $|X_{q_\beta}|$ after tilting its law by $C_\beta^2$.  Then $(\eta_\beta)_{\beta>1}$ is tight and
\begin{equation}
 c_\beta=3-\frac2{r_\beta}
 \quad\text{and}\quad
 r_\beta:=\int_0^\infty\sech^2x\,\eta_\beta(\mathrm d x).
 \label{eq:ft-c-endpoint-law}
\end{equation}
Consequently, for every sequence $(\beta_n)$ satisfying $\lim_{n\to\infty}\beta_n=\infty$, there exist a subsequence, not relabeled, and a probability measure $\eta_\infty$ on $[0,\infty)$ such that $\eta_{\beta_n}$ converges weakly to $\eta_\infty$.  Equivalently,
\begin{equation}
 \lim_{n\to\infty}
 \int_0^\infty\varphi(x)\,\eta_{\beta_n}(\mathrm d x)
 =
 \int_0^\infty\varphi(x)\,\eta_\infty(\mathrm d x)
 \qquad\text{for every bounded continuous }
 \varphi:[0,\infty)\to\R.
 \label{eq:ft-eta-subsequential-limit}
\end{equation}
Moreover,
\begin{equation}
 \lim_{n\to\infty}c_{\beta_n}
 =
 3-\frac2{\int_0^\infty\sech^2x\,\eta_\infty(\mathrm d x)}
 \in\left[\frac13,1\right].
 \label{eq:ft-c-subsequential-limit}
\end{equation}
In particular, $c_\beta$ converges if and only if the $\sech^2$-moment has the same value for every subsequential weak limit of $\eta_\beta$. Uniqueness of the weak limit of $\eta_\beta$ is sufficient.
\end{proposition}

\begin{proof}
Equations \eqref{eq:ft-endpoint-f-ratio} and \eqref{eq:ft-endpoint-atom-ratio} give
\begin{equation*}
 r_\beta
 \stackref{=}{\eqref{eq:ft-endpoint-law},\,\eqref{eq:ft-endpoint-f-ratio}}
 \frac{\E C_\beta^3}{\E C_\beta^2}
 \stackref{=}{\eqref{eq:ft-endpoint-atom-ratio}}
 \frac2{3-c_\beta},
\end{equation*}
which proves \eqref{eq:ft-c-endpoint-law}.  For $R>0$, monotonicity of $f_\beta$ gives
\begin{equation}
 \eta_\beta([R,\infty))
 \leq
 \frac{\int_R^\infty\sech^3x\dd x}
      {\int_0^R\sech^3x\dd x}.
 \label{eq:ft-endpoint-law-tail}
\end{equation}
Indeed, the numerator in \eqref{eq:ft-endpoint-law} is at most $f_\beta(R)\int_R^\infty\sech^3x\dd x$, whereas its denominator is at least $f_\beta(R)\int_0^R\sech^3x\dd x$.  The right side tends to zero as $R\to\infty$, uniformly in $\beta$, proving tightness.

Prokhorov's theorem supplies a subsequence and a probability measure $\eta_\infty$ satisfying \eqref{eq:ft-eta-subsequential-limit}.  Since $\sech^2$ is bounded and continuous, its moments converge along that subsequence, and \eqref{eq:ft-c-endpoint-law} gives the claimed limit. Proposition~\ref{prop:ft-endpoint-quantitative} implies $3/4<r_\beta<1$, so the limiting value lies in $[1/3,1]$.  The final equivalence follows from \eqref{eq:ft-c-endpoint-law} and compactness of the family of endpoint laws.
\end{proof}

\begin{proof}[Proof of Theorem~\ref{thm:finite-temperature}]
This is Proposition~\ref{prop:ft-endpoint-quantitative}.
\end{proof}

\section{Zero-temperature variational structure}
\label{sec:zero-temperature-structure}

\subsection{The variational problem}
\label{sec:zt-variational}

Theorems~\ref{thm:finite-temperature-structure} and~\ref{thm:finite-temperature} do not identify the zero-temperature order parameter: the positive-temperature probability measures concentrate at the endpoint one, whereas their interior structure survives after multiplication by the inverse temperature. We therefore turn to the zero-temperature variational problem.

The class $\cU$, the PDE solution $u_\gamma$, and the functional $\cP$ are defined in \eqref{eq:intro-zero-class}--\eqref{eq:intro-zero-functional}. Let $\gamma_\star$ be the unique minimizer introduced in Section~\ref{sec:introduction}. From now on, write
\begin{equation}
 \gamma:=\gamma_\star,\qquad \nu:=\mathrm d\gamma
 \quad\text{and}\quad S:=\supp_{[0,1)}\nu.
 \label{eq:zt-minimizer-notation}
\end{equation}

For $\alpha\in\cU$ and a convex one-Lipschitz function $g:\R\to\R$, let $u^{g,\alpha}$ denote the solution of
\begin{equation}
 \begin{aligned}
 \partial_tu^{g,\alpha}(t,x)+\frac12\left(u_{xx}^{g,\alpha}(t,x)+\alpha(t)u_x^{g,\alpha}(t,x)^2\right)&=0,\qquad &&\forall (t,x)\in[0,1)\times\R,\\
 u^{g,\alpha}(1,x)&=g(x), &&\forall x\in\R.
 \end{aligned}
 \label{eq:zt-general-PDE}
\end{equation}
The unadorned function $u$ will always mean the solution with coefficient $\gamma$ and terminal datum $|\cdot|$:
\begin{equation*}
 u:=u^{|\cdot|,\gamma}=u_\gamma.
\end{equation*}
For its spatial derivatives, define
\begin{equation}
 \bigl(\mathsf B(t,x),\mathsf C(t,x),\mathsf D(t,x)\bigr)
 :=\bigl(u_x(t,x),u_{xx}(t,x),u_{xxx}(t,x)\bigr),
 \qquad \forall (t,x)\in[0,1)\times\R.
 \label{eq:zt-BCD-shorthand}
\end{equation}

We first record the regularity properties used below. They follow by approximating $|x|$ with $\lambda^{-1}\log\cosh(\lambda x)$, applying the Cole--Hopf and stochastic-control representations, and then passing to the limit; see \cite[Proposition 2 and Theorem 5]{ChenHandschyLerman} and \cite[Propositions 2 and 3]{AuffingerChenZeng}.

\begin{lemma}[PDE and diffusion facts]
\label{lem:zt-PDE-facts}
Let $u:=u^{|\cdot|,\gamma}$ be the solution of \eqref{eq:zt-general-PDE}. For every $T<1$, the function $u$ is continuous in $(t,x)$ and smooth in $x$ on $[0,T]\times\R$.  It is even, convex, and one-Lipschitz in $x$, with
\begin{equation*}
 -1<u_x(t,x)<1,\qquad u_{xx}(t,x)>0
 \quad\text{and}\quad \lim_{x\to\pm\infty}u_x(t,x)=\pm1.
\end{equation*}
The stochastic differential equation
\begin{equation}
 \dd X_t=\gamma(t)u_x(t,X_t)\dd t+\dd W_t
 \quad\text{and}\quad X_0=0,
 \label{eq:zt-optimal-SDE}
\end{equation}
has a unique weak solution. For every fixed $t>0$, $X_t$ has a smooth, positive, and even density, denoted by
\begin{equation}
 \rho_t(x):=\frac{\mathrm d\mathcal L(X_t)}{\mathrm d x}(x),
 \qquad \forall x\in\R.
 \label{eq:zt-rho-density}
\end{equation}
\end{lemma}

Define the consistency function and the obstacle
\begin{equation}
 \Gamma(t):=\E\,u_x(t,X_t)^2,\qquad h(t):=\Gamma(t)-t
 \quad\text{and}\quad G(q):=\int_q^1h(t)\dd t.
 \label{eq:zt-Gamma-h-G}
\end{equation}
The directional derivative formula of \cite[Propositions 3 and 4]{ChenHandschyLerman} is
\begin{equation}
 D\cP(\gamma)[\eta-\gamma]
 =\frac12\int_0^1(\eta(t)-\gamma(t))h(t)\dd t,
 \qquad \eta\in\cU.
 \label{eq:zt-directional-derivative}
\end{equation}

\begin{proposition}[Variational conditions]
\label{prop:zt-variational-conditions}
Let $u:=u^{|\cdot|,\gamma}$ solve \eqref{eq:zt-general-PDE}, let $X$ solve \eqref{eq:zt-optimal-SDE}, and let $\Gamma,h,G$ be defined by \eqref{eq:zt-Gamma-h-G}. Then the functions in \eqref{eq:zt-Gamma-h-G} satisfy
\begin{equation}
 G(q)\geq0,\qquad \forall q\in[0,1),
 \quad\text{and}\quad
 G(q)=0,\qquad \forall q\in S.
 \label{eq:zt-G-nonnegative}
\end{equation}
At every $q\in S$,
\begin{equation}
 \Gamma(q)=q\quad\text{and}\quad
 \E\,u_{xx}(q,X_q)^2\leq1.
 \label{eq:zt-consistency-stability}
\end{equation}
If $0\leq a<b<1$ and $\gamma$ is constant on $(a,b)$, then
\begin{equation}
 \Gamma'(a+)=\E u_{xx}(a,X_a)^2
 \quad\text{and}\quad
 \Gamma'(b-)=\E u_{xx}(b,X_b)^2.
 \label{eq:zt-gap-endpoint-derivatives}
\end{equation}
Moreover,
\begin{equation*}
 0\in S.
\end{equation*}
\end{proposition}

\begin{proof}
Fix $q<1$ and $c>0$.  Adding an atom $c\delta_q$ to $\nu$ replaces $\gamma$ by $\eta(t):=\gamma(t)+c\1_{[q,1)}(t)$.  Minimality and \eqref{eq:zt-directional-derivative} give $G(q)\geq0$.

Both $2\gamma$ and $0$ belong to $\cU$.  Applying \eqref{eq:zt-directional-derivative} in these two directions gives
\begin{equation}
 \int_0^1\gamma(t)h(t)\dd t=0.
 \label{eq:zt-scaling-stationarity}
\end{equation}
Since $|h(t)|\leq1$ and $\gamma\in L^1$, Fubini's theorem is absolutely applicable:
\begin{align}
 \int_{[0,1)}\int_q^1|h(t)|\dd t\,\nu(\mathrm d q)
 &=\int_0^1\gamma(t)|h(t)|\dd t<\infty,                         \notag\\
 \int_0^1\gamma(t)h(t)\dd t
 &=\int_{[0,1)}G(q)\,\nu(\mathrm d q).                               \label{eq:zt-Fubini-G}
\end{align}
The relations \eqref{eq:zt-scaling-stationarity} and \eqref{eq:zt-Fubini-G} give
\begin{equation*}
 0
 \stackref{=}{\eqref{eq:zt-scaling-stationarity}}
 \int_0^1\gamma(t)h(t)\dd t
 \stackref{=}{\eqref{eq:zt-Fubini-G}}
 \int_{[0,1)}G(q)\,\nu(\mathrm d q).
\end{equation*}
Together with \eqref{eq:zt-G-nonnegative}, this shows that $G=0$ $\nu$-almost everywhere. As $G$ is continuous, it vanishes at every point of $S$.

The two assertions in \eqref{eq:zt-consistency-stability} are exactly the zero-temperature consistency and stability conditions in \cite[Proposition 3]{ChenHandschyLerman}; invoking that proposition also covers the boundary point $q=0$, where a one-sided minimum of $G$ alone would not imply stability. Suppose that $\gamma$ is constant on $(a,b)$. The PDE is classical in time on this interval. Differentiating once in space and applying It\^o's formula along \eqref{eq:zt-optimal-SDE} on a compact subinterval gives
\begin{equation*}
  \dd u_x(t,X_t)=u_{xx}(t,X_t)\dd W_t.
\end{equation*}
Consequently, for $s<t$ in that interval,
\begin{equation*}
 \Gamma(t)-\Gamma(s)
 =\int_s^t\E u_{xx}(r,X_r)^2\dd r.
\end{equation*}
Apply Lemma~\ref{lem:zt-parabolic-stability} with $b<T<T'<1$, $g_n=g=|\cdot|$, and $\alpha_n=\alpha=\gamma$. It follows that $u_{xx}$ is bounded and jointly continuous on $[0,T]\times\R$. The bounded drift on $[0,T]$ gives $\lim_{r\downarrow a}\E|X_r-X_a|^2=0$ and $\lim_{r\uparrow b}\E|X_r-X_b|^2=0$. Therefore
\begin{equation*}
 \lim_{r\downarrow a}\E u_{xx}(r,X_r)^2=\E u_{xx}(a,X_a)^2
 \quad\text{and}\quad
 \lim_{r\uparrow b}\E u_{xx}(r,X_r)^2=\E u_{xx}(b,X_b)^2.
\end{equation*}
Combining these limits with the preceding integral identity proves \eqref{eq:zt-gap-endpoint-derivatives}. Finally, the zero-field argument in the paragraph preceding \cite[Proposition 5]{ChenHandschyLerman} shows that the minimum of the support is zero: a positive minimum would make both zero and that minimum fixed points of the overlap map used there. The diffusion at every positive time has a smooth, strictly positive density by uniform ellipticity and the spatially smooth bounded drift, which supplies the full-support input in that argument. Hence $0\in S$.
\end{proof}

\begin{lemma}[Parabolic stability away from the terminal time]
\label{lem:zt-parabolic-stability}
Let $T<T'<1$, and let $\alpha_n,\alpha\in\cU$ satisfy
\begin{equation*}
 \sup_n\|\alpha_n\|_{L^\infty(0,T')}<\infty
 \quad\text{and}\quad
 \lim_{n\to\infty}\|\alpha_n-\alpha\|_{L^1(0,1)}=0.
\end{equation*}
Let $g_n,g:\R\to\R$ be convex one-Lipschitz functions such that $\lim_{n\to\infty}\|g_n-g\|_\infty=0$, and define
\begin{equation*}
 v_n:=u^{g_n,\alpha_n}\quad\text{and}\quad v:=u^{g,\alpha}
\end{equation*}
as in \eqref{eq:zt-general-PDE}. Then
\begin{equation}
 \lim_{n\to\infty}\|v_n-v\|_{L^\infty([0,T']\times\R)}=0.
 \label{eq:zt-parabolic-uniform-stability}
\end{equation}
For every integer $k\geq1$ and $R<\infty$,
\begin{equation}
 \lim_{n\to\infty}
 \sup_{\substack{0\leq t\leq T\\|x|\leq R}}
 |\partial_x^kv_n(t,x)-\partial_x^kv(t,x)|=0.
 \label{eq:zt-parabolic-Ck-stability}
\end{equation}
For each $n$ and each $k\geq1$, the functions $\partial_x^kv_n$ and $\partial_x^kv$ are jointly continuous on $[0,T]\times\R$. For every $k\geq2$,
\begin{equation}
 \sup_n\sup_{\substack{0\leq t\leq T\\x\in\R}}
 |\partial_x^kv_n(t,x)|
 +\sup_{\substack{0\leq t\leq T\\x\in\R}}
 |\partial_x^kv(t,x)|<\infty.
 \label{eq:zt-global-parabolic-derivative-bound}
\end{equation}
In addition,
\begin{equation}
 \lim_{n\to\infty}
 \sup_{\substack{0\leq t\leq T\\x\in\R}}
 |v_{n,x}(t,x)-v_x(t,x)|=0.
 \label{eq:zt-global-gradient-convergence}
\end{equation}
\end{lemma}

\begin{proof}
See Appendix~\ref{a.pf.lem:zt-parabolic-stability}.
\end{proof}

For $s\geq0$, define $P_sf(x):=\E f(x+\sqrt sZ)$, where $Z\sim N(0,1)$ has mean zero and variance one, and set $P_0f:=f$. Thus $P_s$ has generator $\partial_{xx}/2$. Define the distribution function and upper-tail function of $Z$ by
\begin{align}
 \phi(y)&:=\PP(Z\leq y)=\frac1{\sqrt{2\pi}}\int_{-\infty}^ye^{-z^2/2}\dd z,\notag\\
 \overline\phi(y)&:=1-\phi(y)=\PP(Z>y),
 \qquad \forall y\in\R.
 \label{eq:standard-Gaussian-functions}
\end{align}
For $\lambda>0$, define
\begin{equation}
 h_\lambda(x):=\lambda^{-1}\log\cosh(\lambda x),
 \qquad \forall x\in\R.
 \label{eq:zt-terminal-regularization}
\end{equation}

\begin{lemma}[Terminal regularization and the third-derivative sign]
\label{lem:zt-regularized-terminal-data}
For every $\lambda>0$, define
\begin{equation*}
 u_\lambda:=u^{\lambda,\gamma\wedge\lambda}=u^{h_\lambda,\gamma\wedge\lambda}.
\end{equation*}
For every $T<1$ and $j\geq0$,
\begin{equation}
 \lim_{\lambda\to\infty}\partial_x^ju_\lambda
 =\partial_x^ju^{|\cdot|,\gamma}
 \quad\text{locally uniformly on }[0,T]\times\R.
 \label{eq:zt-lambda-smooth-convergence}
\end{equation}
In particular, for the unadorned solution $u:=u^{|\cdot|,\gamma}$,
\begin{equation}
 \mathsf D(t,x)=u_{xxx}(t,x)\leq0,
 \qquad \forall (t,x)\in[0,1)\times(0,\infty).
 \label{eq:zt-D-negative}
\end{equation}
\end{lemma}

\begin{proof}
See Appendix~\ref{a.pf.lem:zt-regularized-terminal-data}
\end{proof}

\subsection{\texorpdfstring{The function $Q$ and monotonicity of $Q_{xx}/Q$}{The function Q and monotonicity of Qxx/Q}}
\label{sec:zt-Q-properties}

For every $t\in(0,1)$, define $Q(t,\cdot):\R\to(0,\infty)$ in terms of the density $\rho_t$ from \eqref{eq:zt-rho-density} by
\begin{equation}
 Q(t,x):=\rho_t(x)e^{-\gamma(t)u(t,x)},
 \qquad \forall x\in\R.
 \label{eq:zt-Q-def}
\end{equation}
If $t\in(0,1)$ is a jump point of $\gamma$, define the left limit by
\begin{equation*}
 Q(t-,x):=\rho_t(x)e^{-\gamma(t-)u(t,x)},
 \qquad \forall x\in\R.
\end{equation*}
Thus $Q(t,\cdot)$ uses the right-continuous value $\gamma(t)$, whereas $Q(t-,\cdot)$ uses $\gamma(t-)$.

\begin{lemma}[Heat evolution and jump rule]
\label{lem:zt-Q-dynamics}
Let $u:=u^{|\cdot|,\gamma}$ solve \eqref{eq:zt-general-PDE}, let $X$ solve \eqref{eq:zt-optimal-SDE}, and define $Q$ from $u$, $\gamma$, and the density $\rho_t$ of $X_t$ as above. If $0\leq a<b<1$ and $\gamma(r)=m$ for every $r\in(a,b)$, then
\begin{equation}
 \partial_tQ(t,x)=\frac12\partial_{xx}Q(t,x),
 \qquad \forall (t,x)\in(a,b)\times\R.
 \label{eq:zt-Q-heat}
\end{equation}
If $t\in(0,1)$ is a jump point of $\gamma$ and $\delta:=\gamma(t)-\gamma(t-)>0$, then
\begin{equation}
 Q(t,x)=e^{-\delta u(t,x)}Q(t-,x),
 \qquad \forall x\in\R.
 \label{eq:zt-Q-jump}
\end{equation}
\end{lemma}

\begin{proof}
Suppose that $\gamma(r)=m$ for every $r\in(a,b)$. For every $(t,x)\in(a,b)\times\R$, the Fokker--Planck equation and the Parisi PDE give
\begin{equation*}
 \partial_t\rho_t(x)=\frac12\partial_{xx}\rho_t(x)-m\partial_x(u_x(t,x)\rho_t(x))
 \quad\text{and}\quad
 \partial_tu(t,x)=-\frac12(u_{xx}(t,x)+mu_x(t,x)^2).
\end{equation*}
Differentiating $Q(t,x)=\rho_t(x)e^{-mu(t,x)}$ twice in $x$ and once in $t$ gives $\partial_tQ(t,x)=\partial_{xx}Q(t,x)/2$. If $t\in(0,1)$ is a jump point of $\gamma$, the continuity of $u$ and $\rho$ in $t$ gives
\begin{equation*}
 \frac{Q(t,x)}{Q(t-,x)}
 =e^{-(\gamma(t)-\gamma(t-))u(t,x)},
 \qquad \forall x\in\R,
\end{equation*}
which proves \eqref{eq:zt-Q-jump}.
\end{proof}

Extend $\gamma$ by $\gamma(0-):=0$.  Since $\rho_0=\delta_0$, the right-continuous initial measure for $Q$ is
\begin{equation*}
 Q(0,\mathrm d x):=c_0\delta_0(\mathrm d x)
 \quad\text{and}\quad c_0:=e^{-\gamma(0)u(0,0)}.
\end{equation*}
Applying Lemma~\ref{lem:zt-Q-dynamics} first when $\gamma$ is a step function and then passing to the limit gives the following formula.

\begin{lemma}[Stieltjes Brownian-bridge formula]
\label{lem:zt-bridge-formula}
Let $u:=u^{|\cdot|,\gamma}$ solve \eqref{eq:zt-general-PDE}, let $X$ solve \eqref{eq:zt-optimal-SDE}, and let $Q$ be defined from $u$, $\gamma$, and the density of $X$ as above. For every $t\in(0,1)$,
\begin{equation}
 Q(t,x)=c_0p_t(x)\,
 \E_{0\to x}^{\mathrm{BB},t}
 \exp\left\{-\int_{(0,t]}u(s,B_s)\,\mathrm d\gamma(s)\right\},
 \qquad \forall x\in\R,
 \label{eq:zt-Stieltjes-bridge}
\end{equation}
where $p_t(x):=(2\pi t)^{-1/2}e^{-x^2/(2t)}$ and the expectation is over a Brownian bridge from $0$ to $x$ in time $t$.  The atom at zero is contained in $c_0$, while an atom at $t$ is included in the integral, consistently with the right-continuous convention.

There exist nondecreasing step functions $\gamma_n\in\cU$ with the following property. Let $u_n:=u^{|\cdot|,\gamma_n}$ solve \eqref{eq:zt-general-PDE}, let $X^n$ solve \eqref{eq:zt-optimal-SDE} with $(\gamma,u,X)$ replaced by $(\gamma_n,u_n,X^n)$, let $\rho_t^n$ be the density of $X_t^n$, and define
\begin{equation*}
 Q_n(t,x):=\rho_t^n(x)e^{-\gamma_n(t)u_n(t,x)},
 \quad\text{and}\quad
 Q_n(t-,x):=\rho_t^n(x)e^{-\gamma_n(t-)u_n(t,x)},
 \qquad \forall (t,x)\in(0,1)\times\R.
\end{equation*}
For every continuity point $t\in(0,1)$ of $\gamma$ and every compact $K\subset\R$,
\begin{equation}
 \lim_{n\to\infty}\|Q_n(t,\cdot)-Q(t,\cdot)\|_{C^2(K)}=0.
 \label{eq:zt-Q-C2-convergence}
\end{equation}
If $t\in(0,1)$ is a jump point of $\gamma$, then
\begin{equation*}
 \lim_{n\to\infty}\|Q_n(t-,\cdot)-Q(t-,\cdot)\|_{C^2(K)}=0,
 \qquad \forall K\subset\R\text{ compact}.
\end{equation*}
\end{lemma}

\begin{proof}
See Appendix~\ref{a.pf.lem:zt-bridge-formula}.
\end{proof}

For every $t\in(0,1)$, define $\mathfrak r(t,\cdot),H(t,\cdot):\R\to\R$ by
\begin{equation}
 \mathfrak r(t,x):=-\partial_x\log Q(t,x)
 \quad\text{and}\quad
 H(t,x):=\frac{Q_{xx}(t,x)}{Q(t,x)},
 \qquad \forall x\in\R.
 \label{eq:zt-r-H-def}
\end{equation}
Since $Q(t,x)>0$, direct differentiation gives
\begin{align}
 \mathfrak r_x(t,x)
 &=-\partial_{xx}\log Q(t,x)
 =-\frac{Q_{xx}(t,x)}{Q(t,x)}
   +\left(\frac{Q_x(t,x)}{Q(t,x)}\right)^2
 =-H(t,x)+\mathfrak r(t,x)^2,\notag\\
 H(t,x)&=\mathfrak r(t,x)^2-\mathfrak r_x(t,x),
 \qquad \forall x\in\R.
 \label{eq:zt-r-H-identity}
\end{align}

\begin{proposition}[Monotonicity of $Q_{xx}/Q$]
\label{prop:zt-H-monotonicity}
Let $u:=u^{|\cdot|,\gamma}$ solve \eqref{eq:zt-general-PDE}, let $X$ solve \eqref{eq:zt-optimal-SDE}, and define $Q$, $\mathfrak r$, and $H$ from $(u,\gamma,X)$ as above. For every $t\in(0,1)$, the function $Q(t,\cdot)$ is positive, even, and log-concave. Moreover,
\begin{equation}
 x\longmapsto H(t,x)=\frac{Q_{xx}(t,x)}{Q(t,x)}
 \quad\text{is nondecreasing on }(0,\infty).
 \label{eq:zt-H-monotone}
\end{equation}
If $\gamma$ is constant on a positive-length interval $(a,b)$, then $H(t,\cdot)$ is strictly increasing in the order sense on $(0,\infty)$ for every $t\in(a,b)$.
\end{proposition}

\begin{proof}
\smallskip
\noindent\emph{Step 1: Positivity, evenness, and log-concavity for bounded step functions $\gamma$.}

We first assume that $\gamma$ is bounded and that there are $N\geq1$, $0=t_0<t_1<\cdots<t_N=1$, and $0\leq m_0\leq m_1\leq\cdots\leq m_{N-1}<\infty$ such that
\begin{equation}
 \gamma(t)=m_j,
 \qquad \forall t\in[t_j,t_{j+1}),\quad 0\leq j<N.
 \label{eq:zt-step-gamma-assumption}
\end{equation}
Under \eqref{eq:zt-step-gamma-assumption}, we verify the conclusions successively on $(0,t_1)$, on each $(t_j,t_{j+1})$, and at each jump point $t_j$. The inequality $\mathsf D\leq0$ needed at the jumps is \eqref{eq:zt-D-negative}, proved from the regularized terminal data in Lemma~\ref{lem:zt-regularized-terminal-data}. After completing the proof under \eqref{eq:zt-step-gamma-assumption}, we pass to an arbitrary $\gamma\in\cU$ using Lemma~\ref{lem:zt-bridge-formula}.

For every $t\in(0,t_1)$, equation \eqref{eq:zt-Q-heat} and the initial measure $Q(0,\mathrm d x)=c_0\delta_0(\mathrm d x)$ give
\begin{equation*}
 Q(t,x)
 \stackref{=}{\eqref{eq:zt-Q-heat}}
 c_0p_t(x)=\frac{c_0}{\sqrt{2\pi t}}e^{-x^2/(2t)},
 \qquad \forall x\in\R.
\end{equation*}
Consequently,
\begin{equation*}
 \partial_x\log Q(t,x)=-\frac xt
 \quad\text{and}\quad
 \partial_{xx}\log Q(t,x)=-\frac1t,
 \qquad \forall x\in\R,
\end{equation*}
and \eqref{eq:zt-r-H-identity} gives
\begin{equation*}
 \mathfrak r(t,x)=\frac xt\quad\text{and}\quad
 H(t,x)=\frac{x^2}{t^2}-\frac1t,
 \qquad \forall x\in\R.
\end{equation*}
Thus $Q(t,\cdot)$ is positive, even, and log-concave, and $H(t,\cdot)$ is increasing on $(0,\infty)$. If $\gamma=m$ on $(a,b)$ and $a<r<t<b$, then \eqref{eq:zt-Q-heat} and the definition of $P_s$ give
\begin{equation}
 Q(t,x)=P_{t-r}Q(r,\cdot)(x)
 =\int_\R p_{t-r}(x-y)Q(r,y)\dd y,
 \qquad \forall x\in\R.
 \label{eq:zt-Q-heat-convolution}
\end{equation}
Since $(x,y)\mapsto p_{t-r}(x-y)Q(r,y)$ is log-concave whenever $Q(r,\cdot)$ is log-concave, the Pr\'ekopa--Leindler theorem applied to \eqref{eq:zt-Q-heat-convolution} shows that $Q(t,\cdot)$ is log-concave. If $t\in(0,1)$ is a jump point of $\gamma$ and $\delta_t:=\gamma(t)-\gamma(t-)>0$, then \eqref{eq:zt-Q-jump} gives
\begin{equation*}
 \log Q(t,\cdot)=\log Q(t-,\cdot)-\delta_tu(t,\cdot),
\end{equation*}
which is concave because $\log Q(t-,\cdot)$ is concave and $u(t,\cdot)$ is convex. Both operations preserve evenness.

\smallskip
\noindent\emph{Step 2: Preservation of \eqref{eq:zt-H-monotone} by the heat equation.}

Fix $a<r<t<b$, assume that $H(r,\cdot)$ is nondecreasing on $(0,\infty)$, put $\tau:=t-r$, and define
\begin{align*}
 \widetilde Q(x)&:=Q(t,x)=P_\tau Q(r,\cdot)(x)\quad\text{and}\quad
 \widetilde{\mathfrak r}(x):=-\partial_x\log\widetilde Q(x),
 \qquad \forall x\in\R,\\
 \widetilde H(x)&:=\frac{\widetilde Q_{xx}(x)}{\widetilde Q(x)}=H(t,x)
 =\widetilde{\mathfrak r}(x)^2-\widetilde{\mathfrak r}_x(x),
 \qquad \forall x\in\R.
\end{align*}
Estimate \eqref{eq:zt-Q-Gaussian-derivative-bound}, with the explicit jump multiplier when necessary, makes $Q(r,\cdot),Q_x(r,\cdot),Q_{xx}(r,\cdot)$ absolutely integrable against every heat kernel below. We may therefore differentiate under the convolution. Since $Q_{xx}(r,y)=Q(r,y)H(r,y)$ for every $y\in\R$,
\begin{equation}
 \widetilde H(x)=\frac{(P_\tau Q(r,\cdot))_{xx}(x)}{P_\tau Q(r,\cdot)(x)}
 =\frac{P_\tau(Q_{xx}(r,\cdot))(x)}{P_\tau Q(r,\cdot)(x)}
 \stackref{=}{\eqref{eq:zt-r-H-identity}}
 \frac{P_\tau(Q(r,\cdot)H(r,\cdot))(x)}{P_\tau Q(r,\cdot)(x)},
 \qquad \forall x\in\R.
 \label{eq:zt-H-heat-ratio}
\end{equation}
For every even function $f$ for which the integral below is finite, folding the Gaussian convolution at zero gives
\begin{equation*}
 P_\tau f(x)=\int_0^\infty\mathcal K_\tau(x,y)f(y)\dd y,
 \qquad \forall x>0,
\end{equation*}
where
\begin{equation*}
 \mathcal K_\tau(x,y):=p_\tau(x-y)+p_\tau(x+y)
 =\frac{2}{\sqrt{2\pi\tau}}e^{-(x^2+y^2)/(2\tau)}\cosh(xy/\tau),
 \qquad \forall (x,y)\in(0,\infty)^2.
\end{equation*}
Its logarithmic mixed derivative is
\begin{equation*}
 \partial_x\partial_y\log\mathcal K_\tau(x,y)
 =\frac1\tau\tanh(xy/\tau)+\frac{xy}{\tau^2}\sech^2(xy/\tau)>0
\end{equation*}
for $x,y>0$. Recall that a positive kernel $\mathcal K$ on $(0,\infty)^2$ is strictly totally positive of order two, abbreviated strictly TP$_2$, if
\begin{equation}
 \mathcal K(x_1,y_1)\mathcal K(x_2,y_2)
 -\mathcal K(x_1,y_2)\mathcal K(x_2,y_1)>0
 \label{eq:zt-strict-TP2-definition}
\end{equation}
whenever $0<x_1<x_2$ and $0<y_1<y_2$; see \cite{Karlin1968} for the general theory of totally positive kernels. In the present case,
\begin{align*}
 &\log\frac{\mathcal K_\tau(x_1,y_1)\mathcal K_\tau(x_2,y_2)}
 {\mathcal K_\tau(x_1,y_2)\mathcal K_\tau(x_2,y_1)}=\int_{x_1}^{x_2}\int_{y_1}^{y_2}
 \partial_x\partial_y\log\mathcal K_\tau(x,y)\dd y\dd x>0.
\end{align*}
Exponentiating proves \eqref{eq:zt-strict-TP2-definition} for $\mathcal K_\tau$. If $0<x_1<x_2$, then \eqref{eq:zt-H-heat-ratio} gives
\begin{align}
 &P_\tau Q(r,\cdot)(x_1)P_\tau Q(r,\cdot)(x_2)
 [\widetilde H(x_2)-\widetilde H(x_1)]=\int_{0<y_1<y_2}Q(r,y_1)Q(r,y_2)
 [H(r,y_2)-H(r,y_1)]\notag\\
 &\qquad\quad\times\left[
 \mathcal K_\tau(x_2,y_2)\mathcal K_\tau(x_1,y_1)
 -\mathcal K_\tau(x_2,y_1)\mathcal K_\tau(x_1,y_2)
 \right]\dd y_1\dd y_2.                                      \label{eq:zt-TP2-ratio}
\end{align}
The two bracketed factors in the integrand in \eqref{eq:zt-TP2-ratio} are nonnegative by the monotonicity of $H(r,\cdot)$ and \eqref{eq:zt-strict-TP2-definition}, respectively. Therefore $H(t,\cdot)=\widetilde H$ is nondecreasing on $(0,\infty)$, and it is strictly increasing in the order sense unless $H(r,\cdot)$ is constant.

\smallskip
\noindent\emph{Step 3: Preservation of \eqref{eq:zt-H-monotone} at every jump of $\gamma$.}

To prove that \eqref{eq:zt-H-monotone} is preserved at a jump point $t\in(0,1)$ of $\gamma$, suppose that $\delta_t:=\gamma(t)-\gamma(t-)>0$ and define $Q_t^-,Q_t^+:\R\to(0,\infty)$ by
\begin{equation*}
 Q_t^-(x):=Q(t-,x)
 \quad\text{and}\quad
 Q_t^+(x):=Q(t,x),
 \qquad \forall x\in\R.
\end{equation*}
Define $\mathfrak r_t^\pm,H_t^\pm:\R\to\R$ by
\begin{equation*}
 \mathfrak r_t^\pm(x):=-\partial_x\log Q_t^\pm(x)\quad\text{and}\quad
 H_t^\pm(x):=\frac{(Q_t^\pm)_{xx}(x)}{Q_t^\pm(x)},
 \qquad \forall x\in\R.
\end{equation*}
Thus $\mathfrak r_t^-$ and $H_t^-$ are obtained from $Q(t-,\cdot)$ by the same definitions as $\mathfrak r(t,\cdot)$ and $H(t,\cdot)$, while $\mathfrak r_t^+=\mathfrak r(t,\cdot)$ and $H_t^+=H(t,\cdot)$. Equations \eqref{eq:zt-BCD-shorthand} and \eqref{eq:zt-Q-jump} give
\begin{equation*}
 \mathfrak r_t^+(x)
 \stackref{=}{\eqref{eq:zt-Q-jump}}
 \mathfrak r_t^-(x)+\delta_t\mathsf B(t,x),
 \qquad \forall x\in\R,
\end{equation*}
and therefore
\begin{align}
 H_t^+(x)
 &=H_t^-(x)+2\delta_t\mathfrak r_t^-(x)\mathsf B(t,x)
 +\delta_t^2\mathsf B(t,x)^2-\delta_t\mathsf C(t,x),            \label{eq:zt-H-jump}\\
 (H_t^+)_x(x)
 &=(H_t^-)_x(x)+2\delta_t\bigl((\mathfrak r_t^-)_x(x)\mathsf B(t,x)
 +\mathfrak r_t^-(x)\mathsf C(t,x)\bigr)\notag\\
 &\qquad+2\delta_t^2\mathsf B(t,x)\mathsf C(t,x)-\delta_t\mathsf D(t,x),
 \qquad \forall x\in\R.                                      \label{eq:zt-Hx-jump}
\end{align}
For every $x>0$, evenness and log-concavity give $\mathfrak r_t^-(x),(\mathfrak r_t^-)_x(x)\geq0$, while evenness and convexity give $\mathsf B(t,x),\mathsf C(t,x)\geq0$. Lemma~\ref{lem:zt-regularized-terminal-data} gives \eqref{eq:zt-D-negative}. Consequently every term on the right side of \eqref{eq:zt-Hx-jump} is nonnegative, so each jump in \eqref{eq:zt-step-gamma-assumption} preserves \eqref{eq:zt-H-monotone}. This completes the proof that $Q(t,\cdot)$ is positive, even, and log-concave and that $H(t,\cdot)$ is nondecreasing on $(0,\infty)$ under \eqref{eq:zt-step-gamma-assumption}.

\smallskip
\noindent\emph{Step 4: Passage from \eqref{eq:zt-step-gamma-assumption} to arbitrary $\gamma\in\cU$.}

Now let $\gamma\in\cU$ be arbitrary, and choose the step functions $\gamma_n$ in Lemma~\ref{lem:zt-bridge-formula}. For every continuity point $t\in(0,1)$ of $\gamma$, equation \eqref{eq:zt-Q-C2-convergence} gives locally uniform convergence of $Q_n(t,\cdot)$ and $Q_{n,xx}(t,\cdot)$. The positivity of $Q(t,\cdot)$ follows from its definition in terms of $\rho_t$, while evenness and log-concavity pass to the limit. Moreover, for $0<x_1<x_2$,
\begin{equation*}
 H(t,x_1)=\lim_{n\to\infty}H_n(t,x_1)
 \stackref{\leq}{\text{step case}}
 \lim_{n\to\infty}H_n(t,x_2)=H(t,x_2).
\end{equation*}
If $t\in(0,1)$ is a jump point of $\gamma$, the same argument first gives these conclusions for $Q(t-,\cdot)$. Equation \eqref{eq:zt-Q-jump} shows that $Q(t,\cdot)$ is positive and even and that $\log Q(t,\cdot)=\log Q(t-,\cdot)-\delta_tu(t,\cdot)$ is concave. Since $Q(t-,\cdot)$ is smooth and positive, $(H_t^-)_x(x)\geq0$ for every $x>0$; equations \eqref{eq:zt-H-jump}--\eqref{eq:zt-Hx-jump} and \eqref{eq:zt-D-negative} give $(H_t^+)_x(x)\geq0$ for every $x>0$.

\smallskip
\noindent\emph{Step 5: Strict increase of $H(t,\cdot)$ when $\gamma$ is constant on $(a,b)$.}

Finally, suppose that $\gamma$ is constant on $(a,b)$ and let $a<s<t<b$. Then
\begin{equation*}
 Q(t,x)=P_{t-s}Q(s,\cdot)(x),
 \qquad \forall x\in\R.
\end{equation*}
The function $H(s,\cdot)$ cannot be constant. Indeed, if $H(s,x)=c$ for every $x\in\R$, then $\partial_{xx}Q(s,x)=cQ(s,x)$ for every $x\in\R$. If $c>0$, every positive even solution grows like a hyperbolic cosine; if $c=0$, it is affine; and if $c<0$, every nonzero solution changes sign. None is positive, even, and integrable. Since $H(s,\cdot)$ is continuous and nonconstant, $H(s,y_2)-H(s,y_1)>0$ on a subset of $\{0<y_1<y_2\}$ with positive two-dimensional Lebesgue measure. Applying \eqref{eq:zt-TP2-ratio} with $r=s$ and $\tau=t-s$ gives
\begin{equation*}
 H(t,x_2)-H(t,x_1)
 \stackref{>}{\eqref{eq:zt-TP2-ratio}}0,
 \qquad \forall 0<x_1<x_2.
\end{equation*}
This proves the final claim.
\end{proof}

We continue to use the shorthand in \eqref{eq:zt-BCD-shorthand}. The next lemma gives the It\^o identities for the original right-continuous function $\gamma$, without presupposing its continuity.

\begin{lemma}[It\^o identities for nondecreasing $\gamma$]
\label{lem:zt-Ito-identities-general}
Let $u:=u^{|\cdot|,\gamma}$ solve \eqref{eq:zt-general-PDE}, let $X$ solve \eqref{eq:zt-optimal-SDE}, and define $\Gamma$ by \eqref{eq:zt-Gamma-h-G}. The function $\Gamma$ belongs to $C^1([0,1))$ and
\begin{equation}
 \Gamma'(t)=\E\,\mathsf C(t,X_t)^2,
 \qquad 0\leq t<1.
 \label{eq:zt-Gamma-prime}
\end{equation}
The function on the right side of \eqref{eq:zt-Gamma-prime} is locally absolutely continuous, and, for $0\leq s\leq t<1$,
\begin{equation}
 \Gamma'(t)-\Gamma'(s)
 =\int_s^t\E\left[\mathsf D(r,X_r)^2
             -2\gamma(r)\mathsf C(r,X_r)^3\right]\dd r.
 \label{eq:zt-Gamma-prime-integral}
\end{equation}
Consequently its derivative equals the displayed integrand for almost every $t<1$.  If $\gamma=m$ on an open interval, then $\Gamma$ is twice continuously differentiable there and
\begin{equation}
 \Gamma''(t)=\E\left[\mathsf D(t,X_t)^2-2m\mathsf C(t,X_t)^3\right].
 \label{eq:zt-Gamma-second-x}
\end{equation}
\end{lemma}

\begin{proof}
The semimartingale identities \cite[Lemma~3, equations~(14)--(15)]{ChenHandschyLerman} apply to every $\gamma\in\cU$ and to the terminal datum $|x|$. Under the identifications $\Phi_\gamma=u$, $X_\gamma=X$, and $\xi''\equiv1$, they give
\begin{equation}
 \dd\mathsf B(r,X_r)=\mathsf C(r,X_r)\dd W_r
 \quad\text{and}\quad
 \dd\mathsf C(r,X_r)
 =-\gamma(r)\mathsf C(r,X_r)^2\dd r
 +\mathsf D(r,X_r)\dd W_r.                                      \label{eq:zt-derivative-SDEs}
\end{equation}
Fix $T<1$. The bounds in \eqref{eq:zt-global-parabolic-derivative-bound} imply that all stochastic integrals obtained from \eqref{eq:zt-derivative-SDEs} on $[0,T]$ are square-integrable martingales. It\^o's formula for $\mathsf B(r,X_r)^2$ and $\mathsf C(r,X_r)^2$ therefore gives, for $0\leq s\leq t\leq T$,
\begin{align}
 \Gamma(t)-\Gamma(s)
 &=\int_s^t\E\,\mathsf C(r,X_r)^2\dd r,                               \label{eq:zt-general-Gamma-integral}\\
 \E\,\mathsf C(t,X_t)^2-\E\,\mathsf C(s,X_s)^2
 &=\int_s^t\E\left[\mathsf D(r,X_r)^2
       -2\gamma(r)\mathsf C(r,X_r)^3\right]\dd r.                   \label{eq:zt-general-C-integral}
\end{align}
Lemma~\ref{lem:zt-parabolic-stability} shows that $\mathsf C$ is jointly continuous and bounded on $[0,T]\times\R$. Since $X$ has continuous paths, the map $r\mapsto\E\,\mathsf C(r,X_r)^2$ is continuous. Hence \eqref{eq:zt-general-Gamma-integral} gives $\Gamma\in C^1([0,T])$ and \eqref{eq:zt-Gamma-prime}. Equation \eqref{eq:zt-general-C-integral} then shows that this derivative is absolutely continuous and proves \eqref{eq:zt-Gamma-prime-integral}. Since $T<1$ was arbitrary, these conclusions hold on $[0,1)$.

If $\gamma(t)=m$ for every $t$ in an open interval $I$, then the map
\begin{equation*}
 t\longmapsto\E\left[\mathsf D(t,X_t)^2-2m\mathsf C(t,X_t)^3\right]
\end{equation*}
is continuous on $I$. Differentiating \eqref{eq:zt-Gamma-prime-integral} on $I$ proves \eqref{eq:zt-Gamma-second-x}.
\end{proof}

\subsection{\texorpdfstring{The inequalities for $\mathsf K$ and $\mathsf J$}{The inequalities for K and J}}
\label{sec:zt-KJ}

For fixed $t<1$, equation \eqref{eq:zt-BCD-shorthand} and Lemma~\ref{lem:zt-PDE-facts} show that $x\mapsto\mathsf B(t,x)$ is a smooth increasing bijection from $\R$ onto $(-1,1)$. For every $B\in(-1,1)$, define $x(t,B)$ and $\mathsf c(t,B)$ by
\begin{equation}
 \mathsf B(t,x(t,B)):=B\quad\text{and}\quad \mathsf c(t,B):=\mathsf C(t,x(t,B)).
 \label{eq:zt-slope-curvature}
\end{equation}
Thus $\mathsf C(t,x)$ is a function of $(t,x)\in[0,1)\times\R$, whereas $\mathsf c(t,B)$ is a function of $(t,B)\in[0,1)\times(-1,1)$. Suppose that $\gamma=m$ on an open interval $I\subset(0,1)$. For every $(t,B)\in I\times(-1,1)$, define
\begin{equation}
 \mathsf z(t,B):=-\frac12\mathsf c_B(t,B),\qquad \mathsf K(t,B):=\frac{\mathsf z(t,B)}B-m
 \quad\text{and}\quad \mathsf J(t,B):=\mathsf K(t,B)+B\mathsf K_B(t,B).
 \label{eq:zt-slope-variables}
\end{equation}
Here $\mathsf c_B:=\partial_B\mathsf c$ and $\mathsf K_B:=\partial_B\mathsf K$; a subscript $t$ below means differentiation in $t$ with $B$ fixed. The quotients in \eqref{eq:zt-slope-variables} have smooth even extensions through $B=0$.

We first isolate the positive-temperature Cole--Hopf inequalities that will be rescaled below. For $a>0$, let
\begin{equation}
 \mathcal T_{a,r}f(x):=\frac1a\log\E
 \exp\{a f(x+\sqrt r Z)\},\qquad Z\sim N(0,1),
 \label{eq:zt-Cole-Hopf-operator}
\end{equation}
and define $\mathcal T_{0,r}:=P_r$.

For every smooth, even, strictly convex function $f$ such that $f':\R\to(-1,1)$ is a bijection, define
\begin{equation}
 \mathsf B_f(x):=f'(x),\qquad x_f(B):=\mathsf B_f^{-1}(B),\qquad \mathsf c_f(B):=f''(x_f(B))
 \quad\text{and}\quad \mathsf z_f(B):=-\frac12\partial_B\mathsf c_f(B).
 \label{eq:zt-finite-slope-variables}
\end{equation}
For $a>0$, also define
\begin{equation*}
 \mathsf K_{f,a}(B):=\frac{\mathsf z_f(B)}B-a
 \quad\text{and}\quad \mathsf J_{f,a}(B):=\mathsf K_{f,a}(B)+B\partial_B\mathsf K_{f,a}(B),
 \qquad \forall B\in(-1,1),
\end{equation*}
using the smooth extensions at $B=0$.

\begin{remark}[Relation to Lopatto]
\label{rem:zt-finite-KJ}
The inequalities $\mathsf K,\mathsf K_B,\mathsf J,\mathsf J_B\geq0$ and $\mathsf c\mathsf J_B\leq3\mathsf z\mathsf J$ below originate in \cite[Proposition~3.6]{Lopatto2026v2}. For completeness, Appendix~\ref{app:finite-KJ} gives a streamlined proof containing all regularity, endpoint, and maximum-principle inputs used here. The subsequent scaling to $|x|$, passage to a general zero-temperature order parameter, properties of $Q:=\rho e^{-\gamma u}$, and support argument remain in the main text.

A direct initialization from $|x|$ would avoid the scaling through $\lambda^{-1}\log\cosh(\lambda x)$, but it does not shorten the argument. At terminal time, $u_x(1,x)=\operatorname{sign}(x)$ and $u_{xx}(1,\cdot)=2\delta_0$, so $x\mapsto u_x(1,x)$ has no inverse on $(-1,1)$. After the first positive Cole--Hopf evolution one obtains an explicit expression involving two Gaussian distribution functions; verifying $\mathsf K_B,\mathsf J_B\geq0$ and $\mathsf c\mathsf J_B\leq3\mathsf z\mathsf J$ then requires additional inequalities for their ratios. Starting instead from the smooth function $\log\cosh$ avoids this singular initialization.
\end{remark}

\begin{proposition}[Finite Cole--Hopf inequalities]
\label{prop:zt-finite-KJ}
Begin with $(f,a)=(\log\cosh,1)$.  Perform finitely many operations of either type:
\begin{enumerate}[label=\textup{(\roman*)},leftmargin=2.2em]
\item replace $f$ by $\mathcal T_{a,r}f$ for $r\geq0$, leaving $a$ fixed;
\item replace the parameter $a$ by some $b\in(0,a)$, leaving $f$ fixed.
\end{enumerate}
For every resulting pair $(f,a)$, use the functions in \eqref{eq:zt-finite-slope-variables} and the corresponding $\mathsf K_{f,a}$ and $\mathsf J_{f,a}$. Then, for every $0\leq B<1$,
\begin{equation*}
 \mathsf K_{f,a}(B),\partial_B\mathsf K_{f,a}(B),\mathsf J_{f,a}(B),\partial_B\mathsf J_{f,a}(B)\geq0
 \quad\text{and}\quad \mathsf c_f(B)\partial_B\mathsf J_{f,a}(B)\leq3\mathsf z_f(B)\mathsf J_{f,a}(B).
\end{equation*}
\end{proposition}

\begin{proof}
Appendix~\ref{app:finite-KJ} proves the stronger final inequality $\partial_B(\mathsf c_f^{3/2}\mathsf J_{f,a})\leq0$. Since $\partial_B\mathsf c_f=-2\mathsf z_f$, the two forms are identical:
\begin{equation*}
 \partial_B(\mathsf c_f^{3/2}\mathsf J_{f,a})
 =\mathsf c_f^{1/2}(\mathsf c_f\partial_B\mathsf J_{f,a}-3\mathsf z_f\mathsf J_{f,a}).
\end{equation*}
\end{proof}

\begin{proposition}[Zero-temperature inequalities for $\mathsf K$ and $\mathsf J$]
\label{prop:zt-zero-temp-KJ}
Let $u:=u^{|\cdot|,\gamma}$ solve \eqref{eq:zt-general-PDE}, and define $\mathsf B,\mathsf C,\mathsf D$ by \eqref{eq:zt-BCD-shorthand}. On every open interval on which $\gamma=m$ is constant, the functions in \eqref{eq:zt-slope-curvature}--\eqref{eq:zt-slope-variables} satisfy
\begin{equation}
 \mathsf K\geq0,\quad \mathsf K_B\geq0,\quad \mathsf J\geq0,\quad \mathsf J_B\geq0,\quad
 \mathsf c\mathsf J_B\leq3\mathsf z\mathsf J,\qquad 0\leq B<1.
 \label{eq:zt-zero-temp-five}
\end{equation}
\end{proposition}

\begin{proof}
For $\lambda>0$, recall $h_\lambda$ from \eqref{eq:zt-terminal-regularization}.
Introduce $(\mathcal S_\lambda f)(y):=\lambda f(y/\lambda)$.  A direct change of variables in \eqref{eq:zt-Cole-Hopf-operator} gives
\begin{equation}
 \mathcal S_\lambda\mathcal T_{a,r}f
 =\mathcal T_{a/\lambda,\lambda^2r}\mathcal S_\lambda f
 \quad\text{and}\quad \mathcal S_\lambda h_\lambda=\log\cosh.
 \label{eq:zt-Cole-Hopf-scaling}
\end{equation}
If $\bar f:=\mathcal S_\lambda f$ and $y:=\lambda x$, then
\begin{equation*}
 \mathsf B_{\bar f}(y)=\mathsf B_f(x),\qquad x_{\bar f}(B)=\lambda x_f(B),\qquad
 \mathsf c_f(B)=\lambda\mathsf c_{\bar f}(B)
 \quad\text{and}\quad \mathsf z_f(B)=\lambda\mathsf z_{\bar f}(B).
\end{equation*}
When $a=\lambda\bar a$, one likewise has $\mathsf K_{f,a}=\lambda\mathsf K_{\bar f,\bar a}$ and $\mathsf J_{f,a}=\lambda\mathsf J_{\bar f,\bar a}$. Thus all five inequalities in \eqref{eq:zt-zero-temp-five} are homogeneous under \eqref{eq:zt-Cole-Hopf-scaling}. At the initial pair $(h_\lambda,\lambda)$,
\begin{equation*}
 \mathsf B_{h_\lambda}(x)=\tanh(\lambda x),\qquad \mathsf c_{h_\lambda}(B)=\lambda(1-B^2),\qquad \mathsf z_{h_\lambda}(B)=\lambda B
 \quad\text{and}\quad \mathsf K_{h_\lambda,\lambda}(B)=\mathsf J_{h_\lambda,\lambda}(B)=0.
\end{equation*}
Consequently, for every nondecreasing step function $\gamma$ bounded by $\lambda$,
\begin{equation*}
 \mathsf K,\mathsf K_B,\mathsf J,\mathsf J_B
 \stackref{\geq}{\mathrm{P.\,\ref{prop:zt-finite-KJ}},\,\eqref{eq:zt-Cole-Hopf-scaling}}0
 \quad\text{and}\quad
 \mathsf c\mathsf J_B
 \stackref{\leq}{\mathrm{P.\,\ref{prop:zt-finite-KJ}},\,\eqref{eq:zt-Cole-Hopf-scaling}}
 3\mathsf z\mathsf J.
\end{equation*}
The case $a=0$ follows by taking a decreasing limit of positive values of $a$.

It remains to pass to the terminal condition $|x|$ and to the given $\gamma$. Put $\gamma_\lambda:=\gamma\wedge\lambda$ and approximate it in $L^1(0,1)$ by nondecreasing step functions $\gamma_{\lambda,n}$ taking values in $[0,\lambda]$. When checking a fixed compact subinterval of a gap on which $\gamma=m$, choose the partitions so that $\gamma_{\lambda,n}=m$ on that subinterval for every $n$ once $\lambda>m$. This exact equality is compatible with the $L^1$ approximation by nondecreasing step functions. Define
\begin{equation*}
 u_{\lambda,n}:=u^{h_\lambda,\gamma_{\lambda,n}}
 \quad\text{and}\quad
 u_\lambda:=u^{h_\lambda,\gamma_\lambda},
\end{equation*}
where both functions solve \eqref{eq:zt-general-PDE} with the displayed coefficients and terminal datum $h_\lambda$. At fixed $\lambda$, apply Lemma~\ref{lem:zt-parabolic-stability} after inserting an intermediate time $T'$ with $T<T'<1$. It gives, for every $j\geq0$,
\begin{equation*}
 \lim_{n\to\infty}\partial_x^ju_{\lambda,n}=\partial_x^ju_\lambda
 \quad\text{locally uniformly on }[0,T]\times\R.
\end{equation*}
The limit $\lambda\to\infty$ is \eqref{eq:zt-lambda-smooth-convergence}. We now make the change from $x$ to $B$ quantitative. Fix compact sets $I_t\Subset I$ and $I_B\Subset(-1,1)$. Choose $R<\infty$ such that
\begin{equation*}
 R>1+\max_{(t,B)\in I_t\times I_B}|x(t,B)|,
\end{equation*}
and define
\begin{equation*}
 \delta_R:=\min_{(t,B)\in I_t\times I_B}\min\{B-\mathsf B(t,-R),\mathsf B(t,R)-B\}>0
 \quad\text{and}\quad c_R:=\min_{\substack{t\in I_t\\|x|\leq R}}\mathsf C(t,x)>0.
\end{equation*}
For $j\in\{1,\ldots,5\}$, set
\begin{equation*}
 \varepsilon_{\lambda,n}^{(j)}(R):=\sup_{\substack{t\in I_t\\|x|\leq R}}
 |\partial_x^ju_{\lambda,n}(t,x)-\partial_x^ju(t,x)|.
\end{equation*}
The fixed-$\lambda$ convergence above and \eqref{eq:zt-lambda-smooth-convergence} give
\begin{equation}
 \lim_{\lambda\to\infty}\limsup_{n\to\infty}\varepsilon_{\lambda,n}^{(j)}(R)=0,
 \qquad j\in\{1,\ldots,5\}.
 \label{eq:zt-approx-spatial-derivatives}
\end{equation}
Define $\mathsf B_{\lambda,n}:=\partial_xu_{\lambda,n}$, let $x_{\lambda,n}(t,B)$ be the inverse of $x\mapsto\mathsf B_{\lambda,n}(t,x)$, and set $\mathsf c_{\lambda,n}(t,B):=\partial_{xx}u_{\lambda,n}(t,x_{\lambda,n}(t,B))$. If $\varepsilon_{\lambda,n}^{(1)}(R)<\delta_R/2$, monotonicity gives $x_{\lambda,n}(t,B)\in[-R,R]$ for every $(t,B)\in I_t\times I_B$. For such indices, the mean-value theorem and the definition of $c_R$ give
\begin{equation*}
 c_R|x_{\lambda,n}(t,B)-x(t,B)|
 \leq|\mathsf B(t,x_{\lambda,n}(t,B))-\mathsf B(t,x(t,B))|
 \leq\varepsilon_{\lambda,n}^{(1)}(R).
\end{equation*}
Moreover, if $\varepsilon_{\lambda,n}^{(2)}(R)\leq c_R/2$, then
\begin{equation*}
 \mathsf c_{\lambda,n}(t,B)\geq\mathsf C(t,x_{\lambda,n}(t,B))-\varepsilon_{\lambda,n}^{(2)}(R)\geq\frac{c_R}{2},
 \qquad \forall (t,B)\in I_t\times I_B.
\end{equation*}
For any one of $u_{\lambda,n}$ and $u$, write $v$ for that function and $\mathsf c_v(t,B):=v_{xx}(t,x_v(t,B))$. At fixed $t$,
\begin{equation*}
 \partial_B=\frac1{\mathsf c_v}\partial_x,\qquad (\mathsf c_v)_B=\frac{v_{xxx}}{\mathsf c_v}
 \quad\text{and}\quad
 (\mathsf c_v)_{BB}=\frac{v_{xxxx}}{\mathsf c_v^2}-\frac{v_{xxx}^2}{\mathsf c_v^3},
\end{equation*}
where every spatial derivative on the right is evaluated at $(t,x_v(t,B))$. One further differentiation gives
\begin{equation*}
 (\mathsf c_v)_{BBB}=\frac{v_{xxxxx}}{\mathsf c_v^3}
 -\frac{4v_{xxx}v_{xxxx}}{\mathsf c_v^4}
 +\frac{3v_{xxx}^3}{\mathsf c_v^5}.
\end{equation*}
Consequently, \eqref{eq:zt-approx-spatial-derivatives} and the two preceding lower bounds imply
\begin{equation}
 \lim_{\lambda\to\infty}\limsup_{n\to\infty}
 \|\partial_B^j\mathsf c_{\lambda,n}-\partial_B^j\mathsf c\|_{L^\infty(I_t\times I_B)}=0,
 \qquad j\in\{0,1,2,3\}.
 \label{eq:zt-approx-slope-derivatives}
\end{equation}
The apparent quotient at $B=0$ causes no loss: $\mathsf c(t,\cdot)$ is even, $\mathsf z(t,\cdot)$ is odd, and
\begin{align*}
 \mathsf K(t,B)&=\int_0^1\mathsf z_B(t,\theta B)\dd\theta-m,\\
 \mathsf K_B(t,B)&=\int_0^1\theta\mathsf z_{BB}(t,\theta B)\dd\theta,\qquad
 \mathsf J(t,B)=\mathsf z_B(t,B)-m\quad\text{and}\quad \mathsf J_B(t,B)=\mathsf z_{BB}(t,B).
\end{align*}
These formulas hold at $B=0$ and, together with \eqref{eq:zt-approx-slope-derivatives}, give uniform convergence of $\mathsf z,\mathsf K,\mathsf K_B,\mathsf J$, and $\mathsf J_B$ on $I_t\times I_B$. Since $\gamma_{\lambda,n}=m$ on $I_t$, the inequalities pass first in the limit $n\to\infty$ and then in the limit $\lambda\to\infty$. Finally, $I_t$ and $I_B$ were arbitrary, so the conclusion holds for every $t\in I$ and $0\leq B<1$.
\end{proof}

The following identities will drive the crossing calculation.

\begin{lemma}[Fixed-slope evolution]
\label{lem:zt-fixed-slope}
Let $u:=u^{|\cdot|,\gamma}$ solve \eqref{eq:zt-general-PDE}, and define the functions of $(t,B)$ by \eqref{eq:zt-slope-curvature}--\eqref{eq:zt-slope-variables}. On an interval where $\gamma=m$, derivatives at fixed $B$ satisfy
\begin{equation}
 x_t=-\mathsf K B,\qquad \mathsf c_t=\mathsf c^2\mathsf J
 \quad\text{and}\quad \mathsf z_t=2\mathsf c\mathsf z\mathsf J-\frac12\mathsf c^2\mathsf J_B.
 \label{eq:zt-fixed-B-evolution}
\end{equation}
If
\begin{equation}
 \Phi:=2\mathsf z^2-m\mathsf c\quad\text{and}\quad
 \cG:=2\mathsf c\mathsf z(3\mathsf z\mathsf J-\mathsf c\mathsf J_B),
 \label{eq:zt-Phi-G-def}
\end{equation}
then
\begin{align}
 \Phi_B&=2\mathsf z(2\mathsf J+3m),                                             \label{eq:zt-Phi-B}\\
 \Phi_t&=\mathsf c\mathsf J\Phi+\cG\quad\text{and}\quad \cG\geq0.           \label{eq:zt-Phi-time}
\end{align}
In particular, if $m>0$, then $\Phi$ is strictly increasing on $0<B<1$.
\end{lemma}

\begin{proof}
Fix $t$ in the interval under consideration, let $B\in(-1,1)$, and write $x:=x(t,B)$, so that $B=\mathsf B(t,x)$ and $\mathsf c(t,B)=\mathsf C(t,x)$ by \eqref{eq:zt-slope-curvature}. First,
\begin{equation}
 \mathsf D(t,x)
 \stackref{=}{\eqref{eq:zt-BCD-shorthand}}
 \partial_x\mathsf C(t,x)
 \stackref{=}{\eqref{eq:zt-BCD-shorthand},\,\eqref{eq:zt-slope-curvature}}
 \mathsf c(t,B)\mathsf c_B(t,B)
 \stackref{=}{\eqref{eq:zt-slope-variables}}
 -2\mathsf c(t,B)\mathsf z(t,B).
 \label{eq:zt-D-fixed-slope}
\end{equation}
Differentiating \eqref{eq:intro-zero-PDE} in $x$ at fixed $x$ now gives
\begin{equation}
 \partial_t\mathsf B(t,x)
 \stackref{=}{\eqref{eq:intro-zero-PDE}}
 -\frac12\bigl(\mathsf D(t,x)+2m\mathsf B(t,x)\mathsf C(t,x)\bigr)
 \stackref{=}{\eqref{eq:zt-D-fixed-slope}}
 \mathsf c(t,B)(\mathsf z(t,B)-mB)
 \stackref{=}{\eqref{eq:zt-slope-variables}}
 \mathsf c(t,B)\mathsf K(t,B)B.
 \label{eq:zt-Bt-fixed-x}
\end{equation}
At fixed $B$, differentiation of the identity $\mathsf B(t,x(t,B))=B$ gives
\begin{align}
 0
 &\stackref{=}{\eqref{eq:zt-slope-curvature}}
 \frac{\mathrm d}{\mathrm d t}\mathsf B(t,x(t,B))
 \stackref{=}{\text{chain rule}}
 \partial_t\mathsf B(t,x)+\mathsf C(t,x)x_t(t,B)\notag\\
 &\stackref{=}{\eqref{eq:zt-Bt-fixed-x},\,\eqref{eq:zt-slope-curvature}}
 \mathsf c(t,B)\bigl(\mathsf K(t,B)B+x_t(t,B)\bigr).
 \label{eq:zt-inverse-time-chain}
\end{align}
Since $\mathsf c(t,B)>0$, equation \eqref{eq:zt-inverse-time-chain} yields
\begin{equation}
 x_t(t,B)=-\mathsf K(t,B)B.
 \label{eq:zt-xt-fixed-B}
\end{equation}
We next compute the derivative of $\mathsf c(t,B)=\mathsf C(t,x(t,B))$ while keeping $B$ fixed. Using $\partial_t\mathsf C=\partial_x(\partial_t\mathsf B)$ at fixed $x$ and $\partial_x=\mathsf c\partial_B$ after the change of variables $x\mapsto B=\mathsf B(t,x)$, we obtain
\begin{align}
 \mathsf c_t(t,B)
 &\stackref{=}{\eqref{eq:zt-slope-curvature}}
 \partial_t\mathsf C(t,x)+\mathsf D(t,x)x_t(t,B)\notag\\
 &\stackref{=}{\partial_t\partial_x=\partial_x\partial_t}
 \partial_x\bigl(\partial_t\mathsf B(t,x)\bigr)+\mathsf D(t,x)x_t(t,B)\notag\\
 &\stackref{=}{\eqref{eq:zt-Bt-fixed-x},\,\eqref{eq:zt-slope-curvature}}
 \mathsf c(t,B)\partial_B\bigl(\mathsf c(t,B)\mathsf K(t,B)B\bigr)+\mathsf D(t,x)x_t(t,B)\notag\\
 &\stackref{=}{\eqref{eq:zt-D-fixed-slope},\,\eqref{eq:zt-xt-fixed-B}}
 \mathsf c(t,B)\partial_B\bigl(\mathsf c(t,B)\mathsf K(t,B)B\bigr)-\mathsf c(t,B)\mathsf c_B(t,B)\mathsf K(t,B)B\notag\\
 &\stackref{=}{\text{product rule}}
 \mathsf c(t,B)\left[\mathsf c_B(t,B)\mathsf K(t,B)B+\mathsf c(t,B)\bigl(\mathsf K(t,B)+B\mathsf K_B(t,B)\bigr)\right]\notag\\
 &\qquad-\mathsf c(t,B)\mathsf c_B(t,B)\mathsf K(t,B)B\notag\\
 &=\mathsf c(t,B)^2\bigl(\mathsf K(t,B)+B\mathsf K_B(t,B)\bigr)
 \stackref{=}{\eqref{eq:zt-slope-variables}}
 \mathsf c(t,B)^2\mathsf J(t,B).
 \label{eq:zt-ct-fixed-B}
\end{align}
From this point to the end of the proof, every function without displayed arguments is evaluated at $(t,B)$. Equations \eqref{eq:zt-slope-variables} and \eqref{eq:zt-ct-fixed-B} give
\begin{align}
 \mathsf z_t
 &\stackref{=}{\eqref{eq:zt-slope-variables}}
 -\frac12\partial_t(\mathsf c_B)
 \stackref{=}{\partial_t\partial_B=\partial_B\partial_t}
 -\frac12\partial_B(\mathsf c_t)\notag\\
 &\stackref{=}{\eqref{eq:zt-ct-fixed-B}}
 -\frac12\partial_B(\mathsf c^2\mathsf J)
 \stackref{=}{\eqref{eq:zt-slope-variables}}
 2\mathsf c\mathsf z\mathsf J-\frac12\mathsf c^2\mathsf J_B.
 \label{eq:zt-z-time-fixed-B}
\end{align}
Equations \eqref{eq:zt-xt-fixed-B}, \eqref{eq:zt-ct-fixed-B}, and \eqref{eq:zt-z-time-fixed-B} prove \eqref{eq:zt-fixed-B-evolution}. The definitions in \eqref{eq:zt-slope-variables} also give
\begin{align}
 \mathsf z
 &\stackref{=}{\eqref{eq:zt-slope-variables}}
 (m+\mathsf K)B,\notag\\
 \mathsf z_B
 &\stackref{=}{\eqref{eq:zt-slope-variables}}
 m+\mathsf K+B\mathsf K_B
 \stackref{=}{\eqref{eq:zt-slope-variables}}
 m+\mathsf J.
 \label{eq:zt-z-KJ-identities}
\end{align}
Therefore
\begin{equation*}
 \Phi_B
 \stackref{=}{\eqref{eq:zt-Phi-G-def}}
 4\mathsf z\mathsf z_B-m\mathsf c_B
 \stackref{=}{\eqref{eq:zt-slope-variables},\,\eqref{eq:zt-z-KJ-identities}}
 2\mathsf z(2\mathsf J+3m),
\end{equation*}
which proves \eqref{eq:zt-Phi-B}. For the time derivative,
\begin{align*}
 \Phi_t
 &\stackref{=}{\eqref{eq:zt-Phi-G-def}}
 4\mathsf z\mathsf z_t-m\mathsf c_t\stackref{=}{\eqref{eq:zt-fixed-B-evolution}}
 4\mathsf z\left(2\mathsf c\mathsf z\mathsf J-\frac12\mathsf c^2\mathsf J_B\right)-m\mathsf c^2\mathsf J\\
 &=\mathsf c\mathsf J(2\mathsf z^2-m\mathsf c)+2\mathsf c\mathsf z(3\mathsf z\mathsf J-\mathsf c\mathsf J_B)
 \stackref{=}{\eqref{eq:zt-Phi-G-def}}
 \mathsf c\mathsf J\Phi+\cG.
\end{align*}
Moreover, for $0\leq B<1$,
\begin{equation*}
 \cG=2\mathsf c\mathsf z(3\mathsf z\mathsf J-\mathsf c\mathsf J_B)
 \stackref{\geq}{\eqref{eq:zt-zero-temp-five},\,\eqref{eq:zt-z-KJ-identities}}0,
\end{equation*}
which proves \eqref{eq:zt-Phi-time}. Finally, if $m>0$ and $0<B<1$, then \eqref{eq:zt-zero-temp-five}, \eqref{eq:zt-z-KJ-identities}, and \eqref{eq:zt-Phi-B} give
\begin{equation*}
 \Phi_B
 \stackref{=}{\eqref{eq:zt-Phi-B}}
 2\mathsf z(2\mathsf J+3m)>0.
\end{equation*}
\end{proof}

\subsection{Other technical lemmas}
\label{sec:zt-tails}

For the next section, we need the following two results.

\begin{lemma}[Uniform tail estimates]
\label{lem:zt-tails}
Let $u:=u^{|\cdot|,\gamma}$ solve \eqref{eq:zt-general-PDE}, let $X$ solve \eqref{eq:zt-optimal-SDE}, and define $\rho_t,Q,\mathfrak r,H$ from $(u,\gamma,X)$ as above. Let $(a,b)\subset[0,1)$ be an interval on which $\gamma(t)=m$, and let $I\Subset(a,b)$. Define
\begin{equation*}
 \widehat{\mathsf z}(t,x):=\mathsf z(t,\mathsf B(t,x)),
 \qquad \forall (t,x)\in I\times\R.
\end{equation*}
Uniformly for $t\in I$ as $x\to+\infty$,
\begin{align}
 \log\mathsf C(t,x)&=-\frac{x^2}{2(1-t)}+O(x+\log x),                  \label{eq:zt-C-Gaussian-tail}\\
 \widehat{\mathsf z}(t,x)&=O(x),\qquad \partial_x\widehat{\mathsf z}(t,x)=O(x^2)\quad\text{and}\quad \partial_{xx}\widehat{\mathsf z}(t,x)=O(x^3),\notag\\
 0\leq\mathfrak r_x&\leq\frac1{t-a},\qquad
 \mathfrak r=O(x)\quad\text{and}\quad H=O(1+x^2).             \label{eq:zt-score-tail}
\end{align}
There are constants $c_I,C_I>0$ such that
\begin{equation}
 \rho_t(x)\leq C_Ie^{-c_Ix^2+C_I|x|},
 \qquad t\in I,\ x\in\R.
 \label{eq:zt-rho-Gaussian-tail}
\end{equation}
\end{lemma}

Hence, all weighted integrals and integrations by parts in Subsection~\ref{sec:zt-crossing} are consequently absolutely justified.

\begin{proof}
See Appendix~\ref{a.pf.lem:zt-tails}.
\end{proof}

\begin{lemma}[Regularity of polynomial moments of spatial derivatives]
\label{lem:zt-polynomial-moment-regularity}
Let $u:=u^{|\cdot|,\gamma}$ solve \eqref{eq:zt-general-PDE} and let $X$ solve \eqref{eq:zt-optimal-SDE}. For $k\geq1$, put
\begin{equation*}
 Z_k(t):=\partial_x^ku(t,X_t).
\end{equation*}
For every $T<1$, every $N\geq1$, and every polynomial $P$ in $N$ variables, the map
\begin{equation*}
 M_P(t):=\E P(Z_1(t),\ldots,Z_N(t))
\end{equation*}
is continuous on $[0,T]$.  If in addition $\gamma\in C^r([0,T])$ for an integer $r\geq0$, then
\begin{equation*}
 M_P\in C^{r+1}([0,T]),
\end{equation*}
where derivatives at zero and $T$ are understood one-sided.
\end{lemma}

\begin{proof}
See Appendix~\ref{a.pf.lem:zt-polynomial-moment-regularity}.
\end{proof}

\section{Zero-temperature gap exclusion}
\label{sec:zero-temperature-support}

\subsection{The arbitrary-gap crossing theorem}
\label{sec:zt-crossing}

We now prove the central statement. Throughout this subsection, let $u:=u^{|\cdot|,\gamma}$ solve \eqref{eq:zt-general-PDE}, let $X$ solve \eqref{eq:zt-optimal-SDE}, and let $\rho_t$ and $\Gamma$ be defined by \eqref{eq:zt-rho-density} and \eqref{eq:zt-Gamma-h-G}. Suppose that $\gamma(t)=m>0$ for every $t\in(a,b)$. We use $(\mathsf B,\mathsf C,\mathsf D)$ from \eqref{eq:zt-BCD-shorthand}, $Q$ from \eqref{eq:zt-Q-def}, $(\mathfrak r,H)$ from \eqref{eq:zt-r-H-def}, $(x,\mathsf c)$ from \eqref{eq:zt-slope-curvature}, $(\mathsf z,\mathsf K,\mathsf J)$ from \eqref{eq:zt-slope-variables}, and $(\Phi,\cG)$ from \eqref{eq:zt-Phi-G-def}. All differentiations in $t$ and estimates below are first taken for $t\in I_t$, where $I_t\Subset(a,b)$ is arbitrary. For every $(t,B)\in(a,b)\times[0,1)$, define
\begin{equation}
\begin{aligned}
 \kappa(t,B)&:=\mathsf K(t,B)B=\mathsf z(t,B)-mB,\\
 N(t,B)&:=\mathfrak r(t,x(t,B))+\kappa(t,B),\\
 \text{and}\qquad w(t,B)&:=2\rho_t(x(t,B))\mathsf c(t,B).
\end{aligned}
\label{eq:zt-crossing-kappa-N-w}
\end{equation}
When arguments are omitted below, every function of $(t,B)$ is evaluated at $(t,B)$, while
\begin{equation*}
 \mathfrak r:=\mathfrak r(t,x(t,B)),\qquad
 \mathfrak r_x:=\partial_x\mathfrak r(t,x)\big|_{x=x(t,B)},
 \quad\text{and}\quad
 H:=H(t,x(t,B)).
\end{equation*}

\begin{remark}[Relation with Lopatto's crossing computation]
\label{rem:zt-Lopatto-crossing}
After the change of variables $B=\mathsf B(t,x)$ from \eqref{eq:zt-BCD-shorthand} and \eqref{eq:zt-slope-curvature}, the differentiation of $\int_0^1w(t,B)\Phi(t,B)\dd B$, with $w$ and $\Phi$ defined by \eqref{eq:zt-crossing-kappa-N-w} and \eqref{eq:zt-Phi-G-def}, parallels \cite[Proposition~9.1]{Lopatto2026}. We reproduce every identity used below. The additional zero-temperature inputs are the strict increase of $H$ from Proposition~\ref{prop:zt-H-monotonicity}, with $H$ defined in \eqref{eq:zt-r-H-def}, and the bounds in Lemma~\ref{lem:zt-tails}, which justify the limit $B\uparrow1$ and the integrations by parts.
\end{remark}

\begin{proposition}[Arbitrary-gap crossing]
\label{prop:zt-crossing}
Let $u:=u^{|\cdot|,\gamma}$ solve \eqref{eq:zt-general-PDE}, let $X$ solve \eqref{eq:zt-optimal-SDE}, and let $\rho_t$ and $\Gamma$ be defined by \eqref{eq:zt-rho-density} and \eqref{eq:zt-Gamma-h-G}. Suppose that $\gamma(t)=m>0$ for every $t\in(a,b)$. The auxiliary functions used in the proof are $(\mathsf B,\mathsf C,\mathsf D)$ from \eqref{eq:zt-BCD-shorthand}, $Q$ from \eqref{eq:zt-Q-def}, $(\mathfrak r,H)$ from \eqref{eq:zt-r-H-def}, $(x,\mathsf c)$ from \eqref{eq:zt-slope-curvature}, $(\mathsf z,\mathsf K,\mathsf J)$ from \eqref{eq:zt-slope-variables}, $(\Phi,\cG)$ from \eqref{eq:zt-Phi-G-def}, and $(\kappa,N,w)$ from \eqref{eq:zt-crossing-kappa-N-w}. Then $\Gamma\in C^3(a,b)$ and
\begin{equation}
 \Gamma''(t)=0\quad\Longrightarrow\quad\Gamma'''(t)>0,
 \qquad a<t<b.
 \label{eq:zt-crossing-rule}
\end{equation}
\end{proposition}

\begin{proof}
\smallskip
\noindent\emph{Step 1: Conversion of $\Gamma''(t)$ to an integral over $B$.}

Fix $t\in(a,b)$. For $x\in\R$, put $B:=\mathsf B(t,x)$. Equations \eqref{eq:zt-slope-curvature} and \eqref{eq:zt-D-fixed-slope} give
\begin{equation*}
 \mathsf C(t,x)=\mathsf c(t,B),\qquad
 \mathsf D(t,x)=-2\mathsf c(t,B)\mathsf z(t,B)
 \quad\text{and}\quad
 \dd B=\mathsf C(t,x)\dd x=\mathsf c(t,B)\dd x.
\end{equation*}
In particular,
\begin{align}
 \mathsf D(t,x)^2-2m\mathsf C(t,x)^3
 &\stackref{=}{\eqref{eq:zt-D-fixed-slope},\,\eqref{eq:zt-slope-curvature}}
 4\mathsf c(t,B)^2\mathsf z(t,B)^2-2m\mathsf c(t,B)^3 \notag\\
 &\stackref{=}{\eqref{eq:zt-Phi-G-def}}
 2\mathsf C(t,x)^2\Phi(t,B).
 \label{eq:zt-Gamma-integrand-Phi}
\end{align}
By Lemma~\ref{lem:zt-PDE-facts} and \eqref{eq:zt-BCD-shorthand},
\begin{equation}
 \rho_t(-x)=\rho_t(x),\qquad \mathsf B(t,-x)=-\mathsf B(t,x)
 \quad\text{and}\quad \mathsf C(t,-x)=\mathsf C(t,x),
 \qquad \forall x\in\R.
 \label{eq:zt-crossing-x-parity}
\end{equation}
Consequently, \eqref{eq:zt-slope-curvature}, \eqref{eq:zt-slope-variables}, and \eqref{eq:zt-Phi-G-def} give
\begin{equation}
\begin{aligned}
 x(t,-B)&=-x(t,B)\quad\text{and}\quad \mathsf c(t,-B)=\mathsf c(t,B),\\
 \mathsf z(t,-B)&=-\mathsf z(t,B)\quad\text{and}\quad \Phi(t,-B)=\Phi(t,B),
 \qquad \forall B\in(-1,1).
\end{aligned}
\label{eq:zt-crossing-B-parity}
\end{equation}
We may now derive the identity used below without suppressing the change of variables:
\begin{align}
 \cI(t)&:=\int_0^1w(t,B)\Phi(t,B)\dd B \stackref{=}{\eqref{eq:zt-crossing-kappa-N-w},\,\eqref{eq:zt-slope-curvature}}
 2\int_0^\infty\rho_t(x)\mathsf C(t,x)^2\Phi\bigl(t,\mathsf B(t,x)\bigr)\dd x \notag\\
 &\stackref{=}{\eqref{eq:zt-crossing-x-parity},\,\eqref{eq:zt-crossing-B-parity}}
 \int_\R\rho_t(x)\mathsf C(t,x)^2\Phi\bigl(t,\mathsf B(t,x)\bigr)\dd x \stackref{=}{\eqref{eq:zt-rho-density}}
 \E\!\left[\mathsf C(t,X_t)^2\Phi\bigl(t,\mathsf B(t,X_t)\bigr)\right] \notag\\
 &\stackref{=}{\eqref{eq:zt-Gamma-integrand-Phi},\,\eqref{eq:zt-Gamma-second-x}}
 \frac12\Gamma''(t).
 \label{eq:zt-Gamma-second-I}
\end{align}

\smallskip
\noindent\emph{Step 2: Derivation and justification of the formula for $\cI'(t)$.}

We next compute the derivatives of the weight. Equation \eqref{eq:zt-Q-def} and $\gamma(t)=m$ give
\begin{equation*}
 \partial_x\log\rho_t(x)\big|_{x=x(t,B)}
 \stackref{=}{\eqref{eq:zt-Q-def},\,\eqref{eq:zt-r-H-def}}
 mB-\mathfrak r.
\end{equation*}
Equations \eqref{eq:zt-slope-curvature} and \eqref{eq:zt-slope-variables} give $x_B=\mathsf c^{-1}$ and $\mathsf c_B=-2\mathsf z$. Hence
\begin{equation}
 \frac{w_B}{w}
 =\frac{mB-\mathfrak r-2\mathsf z}{\mathsf c}
 \stackref{=}{\eqref{eq:zt-slope-variables}}
 -\frac{N+\mathsf z}{\mathsf c}
 \quad\text{and}\quad
 N_B
 \stackref{=}{\eqref{eq:zt-crossing-kappa-N-w},\,\eqref{eq:zt-slope-variables}}
 \frac{\mathfrak r_x}{\mathsf c}+\mathsf J.
 \label{eq:zt-wB-NB}
\end{equation}
Equations \eqref{eq:zt-Q-heat} and \eqref{eq:zt-r-H-def} give $Q_t/Q=H/2$ on $(a,b)\times\R$. At fixed $B$, \eqref{eq:zt-fixed-B-evolution} and \eqref{eq:zt-crossing-kappa-N-w} give $x_t=-\kappa$, $\mathsf c_t=\mathsf c^2\mathsf J$, and
\begin{equation*}
 u_t+u_x\,x_t=-\frac12(\mathsf c+mB^2)-B\kappa.
\end{equation*}
Therefore
\begin{align}
 (\log w)_t
 &=m(u_t+Bx_t)+\frac{Q_t}{Q}
   +x_t(\log Q)_x+\frac{\mathsf c_t}{\mathsf c}                         \notag\\
 &\stackref{=}{\mathrm{L.\,\ref{lem:zt-fixed-slope}},\,\eqref{eq:zt-Q-heat}}
 -\frac m2(\mathsf c+mB^2)-mB\kappa
   +\frac12H+\kappa\mathfrak r+\mathsf c\mathsf J.                            \label{eq:zt-logw-first}
\end{align}
Define
\begin{equation}
 R_1:=\kappa\mathfrak r+\frac12\kappa^2
 \quad\text{and}\quad
 R_0:=2\mathsf c\mathsf J-\frac m2\mathsf c-\frac12\mathsf z^2.
 \label{eq:zt-R1-R0-def}
\end{equation}
Because
\begin{equation*}
 -\frac{m^2B^2}{2}-mB\kappa=\frac{\kappa^2-\mathsf z^2}{2},
\end{equation*}
we may rewrite \eqref{eq:zt-logw-first} as
\begin{equation}
 (\log w)_t
 \stackref{=}{\eqref{eq:zt-logw-first}}
 \frac12H+R_1+R_0-\mathsf c\mathsf J,
 \label{eq:zt-logw-R}
\end{equation}

Formally combining \eqref{eq:zt-Phi-time} and \eqref{eq:zt-logw-R} gives
\begin{equation}
 \cI'(t)
 \stackref{=}{\eqref{eq:zt-Phi-time},\,\eqref{eq:zt-logw-R}}
 \frac12\int_0^1w\Phi H\dd B
 +\int_0^1wR_1\Phi\dd B
 +\int_0^1w(R_0\Phi+\cG)\dd B.
 \label{eq:zt-I-prime-decomposition}
\end{equation}
We justify this differentiation by truncation.  For $0<B_0<1$, every quantity is smooth on a compact set $I_t\times[0,B_0]$, so the differentiated identity holds with upper limit $B_0$.  The composite estimates in Lemma~\ref{lem:zt-tails}, together with \eqref{eq:zt-logw-R} and \eqref{eq:zt-Phi-time}, give uniformly for $t\in I_t\Subset(a,b)$
\begin{align*}
 (\log w)_t&=O(1+x^2)\quad\text{and}\quad \Phi_t=O(1+x^4),\\
 |w\Phi|+|\partial_t(w\Phi)|
 &+w|\Phi H|+w|R_1\Phi|+w|R_0\Phi+\cG|
 \leq C_{I_t}w(1+x^4).
\end{align*}
For example, $\partial_t(w\Phi)=w\{(\log w)_t\Phi+\Phi_t\}$, so the first line implies its bound in the second.  The uniform tail estimate \eqref{eq:zt-weighted-uniform-tail} makes the contribution of $B\in(B_0,1)$ tend to zero uniformly in $t$, both before and after differentiation.  Letting $B_0\uparrow1$ therefore proves \eqref{eq:zt-I-prime-decomposition} and shows that $\cI\in C^1(a,b)$. In view of \eqref{eq:zt-Gamma-second-I}, this also proves the asserted regularity $\Gamma\in C^3(a,b)$.

\smallskip
\noindent\emph{Step 3: Positivity of the first two terms in \eqref{eq:zt-I-prime-decomposition} when $\Gamma''(t)=0$.}

Assume from now on that $\Gamma''(t)=0$.  By \eqref{eq:zt-Gamma-second-I},
\begin{equation}
 \int_0^1w\Phi\dd B=0.
 \label{eq:zt-Phi-centered}
\end{equation}
Let $W_0:=\int_0^1w\dd B$.  For any $f$ such that $wf,w\Phi f\in L^1(0,1)$, the centered identity gives
\begin{align}
 \int_0^1w\Phi f\dd B
 =\frac1{2W_0}\int_0^1\!\!\int_0^1
 &w(B_1)w(B_2)[\Phi(B_1)-\Phi(B_2)] \notag\\
 &\times[f(B_1)-f(B_2)]\dd B_1\dd B_2.                       \label{eq:zt-covariance-identity}
\end{align}
By \eqref{eq:zt-Phi-B}, $\Phi$ is strictly increasing. By Proposition~\ref{prop:zt-H-monotonicity}, $H$ is strictly increasing on the gap. Since $w>0$,
\begin{equation}
 \int_0^1w\Phi H\dd B
 \stackref{>}{\eqref{eq:zt-covariance-identity},\,\eqref{eq:zt-Phi-B},\,\mathrm{P.\,\ref{prop:zt-H-monotonicity}}}0.
 \label{eq:zt-H-term-positive}
\end{equation}

The second term in \eqref{eq:zt-I-prime-decomposition} is nonnegative. Indeed,
\begin{equation*}
 (R_1)_B=\mathsf J\mathfrak r+\kappa\frac{\mathfrak r_x}{\mathsf c}+\kappa \mathsf J
 =\mathsf J N+\kappa\frac{\mathfrak r_x}{\mathsf c}
 \stackref{\geq}{\eqref{eq:zt-zero-temp-five}}0,
\end{equation*}
because $\mathsf J,\kappa,N,\mathfrak r_x\geq0$.  Applying \eqref{eq:zt-covariance-identity} with $f=R_1$ gives
\begin{equation}
 \int_0^1wR_1\Phi\dd B
 \stackref{\geq}{\eqref{eq:zt-covariance-identity}}0.
 \label{eq:zt-R1-term-nonnegative}
\end{equation}

\smallskip
\noindent\emph{Step 4: Construction of a nonnegative antiderivative of $w\Phi$ and its upper bound.}

It remains to control the residual term. Define
\begin{equation}
 F(B):=\int_B^1w(\theta)\Phi(\theta)\dd\theta\quad\text{and}\quad
 \tau(B):=\frac{F(B)}{w(B)}.
 \label{eq:zt-F-tau-def}
\end{equation}
At $B=0$, evenness gives $\mathsf z(0)=0$, so $\Phi(0)=-m\mathsf c(0)<0$. On the other hand, $\mathsf z\geq mB$ and $\mathsf c\to0$ as $B\uparrow1$, so $\Phi>0$ near one. By strict monotonicity, there is a unique $B_\Phi\in(0,1)$ such that $\Phi(B_\Phi)=0$. If $B\leq B_\Phi$, then \eqref{eq:zt-Phi-centered} gives
\begin{equation*}
 F(B)
 \stackref{=}{\eqref{eq:zt-Phi-centered}}
 -\int_0^Bw\Phi\dd\theta\geq0;
\end{equation*}
if $B\geq B_\Phi$, nonnegativity follows directly from the definition. Therefore
\begin{equation}
 \tau\geq0.
 \label{eq:zt-tau-lower}
\end{equation}

We prove the complementary bound
\begin{equation}
 \tau\leq \mathsf c\mathsf z.
 \label{eq:zt-tau-upper}
\end{equation}
Define
\begin{equation}
 \mathcal H_0(B):=w(B)\mathsf c(B)\mathsf z(B)-F(B),
 \qquad \forall B\in[0,1).
 \label{eq:zt-Hscript-def}
\end{equation}
Equation \eqref{eq:zt-F-tau-def} gives $F'=-w\Phi$. From \eqref{eq:zt-wB-NB} and
\begin{equation*}
 (\mathsf c\mathsf z)_B
 \stackref{=}{\eqref{eq:zt-slope-variables}}
 -2\mathsf z^2+\mathsf c(m+\mathsf J)
 \stackref{=}{\eqref{eq:zt-Phi-G-def}}
 \mathsf c\mathsf J-\Phi,
\end{equation*}
we obtain
\begin{equation}
 (\mathcal H_0)_B
 \stackref{=}{\eqref{eq:zt-wB-NB}}
 wL
 \quad\text{and}\quad L:=\mathsf c\mathsf J-\mathsf z(N+\mathsf z).
 \label{eq:zt-Hscript-L}
\end{equation}
The slope inequalities imply
\begin{equation*}
 (\mathsf c\mathsf J)_B=-2\mathsf z\mathsf J+\mathsf c\mathsf J_B
 \stackref{\leq}{\eqref{eq:zt-zero-temp-five}}
 \mathsf z\mathsf J.
\end{equation*}
On the other hand, \eqref{eq:zt-wB-NB} and \eqref{eq:zt-z-KJ-identities} give
\begin{equation*}
 [\mathsf z(N+\mathsf z)]_B
 =(m+\mathsf J)(N+\mathsf z)
 +\mathsf z\left(\frac{\mathfrak r_x}{\mathsf c}+\mathsf J+m+\mathsf J\right)>\mathsf z\mathsf J.
\end{equation*}
Thus $L_B<0$ on $(0,1)$.

At $B=0$, \eqref{eq:zt-Phi-centered} and $\mathsf z(0)=0$ give $\mathcal H_0(0)=0$. As $B\uparrow1$, \eqref{eq:zt-F-tau-def} and absolute integrability give $F(B)\to0$, while \eqref{eq:zt-Hscript-def} and \eqref{eq:zt-wCz-boundary} give $\mathcal H_0(B)\to0$. Furthermore, Lemma~\ref{lem:zt-tails} gives $wL\in L^1(0,1)$. The local derivative identity \eqref{eq:zt-Hscript-L} therefore extends $\mathcal H_0$ to an absolutely continuous function on $[0,1]$. Consequently
\begin{equation*}
 \int_0^1wL\dd B
 \stackref{=}{\eqref{eq:zt-Hscript-L}}
 \mathcal H_0(1)-\mathcal H_0(0)=0.
\end{equation*}
Since $L$ is strictly decreasing and $w>0$, $L$ crosses zero exactly once.  Hence $\mathcal H_0$ first increases and then decreases, with zero values at both endpoints.  Thus $\mathcal H_0\geq0$, which is exactly \eqref{eq:zt-tau-upper}.

\smallskip
\noindent\emph{Step 5: Nonnegativity of the residual term in $\cI'(t)$.}

We may now integrate by parts. Equation \eqref{eq:zt-F-tau-def} gives $F'=-w\Phi$, so
\begin{equation}
 \int_0^1w\Phi R_0\dd B
 =-\int_0^1F'R_0\dd B
 \stackref{=}{\text{integration by parts}}
 \int_0^1F(R_0)_B\dd B.
 \label{eq:zt-R0-integration-by-parts}
\end{equation}
There is no boundary contribution at zero because $F(0)=0$. At one, \eqref{eq:zt-tau-upper}, \eqref{eq:zt-wCz-boundary}, and $R_0=O(x^2)$ give $F(B)R_0(B)\to0$. Also,
\begin{equation*}
 (R_0)_B
 \stackref{=}{\eqref{eq:zt-R1-R0-def},\,\eqref{eq:zt-slope-variables}}
 2\mathsf c\mathsf J_B-5\mathsf z\mathsf J.
\end{equation*}
The integral in \eqref{eq:zt-R0-integration-by-parts} is absolutely convergent: by \eqref{eq:zt-tau-lower}--\eqref{eq:zt-tau-upper},
\begin{align*}
 \tau|(R_0)_B|
 &\stackref{\leq}{\eqref{eq:zt-tau-lower},\,\eqref{eq:zt-tau-upper}}
 \mathsf c\mathsf z(2\mathsf c\mathsf J_B+5\mathsf z\mathsf J) \\
 &=2\mathsf z(\mathsf c^2\mathsf J_B)+5\mathsf z^2(\mathsf c\mathsf J)=O(x^4),
\end{align*}
and \eqref{eq:zt-weighted-polynomial-integrability} applies.  Therefore
\begin{equation*}
 \int_0^1w\Phi R_0\dd B
 =\int_0^1w\tau(2\mathsf c\mathsf J_B-5\mathsf z\mathsf J)\dd B.
\end{equation*}
Using the definition of $\cG$ in \eqref{eq:zt-Phi-G-def}, we obtain the exact residual identity
\begin{align*}
 \int_0^1w(R_0\Phi+\cG)\dd B
 =\int_0^1w\{&
 2\mathsf c\mathsf J_B(\tau-\mathsf c\mathsf z) \\
 &+\mathsf z\mathsf J(6\mathsf c\mathsf z-5\tau)\}\dd B.
\end{align*}
Now $\tau-\mathsf c\mathsf z\leq0$ and $\mathsf c\mathsf J_B\leq3\mathsf z\mathsf J$.  Multiplication by the nonpositive quantity $2(\tau-\mathsf c\mathsf z)$ reverses the latter inequality, so
\begin{align}
 \int_0^1w(R_0\Phi+\cG)\dd B
 &\stackref{\geq}{\eqref{eq:zt-zero-temp-five},\,\eqref{eq:zt-tau-upper}}
 \int_0^1w\mathsf z\mathsf J\{6(\tau-\mathsf c\mathsf z)+6\mathsf c\mathsf z-5\tau\}\dd B \notag\\
 &=\int_0^1w\mathsf z\mathsf J\tau\dd B\geq0.                                \label{eq:zt-residual-nonnegative}
\end{align}

\smallskip
\noindent\emph{Step 6: Proof of \eqref{eq:zt-crossing-rule}.}

Finally, if $\cI(t)=0$, then
\begin{equation*}
 \cI'(t)
 \stackref{=}{\eqref{eq:zt-I-prime-decomposition}}
 \frac12\int_0^1w\Phi H\dd B
 +\int_0^1wR_1\Phi\dd B
 +\int_0^1w(R_0\Phi+\cG)\dd B
 \stackref{>}{\eqref{eq:zt-H-term-positive},\,\eqref{eq:zt-R1-term-nonnegative},\,\eqref{eq:zt-residual-nonnegative}}0.
\end{equation*}
Since $\Gamma''=2\cI$ by \eqref{eq:zt-Gamma-second-I}, this is precisely \eqref{eq:zt-crossing-rule}.
\end{proof}

\subsection{Exclusion of internal support gaps}
\label{sec:zt-internal-gaps}

\begin{proposition}[No internal support gap]
\label{prop:zt-no-internal-gap}
There do not exist $0\leq a<b<1$ such that
\begin{equation*}
 a,b\in S\quad\text{and}\quad \nu((a,b))=0.
\end{equation*}
\end{proposition}

\begin{proof}
Recall from \eqref{eq:zt-Gamma-h-G} that
\begin{equation*}
 G(q):=\int_q^1(\Gamma(t)-t)\dd t,
 \qquad \forall q\in[0,1).
\end{equation*}
On $(a,b)$, write $\gamma(t)=m$. This constant $m$ is strictly positive. If $a>0$, then $0\in S$ implies $\nu([0,a])>0$. If $a=0$, the absence of mass in $(0,b)$ and the support condition $0\in S$ imply
\begin{equation*}
 \nu(\{0\})=\nu([0,\varepsilon))>0
\end{equation*}
for every sufficiently small $\varepsilon<b$.  Thus $\gamma(t)=m>0$ throughout the gap, and Proposition~\ref{prop:zt-crossing} applies.

The endpoint consistency and stability conditions in Proposition~\ref{prop:zt-variational-conditions} give
\begin{equation}
 \Gamma(a)=a,\qquad \Gamma(b)=b
 \quad\text{and}\quad \Gamma'(a+)\leq1.
 \label{eq:zt-Gamma-gap-endpoints}
\end{equation}
Moreover, $G(a)=G(b)=0$, and hence
\begin{equation}
 0\stackref{=}{\mathrm{P.\,\ref{prop:zt-variational-conditions}}}
 G(a)-G(b)
 =\int_a^b(\Gamma(t)-t)\dd t.
 \label{eq:zt-Gamma-gap-centered}
\end{equation}

Because $b<1$, Lemma~\ref{lem:zt-parabolic-stability} gives a uniform bound on $u_{xx}$ up to both one-sided gap endpoints. Formula \eqref{eq:zt-Gamma-prime}, weak continuity of the diffusion law, and local continuity of $u_{xx}$ therefore show that $\Gamma'$ has finite one-sided limits and is bounded on $[a,b]$.

By Proposition~\ref{prop:zt-crossing}, every zero of $\Gamma''$ in $(a,b)$ is simple and is crossed from negative to positive. Consequently, $\Gamma''$ has at most one zero: after one upward crossing, a second upward crossing would require an intervening downward zero. Thus $\Gamma''$ has a constant sign on at most two subintervals of $(a,b)$, so $\Gamma'$ is monotone on each of them. The bounded endpoint limits of $\Gamma'$ imply $\Gamma''\in L^1(a,b)$.

We may now integrate twice without an endpoint regularity assumption. Equations \eqref{eq:zt-Gamma-gap-endpoints} and \eqref{eq:zt-Gamma-gap-centered} give
\begin{equation}
 \int_a^b(t-a)(b-t)\Gamma''(t)\dd t
 \stackref{=}{\text{two integrations}}
 (b-a)(\Gamma(a)+\Gamma(b))-2\int_a^b\Gamma(t)\dd t
 \stackref{=}{\eqref{eq:zt-Gamma-gap-endpoints},\,\eqref{eq:zt-Gamma-gap-centered}}0.
 \label{eq:zt-Gamma-weighted-second}
\end{equation}
Because $(t-a)(b-t)>0$ for every $t\in(a,b)$, equation \eqref{eq:zt-Gamma-weighted-second} implies that either $\Gamma''$ changes sign or $\Gamma''\equiv0$ on $(a,b)$. The latter alternative contradicts Proposition~\ref{prop:zt-crossing}, since $\Gamma''(t)=0$ would imply $\Gamma'''(t)>0$ for every $t\in(a,b)$. Hence there is a unique $c\in(a,b)$ such that
\begin{equation*}
 \Gamma''(t)<0\quad\forall t\in(a,c)
 \quad\text{and}\quad
 \Gamma''(t)>0\quad\forall t\in(c,b).
\end{equation*}

A second integration gives
\begin{align*}
 \Gamma(b)-\Gamma(a)-(b-a)\Gamma'(a+)
 &=\int_a^b(b-t)\Gamma''(t)\dd t =\int_a^b\frac{(t-a)(b-t)\Gamma''(t)}{t-a}\dd t.
\end{align*}
For $a<t<c$, the numerator is negative and $(t-a)^{-1}>(c-a)^{-1}$; for $c<t<b$, the numerator is positive and $(t-a)^{-1}<(c-a)^{-1}$. Both inequalities are strict on sets of positive measure. Therefore
\begin{equation}
 \Gamma(b)-\Gamma(a)-(b-a)\Gamma'(a+)
 <\frac1{c-a}\int_a^b(t-a)(b-t)\Gamma''(t)\dd t
 \stackref{=}{\eqref{eq:zt-Gamma-weighted-second}}0.
 \label{eq:zt-Gamma-strict-moment}
\end{equation}
Combining \eqref{eq:zt-Gamma-gap-endpoints} and \eqref{eq:zt-Gamma-strict-moment} yields
\begin{equation*}
 \Gamma(b)-\Gamma(a)
 \stackref{<}{\eqref{eq:zt-Gamma-strict-moment}}
 (b-a)\Gamma'(a+)
 \stackref{\leq}{\eqref{eq:zt-Gamma-gap-endpoints}}b-a,
\end{equation*}
contradicting $\Gamma(b)-\Gamma(a)=b-a$ in \eqref{eq:zt-Gamma-gap-endpoints}.
\end{proof}

\begin{remark}[A gap issuing from zero]
\label{rem:zt-first-gap}
The preceding argument also covers $a=0$. Evenness gives $\Gamma(0)=u_x(0,0)^2=0$, Proposition~\ref{prop:zt-variational-conditions} gives $G(0)=0$ and $\Gamma'(0+)=u_{xx}(0,0)^2\leq1$, and the atom at zero makes the constant gap value $m$ positive. These are precisely the conditions used in \eqref{eq:zt-Gamma-gap-endpoints}. Hence the proof excludes a component $(0,b)$ as well as every component $(a,b)$ with $a>0$.
\end{remark}

\subsection{Exclusion of a terminal gap}
\label{sec:zt-terminal-gap}

\begin{proposition}[No gap ending at one]
\label{prop:zt-no-terminal-gap}
There is no $a<1$ such that $(a,1)\cap S=\varnothing$ and $a\in S$.
\end{proposition}

\begin{proof}
Suppose otherwise.  Then
\begin{equation}
 \gamma(t)=m<\infty,\qquad a<t<1.
 \label{eq:zt-terminal-constant}
\end{equation}
\smallskip
\noindent\emph{Step 1: Proof of \eqref{eq:zt-B-boundary-limit}.}

Write $\varepsilon:=1-t$. For $m>0$, the Cole--Hopf formula on the interval $(a,1)$ where $\gamma(t)=m$ is
\begin{equation}
 e^{mu(1-\varepsilon,x)}
 \stackref{=}{\eqref{eq:zt-general-PDE},\,\eqref{eq:zt-terminal-constant}}
 P_\varepsilon(e^{m|\cdot|})(x)
 =\E e^{m|x+\sqrt\varepsilon Z|},
 \qquad Z\sim N(0,1).
 \label{eq:zt-terminal-Cole-Hopf}
\end{equation}
When $m=0$, the corresponding formula is $u(1-\varepsilon,x)=P_\varepsilon|\cdot|(x)$.

Set
\begin{equation*}
 \mathsf B_\varepsilon(y):=u_x(1-\varepsilon,\sqrt\varepsilon\,y).
\end{equation*}
Differentiating \eqref{eq:zt-terminal-Cole-Hopf}, with the unweighted interpretation when $m=0$, gives
\begin{equation}
 \mathsf B_\varepsilon(y)=
 \frac{\E\!\left[
 \operatorname{sign}(y+Z)e^{m\sqrt\varepsilon|y+Z|}
 \right]}
 {\E e^{m\sqrt\varepsilon|y+Z|}}.
 \label{eq:zt-B-epsilon-ratio}
\end{equation}
For every fixed $y$,
\begin{equation}
 \lim_{\varepsilon\downarrow0}\mathsf B_\varepsilon(y)
 =\E\operatorname{sign}(y+Z)
 =2\phi(y)-1.
 \label{eq:zt-B-boundary-limit}
\end{equation}
\smallskip
\noindent\emph{Step 2: Proof of \eqref{eq:zt-rho-boundary-limit} and \eqref{eq:zt-rho-terminal-Gaussian-bound}.}

For every $(t,x)\in(a,1)\times\R$, equations \eqref{eq:zt-Q-def}, \eqref{eq:zt-terminal-constant}, and \eqref{eq:zt-Q-heat} give
\begin{equation*}
 Q(t,x)\stackref{=}{\eqref{eq:zt-Q-def},\,\eqref{eq:zt-terminal-constant}}\rho_t(x)e^{-mu(t,x)}
 \quad\text{and}\quad
 \partial_tQ(t,x)\stackref{=}{\eqref{eq:zt-Q-heat}}\frac12\partial_{xx}Q(t,x).
\end{equation*}
Fix $t_0\in(a,1)$. Then
\begin{equation*}
 Q(t,\cdot)\stackref{=}{\eqref{eq:zt-Q-heat}}P_{t-t_0}Q(t_0,\cdot),
 \qquad \forall t\in(t_0,1),
\end{equation*}
and hence
\begin{equation}
 Q(1,x):=P_{1-t_0}Q(t_0,\cdot)(x),
 \qquad \forall x\in\R,
 \quad\text{and}\quad
 \lim_{t\uparrow1}\sup_{|x|\leq R}|Q(t,x)-Q(1,x)|=0,
 \qquad \forall R>0.
 \label{eq:zt-Q-terminal-limit}
\end{equation}
The function $Q(1,\cdot)$ is smooth and strictly positive. Therefore
\begin{equation}
 \rho_1(x):=e^{m|x|}Q(1,x)
 \quad\text{is continuous with}\quad \rho_1(0)>0.
 \label{eq:zt-rho-one}
\end{equation}
For every fixed $y$, \eqref{eq:zt-terminal-Cole-Hopf}, its $m=0$ counterpart, and \eqref{eq:zt-Q-terminal-limit} give
\begin{equation*}
 \lim_{\varepsilon\downarrow0}u(1-\varepsilon,\sqrt\varepsilon\,y)=0
 \quad\text{and}\quad
 \lim_{\varepsilon\downarrow0}Q(1-\varepsilon,\sqrt\varepsilon\,y)=Q(1,0).
\end{equation*}
Consequently,
\begin{equation}
 \lim_{\varepsilon\downarrow0}\rho_{1-\varepsilon}(\sqrt\varepsilon\,y)=\rho_1(0).
 \label{eq:zt-rho-boundary-limit}
\end{equation}
The stochastic-control representation gives $u(s,z)\geq0$ for every $(s,z)\in[0,1]\times\R$, so \eqref{eq:zt-Stieltjes-bridge} gives $Q(t,x)\leq c_0p_t(x)$. For every $t_1\in(a,1)$, the continuity of $t\mapsto u(t,0)$ on $[t_1,1]$, the identity $\rho_t=Q(t,\cdot)e^{mu(t,\cdot)}$, and the bound $u(t,x)\leq u(t,0)+|x|$ therefore give, for every $(t,x)\in[t_1,1)\times\R$,
\begin{align}
 \rho_t(x)
 &\stackref{=}{\eqref{eq:zt-Q-def},\,\eqref{eq:zt-terminal-constant}}
 e^{mu(t,x)}Q(t,x)\notag\\
 &\stackref{\leq}{\eqref{eq:zt-Stieltjes-bridge}}
 \frac{c_0}{\sqrt{2\pi t}}
 \exp\left\{-\frac{x^2}{2t}+m u(t,0)+m|x|\right\}\notag\\
 &\leq
 \frac{c_0}{\sqrt{2\pi t_1}}
 \exp\left\{m\sup_{t_1\leq r\leq1}u(r,0)+m^2\right\}e^{-x^2/4}
 =:C_{t_1}e^{-x^2/4}.
 \label{eq:zt-rho-terminal-Gaussian-bound}
\end{align}
Here the second inequality uses $t\leq1$ and $m|x|\leq x^2/4+m^2$.

\smallskip
\noindent\emph{Step 3: Proof of \eqref{eq:zt-terminal-square-root}.}

For $y\geq0$ and $0<\varepsilon<1-a$, define the probability of $y+Z<0$ under the exponential tilt in \eqref{eq:zt-B-epsilon-ratio} by
\begin{align}
 p_\varepsilon(y)
 &:=\frac{\E\!\left[\1_{\{y+Z<0\}}e^{m\sqrt\varepsilon|y+Z|}\right]}
 {\E e^{m\sqrt\varepsilon|y+Z|}}\notag\\
 &\stackref{=}{\text{Gaussian shift}}
 \frac{e^{-m\sqrt\varepsilon y}\overline\phi(y-m\sqrt\varepsilon)}
 {e^{-m\sqrt\varepsilon y}\overline\phi(y-m\sqrt\varepsilon)
  +e^{m\sqrt\varepsilon y}\phi(y+m\sqrt\varepsilon)}.
 \label{eq:zt-p-epsilon}
\end{align}
Equations \eqref{eq:zt-B-epsilon-ratio} and \eqref{eq:zt-p-epsilon} yield
\begin{equation}
 \mathsf B_\varepsilon(y)=1-2p_\varepsilon(y)
 \quad\text{and}\quad
 0\leq p_\varepsilon(y)
 \leq\frac{e^{-m\sqrt\varepsilon y}\overline\phi(y-m\sqrt\varepsilon)}
 {e^{m\sqrt\varepsilon y}\phi(y+m\sqrt\varepsilon)}
 \stackref{\leq}{\phi(y+m\sqrt\varepsilon)\geq\phi(0)=1/2}
 2\overline\phi(y-m\sqrt\varepsilon).
 \label{eq:zt-p-epsilon-bound}
\end{equation}
Since $u(1-\varepsilon,\cdot)$ is even, $\mathsf B_\varepsilon(-y)=-\mathsf B_\varepsilon(y)$. For every sufficiently small $\varepsilon>0$, one has $m\sqrt\varepsilon\leq1$, and hence \eqref{eq:zt-p-epsilon-bound} gives
\begin{equation}
 0\leq1-\mathsf B_\varepsilon(y)^2
 =4p_\varepsilon(|y|)(1-p_\varepsilon(|y|))
 \leq8\overline\phi(|y|-1),
 \qquad \forall y\in\R,
 \label{eq:zt-boundary-dominator}
\end{equation}
Moreover,
\begin{equation*}
 \int_\R\overline\phi(|y|-1)\dd y<\infty.
\end{equation*}

Changing variables $x=\sqrt\varepsilon\,y$ gives
\begin{align*}
 1-\Gamma(1-\varepsilon)
 &=\int_\R\rho_{1-\varepsilon}(x)
   [1-u_x(1-\varepsilon,x)^2]\dd x =\sqrt\varepsilon\int_\R
 \rho_{1-\varepsilon}(\sqrt\varepsilon\,y)
 [1-\mathsf B_\varepsilon(y)^2]\dd y.
\end{align*}
For positive functions $F$ and $H$, we use the notation
\begin{equation*}
 F(\varepsilon)\sim H(\varepsilon)\text{ as }\varepsilon\downarrow0
 \quad\Longleftrightarrow\quad
 \lim_{\varepsilon\downarrow0}\frac{F(\varepsilon)}{H(\varepsilon)}=1.
\end{equation*}
Equations \eqref{eq:zt-B-boundary-limit}, \eqref{eq:zt-rho-boundary-limit}, \eqref{eq:zt-rho-terminal-Gaussian-bound}, and \eqref{eq:zt-boundary-dominator}, together with dominated convergence, imply
\begin{equation*}
 1-\Gamma(1-\varepsilon)
 \stackref{\sim}{\eqref{eq:zt-B-boundary-limit},\,\eqref{eq:zt-rho-boundary-limit},\,\eqref{eq:zt-rho-terminal-Gaussian-bound},\,\eqref{eq:zt-boundary-dominator}}
 \rho_1(0)\sqrt\varepsilon
 \int_\R[1-(2\phi(y)-1)^2]\dd y.
\end{equation*}
To evaluate the integral, let $Z_1,Z_2\sim N(0,1)$ be independent. Fubini's theorem gives
\begin{align*}
 \int_\R\phi(y)\overline\phi(y)\dd y
 &=\int_\R\PP(Z_1\leq y<Z_2)\dd y =\E(Z_2-Z_1)_+=\frac1{\sqrt\pi}.
\end{align*}
Since $1-(2\phi-1)^2=4\phi\overline\phi$, we conclude that
\begin{equation}
 1-\Gamma(t)\sim\frac4{\sqrt\pi}\rho_1(0)\sqrt{1-t},
 \qquad t\uparrow1.
 \label{eq:zt-terminal-square-root}
\end{equation}

\smallskip
\noindent\emph{Step 4: Contradiction with \eqref{eq:zt-G-nonnegative}.}

Set $c:=4\rho_1(0)/\sqrt\pi>0$.  With $\varepsilon:=1-t$, \eqref{eq:zt-terminal-square-root} gives
\begin{equation*}
 \Gamma(t)-t=\varepsilon-c\sqrt\varepsilon
 +o(\sqrt\varepsilon)<0
\end{equation*}
for all $t$ sufficiently close to one.  Therefore, for $q<1$ sufficiently close to one,
\begin{equation*}
 0
 \stackref{\leq}{\eqref{eq:zt-G-nonnegative}}
 G(q)=\int_q^1(\Gamma(t)-t)\dd t<0,
\end{equation*}
which is impossible.
\end{proof}

\begin{proposition}[Smoothness from saturated consistency]
\label{prop:zt-smoothness}
Let $u:=u^{|\cdot|,\gamma}$ solve \eqref{eq:zt-general-PDE}, and let $X$ solve \eqref{eq:zt-optimal-SDE}. Recall from \eqref{eq:zt-minimizer-notation} that $S:=\supp_{[0,1)}\mathrm d\gamma$. If $S=[0,1)$, then
\begin{equation}
 \gamma\in C^\infty([0,1)),
 \qquad
 \gamma(0)=0
 \quad\text{and}\quad
 \mathrm d\gamma(t)=\gamma'(t)\dd t.
 \label{eq:zt-smoothness-conclusion}
\end{equation}
\end{proposition}

\begin{proof}
Since every $t<1$ belongs to $S$, the consistency condition \eqref{eq:zt-consistency-stability} gives
\begin{equation}
 \Gamma(t)=t\qquad(0\leq t<1).
 \label{eq:zt-saturated-consistency}
\end{equation}
Consequently, we have
\begin{align}
 \E\,\mathsf C(t,X_t)^2
 &\stackref{=}{\eqref{eq:zt-saturated-consistency},\,\eqref{eq:zt-Gamma-prime}}1,                                             \label{eq:zt-C-second-moment-one}\\
 \E\,\mathsf D(t,X_t)^2
 &\stackref{=}{\eqref{eq:zt-saturated-consistency},\,\eqref{eq:zt-Gamma-prime-integral}}
 2\gamma(t)\E\,\mathsf C(t,X_t)^3
 \quad\text{for a.e. }t<1.                                   \label{eq:zt-gamma-quotient-ae}
\end{align}
The denominator in \eqref{eq:zt-gamma-quotient-ae} is strictly positive by Lemma~\ref{lem:zt-PDE-facts}.  Hence, with
\begin{equation}
 \overline\gamma(t)
 :=\frac{\E\,\mathsf D(t,X_t)^2}{2\E\,\mathsf C(t,X_t)^3},
 \label{eq:zt-gamma-quotient}
\end{equation}
we have $\gamma=\overline\gamma$ almost everywhere. The first part of Lemma~\ref{lem:zt-polynomial-moment-regularity}, applied on every $[0,T]$ with $T<1$, makes the numerator and denominator in \eqref{eq:zt-gamma-quotient} continuous before any continuity of $\gamma$ is known. Thus $\overline\gamma$ is continuous. Approaching each $t<1$ from the right through the full-measure equality set and using right continuity gives $\gamma(t)=\overline\gamma(t)$, so $\gamma$ is continuous. At time zero, $X_0=0$ and evenness of $u(0,\cdot)$ gives $\mathsf D(0,0)=0$. Formula \eqref{eq:zt-gamma-quotient} then yields $\gamma(0)=0$.

Fix $T<1$. The denominator in \eqref{eq:zt-gamma-quotient} has a positive minimum on $[0,T]$. By the definitions in Lemma~\ref{lem:zt-polynomial-moment-regularity} and \eqref{eq:zt-BCD-shorthand},
\begin{equation*}
 Z_2(t)=\mathsf C(t,X_t)\quad\text{and}\quad Z_3(t)=\mathsf D(t,X_t).
\end{equation*}
If $\gamma\in C^r([0,T])$, apply Lemma~\ref{lem:zt-polynomial-moment-regularity} first to the polynomial $(z_1,z_2,z_3)\mapsto z_3^2$ and then to the polynomial $(z_1,z_2)\mapsto z_2^3$. The numerator and denominator in \eqref{eq:zt-gamma-quotient} are then $C^{r+1}$, so the quotient gives $\gamma\in C^{r+1}([0,T])$. Starting from continuity and iterating yields $\gamma\in C^\infty([0,T])$, including one-sided smoothness at zero. Since $T<1$ was arbitrary, $\gamma\in C^\infty([0,1))$. Since $\gamma$ is nondecreasing, its classical derivative is nonnegative, and its Stieltjes measure is $\mathrm d\gamma(t)=\gamma'(t)\dd t$.
\end{proof}

\subsection{Completion of the zero-temperature proof}
\label{sec:zt-completion}

\begin{proof}[Proof of Theorem~\ref{thm:zero-temperature}]
Suppose that $\overline S^{\,[0,1]}\neq[0,1]$.  Because $0\in S$, a nonempty connected component of the relatively open set $[0,1]\setminus\overline S^{\,[0,1]}$ either lies between two points of $S$ in $[0,1)$ or reaches the endpoint one.  In the first case, $\gamma$ is constant on an internal gap, contradicting Proposition~\ref{prop:zt-no-internal-gap}.  In the second case, it is constant on a terminal gap, contradicting Proposition~\ref{prop:zt-no-terminal-gap}.  Therefore
\begin{equation*}
 \overline S^{\,[0,1]}
 \stackref{=}{\mathrm{P.\,\ref{prop:zt-no-internal-gap}},\,\mathrm{P.\,\ref{prop:zt-no-terminal-gap}}}
 [0,1],
\end{equation*}
and, since $S$ is closed relative to $[0,1)$, it follows that $S=[0,1)$.  Proposition~\ref{prop:zt-smoothness} now gives a nonnegative function
\begin{equation*}
 \rho_\infty=\gamma'=\gamma_\star'\in C^\infty([0,1))
\end{equation*}
such that $\nu_\star(\mathrm d t)=\rho_\infty(t)\dd t$.  This proves \eqref{eq:zero-density-decomposition}.  Because the support of this measure is the already identified set $S=[0,1)$, it also proves \eqref{eq:zero-main} and completes the theorem.
\end{proof}

\appendix

\section{\texorpdfstring{Finite Cole--Hopf inequalities for $\mathsf K$ and $\mathsf J$}{Finite Cole--Hopf inequalities for K and J}}
\label{app:finite-KJ}

This appendix proves Proposition~\ref{prop:zt-finite-KJ}. The inequalities and the organization of the proof are due to \cite[Section~3]{Lopatto2026v2}; the presentation below is restricted to the finite-cascade statement needed in this paper and includes its analytic inputs so that the zero-temperature proof is self-contained.

For $a>0$ and $r\geq0$, recall
\begin{equation}
 \mathcal T_{a,r}f(x)
 :=\frac1a\log P_r(e^{af})(x)
 \quad\text{and}\quad
 P_rg(x):=\E g(x+\sqrt r Z).
 \label{eq:app-Cole-Hopf}
\end{equation}
We first record only the endpoint regularity needed by the maximum principles below.

\begin{lemma}[Finite-cascade endpoint expansion]
\label{lem:app-endpoint-expansion}
Every function obtained from $\log\cosh$ by finitely many operations in Proposition~\ref{prop:zt-finite-KJ} is even, smooth, strictly convex, and satisfies $-1<f'<1$. All its positive-order derivatives are bounded. Moreover, as $x\to+\infty$,
\begin{equation}
 f(x)=x+d_0+d_1e^{-2x}+d_2e^{-4x}+d_3e^{-6x}
 +O(e^{-8x})\quad\text{and}\quad d_1>0,
 \label{eq:app-spatial-expansion}
\end{equation}
with the corresponding differentiated remainder bounds.

Fix such a function $f$, a parameter $a>0$, and $R<\infty$. Define
\begin{equation*}
 u(r,x):=\mathcal T_{a,r}f(x),\qquad \mathsf B(r,x):=u_x(r,x)
 \quad\text{and}\quad \mathsf C(r,x):=u_{xx}(r,x),
 \qquad \forall (r,x)\in[0,R]\times\R.
\end{equation*}
For every $(r,B)\in[0,R]\times(-1,1)$, define $x(r,B)$ and $\mathsf c(r,B)$ by
\begin{equation*}
 \mathsf B(r,x(r,B)):=B\quad\text{and}\quad \mathsf c(r,B):=\mathsf C(r,x(r,B)).
\end{equation*}
With $\delta:=1-B$, there are smooth functions $c_2,c_3$ such that
\begin{equation}
 \mathsf c(r,B)=2\delta+c_2(r)\delta^2+c_3(r)\delta^3
 +O_R(\delta^4).
 \label{eq:app-slope-endpoint-expansion}
\end{equation}
The remainder in \eqref{eq:app-slope-endpoint-expansion} may be differentiated once in $r$ and three times in $B$, with
\begin{equation}
 \partial_r^i\partial_B^jO_R(\delta^4)
 =O_R(\delta^{4-j}),
 \qquad i\in\{0,1\},\quad 0\leq j\leq3.
 \label{eq:app-slope-endpoint-remainder}
\end{equation}
The expansion at $B=-1$ follows by evenness.
\end{lemma}

\begin{proof}
The assertions are immediate for
\begin{equation*}
 \log\cosh x=x-\log2+e^{-2x}-\frac12e^{-4x}
 +\frac13e^{-6x}+O(e^{-8x}).
\end{equation*}
We show that the asserted class is preserved by $f\mapsto\mathcal T_{a,r}f$; replacing $a$ by a smaller value leaves $f$ unchanged. Positivity, evenness, and smoothness follow from Gaussian convolution. Differentiation under the tilted Gaussian law gives
\begin{equation*}
 u_x=\langle f'\rangle
 \quad\text{and}\quad
 u_{xx}=\langle f''\rangle+a\Var(f')>0,
\end{equation*}
so $-1<u_x<1$ and strict convexity are preserved.  Repeated differentiation expresses every spatial derivative as a tilted moment of a bounded polynomial in the derivatives of $f$, proving the stated bounds.

Put $y:=e^{-2x}$.  Exponentiating \eqref{eq:app-spatial-expansion}, applying $P_r$, and using
\begin{equation*}
 P_r(e^{(a-2k)x})
 =e^{(a-2k)x+(a-2k)^2r/2}
\end{equation*}
shows, after factoring out $e^{ax+a^2r/2}$ and taking the logarithm, that $u(r,\cdot)$ again has the form \eqref{eq:app-spatial-expansion}.  The leading coefficient becomes
\begin{equation}
 d_1(r)=d_1(0)e^{2(1-a)r}>0.
 \label{eq:app-leading-coefficient}
\end{equation}
For completeness, the remainder estimate is justified by splitting the Gaussian integral into $\{x+\sqrt rZ\geq x/2\}$ and its complement.  On the first event the differentiated $O(e^{-8x})$ bound may be integrated term by term, uniformly for $r\in[0,R]$; the complementary Gaussian tail is $O_R(e^{-Nx})$ for every $N$.  Differentiation in $r$ is obtained from $\partial_rP_rg=\frac12P_rg''$.  This proves the required uniform differentiated expansion and closes the induction through a finite cascade.

Differentiate the spatial expansion, still writing $y:=e^{-2x}$:
\begin{align*}
 1-B&=2d_1y+4d_2y^2+6d_3y^3+O_R(y^4),\\
 \mathsf c&=4d_1y+16d_2y^2+36d_3y^3+O_R(y^4).
\end{align*}
Because $d_1$ is bounded away from zero on $[0,R]$, the first relation has a uniform inverse
\begin{equation*}
 y=b_1(r)\delta+b_2(r)\delta^2+b_3(r)\delta^3
 +O_R(\delta^4).
\end{equation*}
The differentiated inverse-function theorem gives the remainder bounds through the orders stated in \eqref{eq:app-slope-endpoint-remainder}.  Substitution into the expansion for $\mathsf c$ proves \eqref{eq:app-slope-endpoint-expansion}; its linear coefficient is $4d_1/(2d_1)=2$.
\end{proof}

We use a standard maximum principle in which the diffusion coefficient may vanish at the slope endpoints.

\begin{lemma}[Degenerate-strip comparison]
\label{lem:app-comparison}
Let $v$ be continuous on $[0,R]\times[\ell_-,\ell_+]$ and $C^{1,2}$ in the open strip with positive time.  Suppose
\begin{equation*}
 v_r=dv_{BB}+bv_B+cv+F,
\end{equation*}
where $d>0$ in the open strip, $c$ is bounded above, and all coefficients are continuous on compact subsets.  If $F\geq0$ and $v$ is nonnegative at time zero and on both spatial boundaries, then $v\geq0$ everywhere.  The assertion with all signs reversed also holds.
\end{lemma}

\begin{proof}
Choose $\lambda>\sup c$ and apply the interior minimum argument to $e^{-\lambda r}v$.  A negative minimum cannot occur at the prescribed boundary.  At an interior negative minimum the equation gives a strictly positive time derivative, contradicting the nonpositive left time derivative there.  Applying the result to $-v$ gives the reversed statement.  No positivity of $d$ at the spatial endpoints is used.
\end{proof}

We now prove the stated inequalities. Fix $a>0$ and let $u(r,x):=\mathcal T_{a,r}f(x)$. Define $\mathsf B(r,x)$, $x(r,B)$, and $\mathsf c(r,B)$ as in Lemma~\ref{lem:app-endpoint-expansion}. For every $(r,B)\in[0,\infty)\times(-1,1)$, define
\begin{equation}
 \mathsf A(r,B):=-\frac{\mathsf c_B(r,B)}{2B},\qquad \mathsf K(r,B):=\mathsf A(r,B)-a
 \quad\text{and}\quad \mathsf J(r,B):=-a-\frac12\mathsf c_{BB}(r,B).
 \label{eq:app-AKJ}
\end{equation}
Evenness supplies the smooth values at $B=0$, and differentiation gives
\begin{equation}
 \mathsf c_B=-2(a+\mathsf K)B\quad\text{and}\quad
 \mathsf J=\mathsf K+B\mathsf K_B=(B\mathsf K)_B.
 \label{eq:app-KJ-identities}
\end{equation}
Unless arguments are displayed explicitly in the remainder of this appendix, $\mathsf c,\mathsf A,\mathsf K$, and $\mathsf J$ mean the functions in \eqref{eq:app-AKJ} evaluated at $(r,B)$, and every subscript $B$ denotes $\partial_B$.

\begin{proposition}[Propagation for fixed $a$]
\label{prop:app-fixed-parameter}
Suppose that at $r=0$,
\begin{equation}
 \mathsf K,\mathsf K_B,\mathsf J,\mathsf J_B\geq0\quad\text{and}\quad (\mathsf c^{3/2}\mathsf J)_B\leq0
 \quad(0\leq B<1).
 \label{eq:app-initial-inequalities}
\end{equation}
Then the same inequalities hold for every $r\geq0$.
\end{proposition}

\begin{proof}
All derivatives in this proof are taken at fixed $B$.  From
\begin{equation*}
 u_r=\frac12(u_{xx}+au_x^2)
\end{equation*}
and the identity $B=u_x(r,x(r,B))$, direct differentiation gives
\begin{equation}
 x_r=\mathsf K B\quad\text{and}\quad \mathsf c_r=-\mathsf c^2\mathsf J.
 \label{eq:app-xC-flow}
\end{equation}
Two $B$-derivatives of the second equation yield
\begin{equation}
 \mathsf J_r=\frac{\mathsf c^2}{2}\mathsf J_{BB}+2\mathsf c\mathsf c_B\mathsf J_B
 +\bigl(\mathsf c_B^2-2\mathsf c(\mathsf J+a)\bigr)\mathsf J.
 \label{eq:app-J-PDE}
\end{equation}
For $\kappa:=\mathsf J_B$, one further differentiation gives
\begin{align}
 \kappa_r=&\frac{\mathsf c^2}{2}\kappa_{BB}+3\mathsf c\mathsf c_B\kappa_B
 +\bigl(3\mathsf c_B^2-6\mathsf c(\mathsf J+a)-2\mathsf c\mathsf J\bigr)\kappa
 -6\mathsf c_B(\mathsf J+a)\mathsf J.
 \label{eq:app-kappa-PDE}
\end{align}

We first control the slope endpoints.  Substitute \eqref{eq:app-slope-endpoint-expansion} into $\mathsf c_r=-\mathsf c^2\mathsf J$.  Comparing the coefficients of $\delta^2$ and $\delta^3$ gives
\begin{equation}
 c_2'=4(a+c_2)\quad\text{and}\quad
 c_3'=12c_3+4c_2(a+c_2).
 \label{eq:app-c2c3-flow}
\end{equation}
Set
\begin{equation*}
 j_\infty:=\mathsf J(r,1)=-a-c_2(r)\quad\text{and}\quad
 \kappa_\infty:=\mathsf J_B(r,1)=3c_3(r).
\end{equation*}
Equation \eqref{eq:app-c2c3-flow} becomes
\begin{equation}
 j_\infty'=4j_\infty\quad\text{and}\quad
 \kappa_\infty'=12\kappa_\infty
 +12(j_\infty+a)j_\infty.
 \label{eq:app-endpoint-flow}
\end{equation}
The initial inequalities therefore imply $j_\infty,\kappa_\infty\geq0$ at every later time.

Apply Lemma~\ref{lem:app-comparison} to \eqref{eq:app-J-PDE} on $-1\leq B\leq1$.  Its initial value is nonnegative, and both endpoint values equal $j_\infty\geq0$. Lemma~\ref{lem:app-endpoint-expansion} verifies the required boundary regularity and boundedness of the zero-order coefficient.  Hence
\begin{equation}
 \mathsf J\geq0.
 \label{eq:app-J-positive}
\end{equation}
Since $B\mathsf K(B)\to0$ as $B\downarrow0$, integration of \eqref{eq:app-KJ-identities} gives
\begin{equation}
 B\mathsf K(B)=\int_0^B\mathsf J(v)\dd v.
 \label{eq:app-K-average}
\end{equation}
Thus $\mathsf K\geq0$ and $\mathsf c_B=-2(a+\mathsf K)B\leq0$ for $B\geq0$.

The boundary values for $\kappa$ on $[0,1]$ are $\kappa(r,0)=0$ by parity and $\kappa(r,1)=\kappa_\infty(r)\geq0$.  The source in \eqref{eq:app-kappa-PDE} is nonnegative because $\mathsf c_B\leq0$, $\mathsf J\geq0$, and $\mathsf J+a>0$.  A second application of the comparison lemma gives
\begin{equation}
 \mathsf J_B=\kappa\geq0.
 \label{eq:app-JB-positive}
\end{equation}
The function $\mathsf J$ is therefore nondecreasing, so
\begin{equation*}
 \mathsf K(B)=\frac1B\int_0^B\mathsf J(v)\dd v\leq \mathsf J(B).
\end{equation*}
Together with $\mathsf J=\mathsf K+B\mathsf K_B$, this proves $\mathsf K_B\geq0$.

It remains to propagate the final inequality.  Define
\begin{equation*}
 \mathcal E:=\mathsf c^{3/2}\mathsf J\quad\text{and}\quad \mathcal F:=\mathcal E_B.
\end{equation*}
Using \eqref{eq:app-xC-flow} and \eqref{eq:app-J-PDE}, one obtains
\begin{equation}
 \mathcal E_r
 =\frac{\mathsf c^2}{2}\mathcal E_{BB}
 +\frac{\mathsf c\mathsf c_B}{2}\mathcal E_B
 -V\mathcal E-\Lambda\mathcal E^2,
 \label{eq:app-E-PDE}
\end{equation}
where
\begin{equation*}
 V:=\frac{\mathsf c_B^2}{8}+\frac{\mathsf c(\mathsf J+a)}2\quad\text{and}\quad
 \Lambda:=\frac32\mathsf c^{-1/2}.
\end{equation*}
Since $\mathsf c_{BB}=-2(\mathsf J+a)$,
\begin{equation}
 V_B=\frac{\mathsf c}{2}\mathsf J_B\quad\text{and}\quad
 \Lambda_B=-\frac34\mathsf c^{-3/2}\mathsf c_B.
 \label{eq:app-VLambda-derivatives}
\end{equation}
Differentiating \eqref{eq:app-E-PDE} gives
\begin{align}
 \mathcal F_r={}&\frac{\mathsf c^2}{2}\mathcal F_{BB}
 +\frac{3\mathsf c\mathsf c_B}{2}\mathcal F_B
 +\left[\left(\frac{\mathsf c\mathsf c_B}{2}\right)_B
      -V-2\Lambda\mathcal E\right]\mathcal F
 -V_B\mathcal E-\Lambda_B\mathcal E^2.
 \label{eq:app-F-PDE}
\end{align}
Equations \eqref{eq:app-J-positive}, \eqref{eq:app-JB-positive}, and $\mathsf c_B\leq0$ imply
\begin{equation*}
 \mathcal E\geq0,\qquad V_B\geq0\quad\text{and}\quad \Lambda_B\geq0,
\end{equation*}
so the last two terms in \eqref{eq:app-F-PDE} form a nonpositive source.  The apparent singularities at $B=1$ are harmless because
\begin{equation*}
 \Lambda\mathcal E=\frac32\mathsf c\mathsf J\quad\text{and}\quad
 \Lambda_B\mathcal E^2=-\frac34\mathsf c^{3/2}\mathsf c_B\mathsf J^2.
\end{equation*}
The endpoint expansion also gives
\begin{equation*}
 \mathcal F(r,B)
 =-3\sqrt2\,j_\infty(r)(1-B)^{1/2}
 +O_R((1-B)^{3/2}),
\end{equation*}
and hence $\mathcal F(r,1)=0$; parity gives $\mathcal F(r,0)=0$.  The reversed comparison principle applied to \eqref{eq:app-F-PDE} now yields
\begin{equation}
 (\mathsf c^{3/2}\mathsf J)_B=\mathcal F\leq0.
 \label{eq:app-F-negative}
\end{equation}
This completes the fixed-parameter propagation.
\end{proof}

\begin{proof}[Completion of the proof of
Proposition~\ref{prop:zt-finite-KJ}] It remains to check the replacement of $a$ by a smaller value and the initial function. Keep $f$, hence $B$ and $\mathsf c$, fixed, and let $b\in(0,a)$. Superscripts indicate which parameter is used in \eqref{eq:app-AKJ}. Then
\begin{equation*}
 \mathsf K^{(b)}=\mathsf K^{(a)}+a-b,\qquad
 \mathsf J^{(b)}=\mathsf J^{(a)}+a-b\quad\text{and}\quad
 (\mathsf J^{(b)})_B=(\mathsf J^{(a)})_B.
\end{equation*}
Moreover,
\begin{equation*}
 (\mathsf c^{3/2}\mathsf J^{(b)})_B
 =(\mathsf c^{3/2}\mathsf J^{(a)})_B
 +\frac32(a-b)\mathsf c^{1/2}\mathsf c_B
 \leq(\mathsf c^{3/2}\mathsf J^{(a)})_B.
\end{equation*}
Thus every inequality is preserved when $a$ is replaced by $b\in(0,a)$.

For the initial pair $(f,a)=(\log\cosh,1)$,
\begin{equation*}
 \mathsf B(0,x)=\tanh x,\qquad \mathsf c(0,B)=1-B^2
 \quad\text{and}\quad \mathsf K(0,B)=\mathsf J(0,B)=0.
\end{equation*}
The stated inequalities therefore hold initially. Proposition~\ref{prop:app-fixed-parameter} and the parameter-decrease calculation propagate them through every finite cascade. Finally, with $\mathsf z=-\mathsf c_B/2$,
\begin{equation*}
 (\mathsf c^{3/2}\mathsf J)_B=\mathsf c^{1/2}(\mathsf c\mathsf J_B-3\mathsf z\mathsf J),
\end{equation*}
so \eqref{eq:app-F-negative} implies $\mathsf c\mathsf J_B\leq3\mathsf z\mathsf J$.
\end{proof}

\section{Proofs of technical lemmas}
\label{app:technical-lemmas}

\subsection{Proof of Lemma~\ref{lem:ft-endpoint-density}}
\label{app:ft-endpoint-density}

The argument uses only preservation of even unimodality by the heat flow; the log-concavity and higher derivative inequalities of a general transformed-density theorem are not needed.  The Brownian-bridge identity below is the initial Girsanov calculation also used in \cite[proof of Proposition~7.1]{Lopatto2026}; it is reproduced here in full.  Fix $\beta>1$ throughout.

For $\tau>0$, let
\begin{equation}
 g_\tau(x):=\frac1{\sqrt{2\pi\tau}}
 \exp\left(-\frac{x^2}{2\tau}\right).
 \label{eq:app-ft-Gaussian-kernel}
\end{equation}
For $s>0$, let $\mathfrak b^{\,s}$ be a standard Brownian bridge from zero to zero on $[0,s]$, with covariance
\begin{equation*}
 \E\mathfrak b_t^{\,s}\mathfrak b_r^{\,s}
 =t\wedge r-\frac{tr}{s}.
\end{equation*}

We first derive the representation used in the approximation.  Let $\sigma$ be a probability measure on $[0,1]$, write $\alpha_\sigma(t):=\sigma([0,t])$, and let $u_\sigma$ be its positive-temperature Parisi PDE solution.  Let $p_\sigma(s,\cdot)$ denote the density at time $s$ of
\begin{equation*}
 \dd X_t^\sigma
 =\beta^2\alpha_\sigma(t)
   u_{\sigma,x}(t,X_t^\sigma)\dd t+\beta\dd W_t
 \quad\text{and}\quad X_0^\sigma=0.
\end{equation*}
Then, for $s>0$,
\begin{align}
 &e^{-\alpha_\sigma(s)u_\sigma(s,x)}p_\sigma(s,x)
 =g_{\beta^2s}(x)\,
 \E_{\rm br}\exp\left\{
 -\int_{[0,s]}u_\sigma\left(
 t,\frac tsx+\beta\mathfrak b_t^{\,s}
 \right)\sigma(\mathrm d t)
 \right\}.
 \label{eq:app-ft-bridge-representation}
\end{align}
In particular, the density exists and the two sides are continuous and strictly positive.

Here is a direct verification.  The terminal datum $\log\cosh$ is even, convex, and one-Lipschitz.  The constant-coefficient Cole--Hopf formula, followed by approximation of $\alpha_\sigma$ by step functions, shows that
\begin{equation}
 u_\sigma(t,\cdot)\text{ is even and convex}
 \quad\text{and}\quad |u_{\sigma,x}|\leq1.
 \label{eq:app-ft-u-basic-properties}
\end{equation}
Indeed, evenness and the Lipschitz bound are immediate from the Gaussian formula, while convexity follows from H\"older's inequality when $m>0$ and from Gaussian averaging when $m=0$. These properties pass to the limit under the comparison estimate proved below.

Put $Y_t:=\beta W_t$.  It\^o's formula and the Parisi PDE give
\begin{equation}
 \dd u_\sigma(t,Y_t)
 =-\frac{\beta^2}{2}\alpha_\sigma(t)
   u_{\sigma,x}(t,Y_t)^2\dd t
 +\beta u_{\sigma,x}(t,Y_t)\dd W_t.
 \label{eq:app-ft-Ito-u}
\end{equation}
Regard $\alpha_\sigma$ as a right-continuous function of bounded variation and set $\alpha_\sigma(0-):=0$.  Integration by parts in $\alpha_\sigma(t)u_\sigma(t,Y_t)$ transforms \eqref{eq:app-ft-Ito-u} into
\begin{align*}
 \alpha_\sigma(s)u_\sigma(s,Y_s)
 ={}&\int_{[0,s]}u_\sigma(t,Y_t)\sigma(\mathrm d t)
 -\frac{\beta^2}{2}\int_0^s\alpha_\sigma(t)^2
 u_{\sigma,x}(t,Y_t)^2\dd t\\
 &+\beta\int_0^s\alpha_\sigma(t)
 u_{\sigma,x}(t,Y_t)\dd W_t.
\end{align*}
Consequently,
\begin{equation*}
 Z_s:=\exp\left\{
 \alpha_\sigma(s)u_\sigma(s,Y_s)
 -\int_{[0,s]}u_\sigma(t,Y_t)\sigma(\mathrm d t)
 \right\}
\end{equation*}
is the exponential martingale with stochastic integrand $\beta\alpha_\sigma u_{\sigma,x}$.  It is a true martingale by \eqref{eq:app-ft-u-basic-properties}.  Under the measure with density $Z_s$, Girsanov's theorem turns $Y$ into the diffusion $X^\sigma$. The bounded spatial derivatives of the positive-temperature PDE solution give uniqueness in law for this diffusion. Conditionally on $Y_s=x$, its path is
\begin{equation*}
 Y_t=\frac tsx+\beta\mathfrak b_t^{\,s},
 \qquad 0\leq t\leq s.
\end{equation*}
This proves \eqref{eq:app-ft-bridge-representation}.  The expectation there is finite because $|u_\sigma(t,y)|\leq M+|y|$ for a finite $M$, and the supremum of a Brownian bridge has exponential moments of every order.  The same bound and dominated convergence give continuity in $x$.  The argument may first be carried out for step functions and then passed to the limit; equivalently, \eqref{eq:app-ft-Ito-u} follows from the standard It\^o--Krylov formula.

We next isolate the only order property needed below.  Call a nonnegative function symmetric decreasing if it is even and nonincreasing on $[0,\infty)$.  For $h>0$, define the heat operator with variance $\beta^2h$ by
\begin{equation}
 (\mathsf H_h\varphi)(x)
 :=\int_\R g_{\beta^2h}(x-y)\varphi(y)\dd y
 =\E\varphi(x+\beta\sqrt h Z),
 \qquad Z\sim N(0,1).
 \label{eq:app-ft-heat-convention}
\end{equation}
This operator preserves symmetric decreasing functions.  To see this, write such a function by the layer-cake formula as a positive mixture of indicators of centered intervals.  For $a\geq0$ and $x>0$,
\begin{equation*}
 \frac{\dd}{\dd x}
 \int_{-a}^ag_{\beta^2h}(x-y)\dd y
 =g_{\beta^2h}(x+a)-g_{\beta^2h}(x-a)\leq0.
\end{equation*}
Linearity and monotone convergence prove the assertion.

Suppose now that $\sigma$ is finitely supported.  On an interval on which $\alpha_\sigma=m$ is constant, set
\begin{equation*}
 r_\sigma(t,x)
 :=e^{-m u_\sigma(t,x)}p_\sigma(t,x).
\end{equation*}
The Fokker--Planck equation for $p_\sigma$ and the Parisi PDE are
\begin{align*}
 (p_\sigma)_t
 &=\frac{\beta^2}{2}(p_\sigma)_{xx}
   -\beta^2\partial_x(mu_{\sigma,x}p_\sigma)\qquad\text{and}\qquad
 (u_\sigma)_t
 =-\frac{\beta^2}{2}
   (u_{\sigma,xx}+mu_{\sigma,x}^2).
\end{align*}
Substitution of $p_\sigma=e^{mu_\sigma}r_\sigma$ cancels the drift and all zero-order terms, leaving
\begin{equation}
 (r_\sigma)_t=\frac{\beta^2}{2}(r_\sigma)_{xx}.
 \label{eq:app-ft-r-heat}
\end{equation}
At an atom of $\sigma$ of size $d$ at time $t$, continuity of the diffusion law gives the update
\begin{equation}
 r_\sigma(t,x)
 =e^{-d u_\sigma(t,x)}r_\sigma(t-,x).
 \label{eq:app-ft-r-atom}
\end{equation}
At time zero the transformed measure is a positive multiple of $\delta_0$.  Its first positive heat evolution is a positive multiple of the Gaussian kernel in \eqref{eq:app-ft-Gaussian-kernel}.  Equation \eqref{eq:app-ft-r-heat} and the preceding heat-flow observation preserve the symmetric decreasing property.  At an atom, \eqref{eq:app-ft-u-basic-properties} makes $e^{-d u_\sigma(t,\cdot)}$ positive and symmetric decreasing, so \eqref{eq:app-ft-r-atom} preserves it as well.  Induction through the finitely many atoms proves that $r_\sigma(t,\cdot)$ is symmetric decreasing at every positive time.

We finally pass to the Parisi measure.  Abbreviate
\begin{equation*}
 \mu:=\mu_\beta,\qquad q:=q_\beta,\qquad c:=c_\beta
 \quad\text{and}\quad u:=u_\mu.
\end{equation*}
By Theorem~\ref{thm:finite-temperature-structure}, $q>0$, $\mu([0,q])=1$, and $\mu(\{q\})=c$.  Choose finitely supported measures $\mu_n^-$ on $[0,q)$ with total mass $1-c$ that converge weakly to $\mu|_{[0,q)}$, and set $\mu_n:=\mu_n^-+c\delta_q$. Then $\mu_n$ converges weakly to $\mu$, and their distribution functions satisfy
\begin{equation}
 \lim_{n\to\infty}\int_0^1|\alpha_{\mu_n}(t)
 -\alpha_\mu(t)|\dd t=0.
 \label{eq:app-ft-alpha-approximation}
\end{equation}
For solutions with the same terminal datum and coefficients $\alpha_1,\alpha_2$, comparison with time-dependent constant barriers, using $|u_x|\leq1$, gives
\begin{equation}
 \sup_{t,x}|u_1(t,x)-u_2(t,x)|
 \leq\frac{\beta^2}{2}
 \int_0^1|\alpha_1(t)-\alpha_2(t)|\dd t.
 \label{eq:app-ft-PDE-stability}
\end{equation}
Thus $\lim_{n\to\infty}\|u_{\mu_n}-u\|_\infty=0$.

Since $\alpha_{\mu_n}(q)=\alpha_\mu(q)=1$, the left side of \eqref{eq:app-ft-bridge-representation} at time $q$ is respectively
\begin{equation*}
 f_n(x):=e^{-u_{\mu_n}(q,x)}p_{\mu_n}(q,x)
 \quad\text{and}\quad f_\beta(x):=e^{-u(q,x)}p_\beta(x).
\end{equation*}
For a fixed bridge path and $x$ in a compact set, uniform convergence of $u_{\mu_n}$ and weak convergence of $\mu_n$ give
\begin{align*}
 &\lim_{n\to\infty}\int_{[0,q]}u_{\mu_n}\left(
 t,\frac tqx+\beta\mathfrak b_t^{\,q}
 \right)\mu_n(\mathrm d t)=\int_{[0,q]}u\left(
 t,\frac tqx+\beta\mathfrak b_t^{\,q}
 \right)\mu(\mathrm d t),
\end{align*}
locally uniformly in $x$.  Indeed, as $x$ ranges over a compact set, the corresponding continuous functions of $t\in[0,q]$ form a compact subset of $C([0,q])$, and weak convergence of finite measures is uniform on such a set.  The bound $|u_{\mu_n}(t,y)|\leq M+|y|$, with $M$ independent of $n$, and the exponential moments of the bridge supremum permit dominated convergence in \eqref{eq:app-ft-bridge-representation}.  Hence
\begin{equation}
 \lim_{n\to\infty}f_n=f_\beta
 \quad\text{locally uniformly on }\R.
 \label{eq:app-ft-f-approximation}
\end{equation}
Every $f_n$ is symmetric decreasing by the finite-support argument, so \eqref{eq:app-ft-f-approximation} proves that $f_\beta$ is even and nonincreasing on $[0,\infty)$.  Formula \eqref{eq:app-ft-bridge-representation} gives its continuity and strict positivity.  Finally, \eqref{eq:ft-terminal-value-at-q} gives
\begin{equation*}
 0<f_\beta(x)
 =e^{-\frac{\beta^2}{2}(1-q_\beta)}
   \frac{p_\beta(x)}{\cosh x}
 \leq p_\beta(x).
\end{equation*}
Thus $f_\beta$ is integrable.  Since it is strictly positive, it could be constant on $[0,\infty)$ only if its integral were infinite.  It is therefore nonconstant, completing the proof of Lemma~\ref{lem:ft-endpoint-density}.

\subsection{Proof of Lemma~\ref{lem:zt-parabolic-stability}}\label{a.pf.lem:zt-parabolic-stability}

The comparison principle and the one-Lipschitz property in $x$ give
\begin{equation*}
 \|v_n-v\|_{L^\infty([0,1]\times\R)}
 \stackref{\leq}{\text{comparison}}
 \|g_n-g\|_\infty
 +\frac12\|\alpha_n-\alpha\|_{L^1(0,1)},
\end{equation*}
which proves \eqref{eq:zt-parabolic-uniform-stability}; compare the standard coefficient-stability estimate in \cite[Proposition~2(i), equation~(7)]{ChenHandschyLerman}. Define $\widetilde g_n:=v_n(T',\cdot)$ and $\widetilde g:=v(T',\cdot)$. These functions are convex and one-Lipschitz, and the preceding estimate gives $\lim_{n\to\infty}\|\widetilde g_n-\widetilde g\|_\infty=0$. After reversing time on $[0,T']$, the remaining assertions are the interior derivative estimates and compactness argument in \cite[Proposition~2(ii) and the proof of Proposition~2(iii)]{ChenHandschyLerman}, with $\xi''\equiv1$. Although that proposition is stated with terminal datum $|x|$, its proof uses only convexity, the uniform one-Lipschitz bound, the uniform bound on the coefficients, and the positive distance $T'-T$ from the terminal time. It therefore applies to $\widetilde g_n$ and $\widetilde g$. The cited estimates give the joint continuity and the global bound \eqref{eq:zt-global-parabolic-derivative-bound}, and make every $\partial_x^kv_n$ precompact on $[0,T]\times[-R,R]$. The uniform convergence already proved identifies every subsequential limit with $\partial_x^kv$, which gives \eqref{eq:zt-parabolic-Ck-stability} for the whole sequence.

It remains only to obtain the global convergence of the first derivatives. Let $\varepsilon_n:=\|v_n-v\|_{L^\infty([0,T']\times\R)}$ and let $M$ bound the second derivatives in \eqref{eq:zt-global-parabolic-derivative-bound}. Convexity and forward and backward difference quotients give, for every $h>0$,
\begin{equation*}
 |v_{n,x}(t,x)-v_x(t,x)|
 \leq Mh+\frac{2\varepsilon_n}{h}.
\end{equation*}
If $\varepsilon_n>0$, take $h=\sqrt{\varepsilon_n}$; the case $\varepsilon_n=0$ is immediate. This proves \eqref{eq:zt-global-gradient-convergence}.

\subsection{Proof of Lemma~\ref{lem:zt-regularized-terminal-data}}\label{a.pf.lem:zt-regularized-terminal-data}

Fix $\lambda>0$ and first replace $\gamma\wedge\lambda$ by a nondecreasing step function $\alpha:[0,1)\to[0,\lambda]$. Put
\begin{equation*}
 v(s,x):=u^{h_\lambda,\alpha}(1-s,x)
 \quad\text{and}\quad
 \widehat{\mathsf D}(s,x):=v_{xxx}(s,x),
 \qquad \forall (s,x)\in[0,1]\times\R.
\end{equation*}
Extend $\alpha$ to one by $\alpha(1):=\alpha(1-)$. The rescaled function $w(r,y):=\lambda v(2r/\lambda^2,y/\lambda)$ has initial datum $w(0,y)=\log\cosh y$ and, for every $(r,y)\in[0,\lambda^2/2]\times\R$, solves
\begin{equation*}
 w_r=w_{yy}+m(r)w_y^2
 \quad\text{with}\quad
 m(r):=\lambda^{-1}\alpha(1-2r/\lambda^2)\in[0,1].
\end{equation*}
Extend $m$ by zero on $(\lambda^2/2,\infty)$. Since $m$ is nonincreasing and bounded by one, the spatial regularity and boundedness required below follow from \cite[Theorem~4, Lemma~10, and Corollary~11]{JagannathTobascoDynamic}. On every interval $[s_0,s_1]$ on which $\alpha(1-s)=a$, differentiating the equation for $v$ three times gives
\begin{equation}
 \partial_s\widehat{\mathsf D}
 =\frac12\partial_{xx}\widehat{\mathsf D}
 +av_x\partial_x\widehat{\mathsf D}
 +3av_{xx}\widehat{\mathsf D}.
 \label{eq:zt-regularized-D-equation}
\end{equation}
Evenness gives $\widehat{\mathsf D}(s,0)=0$, while \eqref{eq:zt-terminal-regularization} gives
\begin{equation*}
 \widehat{\mathsf D}(0,x)
 \stackref{=}{\eqref{eq:zt-terminal-regularization}}
 -2\lambda^2\tanh(\lambda x)\sech^2(\lambda x)\leq0,
 \qquad \forall x>0.
\end{equation*}
Choose $M\geq\|3av_{xx}\|_{L^\infty([s_0,s_1]\times\R)}$. Equation \eqref{eq:zt-regularized-D-equation} shows that $e^{-M(s-s_0)}\widehat{\mathsf D}$ satisfies a linear uniformly parabolic equation whose zero-order coefficient is nonpositive. The maximum principle on the unbounded half-line, in the form used in \cite[proof of Lemma~8]{JagannathTobascoDynamic}, preserves $\widehat{\mathsf D}\leq0$ from $s_0$ to $s_1$. Spatial derivatives are continuous at the finitely many jumps of $\alpha$ by the cited regularity result, so induction gives
\begin{equation*}
 \partial_x^3u^{h_\lambda,\alpha}(t,x)\leq0,
 \qquad \forall (t,x)\in[0,1)\times(0,\infty).
\end{equation*}

Choose nondecreasing step functions $\alpha_n:[0,1)\to[0,\lambda]$ such that $\lim_{n\to\infty}\|\alpha_n-\gamma\wedge\lambda\|_{L^1(0,1)}=0$. The comparison principle gives
\begin{equation*}
 \|u^{h_\lambda,\alpha_n}-u_\lambda\|_\infty
 \leq\frac12\|\alpha_n-\gamma\wedge\lambda\|_{L^1(0,1)}.
\end{equation*}
Consequently, Lemma~\ref{lem:zt-parabolic-stability} gives, for every $T<1$,
\begin{equation*}
 \lim_{n\to\infty}\partial_x^3u^{h_\lambda,\alpha_n}
 =\partial_x^3u_\lambda
 \quad\text{locally uniformly on }[0,T]\times\R.
\end{equation*}
Hence
\begin{equation}
 (u_\lambda)_{xxx}(t,x)\leq0,
 \qquad \forall (t,x)\in[0,1)\times(0,\infty).
 \label{eq:zt-ulambda-D-negative}
\end{equation}

The comparison principle, $\|h_\lambda-|\cdot|\|_\infty\leq(\log2)/\lambda$, and $|u_x|,|(u_\lambda)_x|\leq1$ give
\begin{equation*}
 \|u_\lambda-u\|_\infty
 \stackref{\leq}{\text{comparison}}
 \frac{\log2}{\lambda}
 +\frac12\|\gamma-\gamma\wedge\lambda\|_{L^1(0,1)}.
\end{equation*}
The right side tends to zero. Fix $T<T'<1$. Since $\gamma(T')<\infty$, the coefficients $\gamma\wedge\lambda$ are uniformly bounded on $[0,T']$. Lemma~\ref{lem:zt-parabolic-stability}, applied with $(g_n,\alpha_n)=(h_\lambda,\gamma\wedge\lambda)$ and $(g,\alpha)=(|\cdot|,\gamma)$ as $\lambda\to\infty$, gives \eqref{eq:zt-lambda-smooth-convergence}. Finally, for every $(t,x)\in[0,T]\times(0,\infty)$,
\begin{equation*}
 \mathsf D(t,x)
 \stackref{=}{\eqref{eq:zt-lambda-smooth-convergence}}
 \lim_{\lambda\to\infty}\partial_x^3u_\lambda(t,x)
 \stackref{\leq}{\eqref{eq:zt-ulambda-D-negative}}0,
\end{equation*}
which proves \eqref{eq:zt-D-negative} because $T<1$ is arbitrary.

\subsection{Proof of Lemma~\ref{lem:zt-bridge-formula}}\label{a.pf.lem:zt-bridge-formula}

\noindent\emph{Step 1: The bridge formula for step functions and convergence of the PDE data and initial factor.}

Suppose first that $\gamma$ is a step function. Lemma~\ref{lem:zt-Q-dynamics} alternates Gaussian heat evolution and multiplication by $e^{-\Delta\gamma(s)u(s,\cdot)}$. The usual conditioning of Brownian motion on its endpoint gives \eqref{eq:zt-Stieltjes-bridge} with the Stieltjes integral replaced by the finite sum over jump times.

For arbitrary $\gamma\in\cU$, choose nondecreasing step functions $\gamma_n\in\cU$ such that
\begin{align}
 \lim_{n\to\infty}\|\gamma_n-\gamma\|_{L^1(0,1)}&=0
 \quad\text{and}\quad \lim_{n\to\infty}\gamma_n(0)=\gamma(0),\notag\\
 \lim_{n\to\infty}\int_{[0,T]}f(s)\,\mathrm d\gamma_n(s)
 &=\int_{[0,T]}f(s)\,\mathrm d\gamma(s),
 \qquad \forall f\in C([0,T]),
 \label{eq:zt-gamma-weak-Q}
\end{align}
for every continuity point $T<1$ of $\gamma$. If $(t_j)_{j\geq1}$ enumerates the jump points of $\gamma$, the partitions defining $\gamma_n$ may also be chosen so that
\begin{equation*}
 \gamma_n(t_j)-\gamma_n(t_j-)=\gamma(t_j)-\gamma(t_j-),
 \qquad \forall 1\leq j\leq n.
\end{equation*}
In particular, $\lim_{n\to\infty}\gamma_n(s)=\gamma(s)$ at every continuity point $s<1$, while $\lim_{n\to\infty}\gamma_n(t-)=\gamma(t-)$ at every jump point $t<1$.

For every $T<1$, choose a continuity point $T'$ of $\gamma$ with $T<T'<1$. For $j\in\{0,1,2\}$, Lemma~\ref{lem:zt-parabolic-stability}, applied to $u_n=u^{|\cdot|,\gamma_n}$ and $u=u^{|\cdot|,\gamma}$, gives
\begin{equation}
 \lim_{n\to\infty}\partial_x^ju_n=\partial_x^ju
 \quad\text{locally uniformly on }[0,T]\times\R.
 \label{eq:zt-un-smooth-Q}
\end{equation}
For fixed $T<1$, parabolic smoothing and the one-Lipschitz bound also give, uniformly in $n$,
\begin{equation*}
 |u_{n,x}|\leq1\quad\text{and}\quad 0\leq u_{n,xx}\leq M_T
 \quad\text{on }[0,T]\times\R.
\end{equation*}
The control representation with zero control also gives $u_n\geq0$. The initial factors in the bridge formulas for $\gamma_n$ satisfy
\begin{equation}
 \lim_{n\to\infty}c_{0,n}
 =\lim_{n\to\infty}e^{-\gamma_n(0)u_n(0,0)}
 =e^{-\gamma(0)u(0,0)}=c_0,
 \label{eq:zt-c0n-convergence}
\end{equation}
by the convergence of $\gamma_n(0)$ and the uniform convergence of $u_n$.

\smallskip
\noindent\emph{Step 2: Local $C^2$ convergence of the bridge expressions.}

Fix a continuity point $t$ and write a bridge with endpoint $x$ as $B_s^x:=(s/t)x+\widetilde B_s$, where $\widetilde B$ is a centered bridge.  Set
\begin{equation*}
 A_n(x):=\int_{(0,t]}u_n(s,B_s^x)\,\mathrm d\gamma_n(s).
\end{equation*}
For $j\in\{1,2\}$, pathwise differentiation gives
\begin{equation}
 A_n^{(j)}(x)=\int_{(0,t]}\left(\frac{s}{t}\right)^j
 \partial_x^ju_n(s,B_s^x)\,\mathrm d\gamma_n(s).
 \label{eq:zt-An-jth}
\end{equation}
The case $j=1$ is bounded by the one-Lipschitz estimate. For $j=2$, choose a continuity time $T'$ with $t<T'<1$. Equation \eqref{eq:zt-gamma-weak-Q} gives $\sup_n\gamma_n(T')<\infty$, and Lemma~\ref{lem:zt-parabolic-stability}, applied across the positive distance $T'-t$, gives a bound for $u_{n,xx}$ independent of $n$. Since $\sup_n(\mathrm d\gamma_n)((0,t])<\infty$, both quantities in \eqref{eq:zt-An-jth} are therefore uniformly bounded.

Every bridge path has compact range.  For $|x|\leq R$, all arguments $B_s^x$ lie in one random compact set, uniformly in $s$.  For fixed bridge path and $j\in\{0,1,2\}$, put
\begin{equation*}
 g_{n,x}^{(j)}(s)
 :=\left(\frac{s}{t}\right)^j
   \partial_x^ju_n(s,B_s^x)
 \quad\text{and}\quad
 g_x^{(j)}(s)
 :=\left(\frac{s}{t}\right)^j
   \partial_x^ju(s,B_s^x),
\end{equation*}
with the factor $(s/t)^0$ interpreted as one.  Equation \eqref{eq:zt-un-smooth-Q} gives
\begin{equation*}
 \lim_{n\to\infty}
 \sup_{\substack{|x|\leq R\\0\leq s\leq t}}
 |g_{n,x}^{(j)}(s)-g_x^{(j)}(s)|=0.
\end{equation*}
The map $x\mapsto g_x^{(j)}$ is continuous from $[-R,R]$ to $C([0,t])$, because $(s,x)\mapsto B_s^x$ and the relevant spatial derivative of $u$ are jointly continuous on the compact set under consideration.  Hence the family $\{g_x^{(j)}:|x|\leq R\}$ is compact in $C([0,t])$.  Equation \eqref{eq:zt-gamma-weak-Q}, a finite $\varepsilon$-net for this compact family, and uniform boundedness of the masses $\mathrm d\gamma_n([0,t])$ give
\begin{equation*}
 \lim_{n\to\infty}
 \sup_{|x|\leq R}
 \left|\int_{(0,t]}g_x^{(j)}\dd(\gamma_n-\gamma)\right|=0.
\end{equation*}
For the passage from $[0,t]$ to $(0,t]$, we use the already imposed identity $\lim_{n\to\infty}\gamma_n(0)=\gamma(0)$ for the atoms at zero; when $j\geq1$, the factor $(s/t)^j$ also makes the integrand vanish at zero. Combining the last two displays shows pathwise and locally uniformly in $x$ that
\begin{equation}
 \lim_{n\to\infty}A_n^{(j)}=A^{(j)}
 \quad\text{locally uniformly in }x,
 \qquad j\in\{0,1,2\}.
 \label{eq:zt-An-C2-convergence}
\end{equation}
Here $j=0$ uses local uniform convergence of $u_n$, while $j\in\{1,2\}$ uses \eqref{eq:zt-An-jth}. Since
\begin{equation*}
 \partial_xe^{-A_n}=-A_n'e^{-A_n}
 \quad\text{and}\quad
 \partial_{xx}e^{-A_n}=\bigl((A_n')^2-A_n''\bigr)e^{-A_n},
\end{equation*}
the bounds above, $A_n\geq0$, and \eqref{eq:zt-An-C2-convergence} provide an integrable dominator under the bridge expectation. Together with \eqref{eq:zt-c0n-convergence}, this proves local $C^2$ convergence of the complete bridge expression, including its initial prefactor.

Denote the resulting $C^2_{\rm loc}$ bridge limit temporarily by $\widehat Q$. Differentiating the Gaussian prefactor as well as the bridge expectation gives
\begin{equation}
 \sum_{j=0}^2|\partial_x^jQ_n(t,x)|
 \leq C_t(1+|x|^2)p_t(x),
 \label{eq:zt-Q-Gaussian-derivative-bound}
\end{equation}
with $C_t$ uniform in $n$.  The same estimate holds for $\widehat Q$.

\smallskip
\noindent\emph{Step 3: Identification of $\widehat Q$ with $Q$ and treatment of jump points.}

It remains to identify $\widehat Q$ with the quantity defined from the actual diffusion density.  Put
\begin{equation*}
 b_n(s,x):=\gamma_n(s)u_{n,x}(s,x)\quad\text{and}\quad
 b(s,x):=\gamma(s)u_x(s,x).
\end{equation*}
Choose $T'$ with $t<T'<1$ at which $\gamma$ is continuous.  The step approximations may be chosen with $\sup_n\gamma_n(T')<\infty$.  The comparison estimate and the last part of Lemma~\ref{lem:zt-parabolic-stability} give
\begin{equation*}
 \lim_{n\to\infty}
 \sup_{\substack{0\leq s\leq t\\x\in\R}}
 |u_{n,x}(s,x)-u_x(s,x)|=0.
\end{equation*}
Consequently,
\begin{align}
 \int_0^t\|b_n(s,\cdot)-b(s,\cdot)\|_\infty\dd s
 &\leq\|\gamma_n-\gamma\|_{L^1(0,t)}
 +\|\gamma\|_{L^1(0,t)}
   \sup_{s,x}|u_{n,x}-u_x|,\notag\\
 \lim_{n\to\infty}\int_0^t
 \|b_n(s,\cdot)-b(s,\cdot)\|_\infty\dd s&=0.
 \label{eq:zt-drift-L1-convergence}
\end{align}
All these drifts are uniformly bounded on $[0,t]\times\R$.

For completeness, this drift convergence implies convergence of the diffusion laws directly.  Under Wiener measure, let
\begin{equation*}
 Z_n:=\exp\left\{\int_0^tb_n(s,W_s)\dd W_s
 -\frac12\int_0^tb_n(s,W_s)^2\dd s\right\},
\end{equation*}
and define $Z$ with $b$ in place of $b_n$.  The uniform drift bound and \eqref{eq:zt-drift-L1-convergence} imply
\begin{equation*}
 \lim_{n\to\infty}\int_0^t\|b_n-b\|_\infty^2\dd s=0.
\end{equation*}
It\^o's isometry makes the stochastic integrals converge in $L^2$, and the finite-variation terms converge in $L^1$.  Moreover, for some $p>1$,
\begin{equation*}
 \sup_n\E Z_n^p<\infty,
\end{equation*}
because $Z_n^p$ is a stochastic exponential with integrand $pb_n$ times the bounded factor $\exp\{\frac12(p^2-p)\int_0^tb_n^2\}$. Thus $\lim_{n\to\infty}\|Z_n-Z\|_{L^1}=0$. Girsanov's theorem shows that the one-time laws with densities $\rho_t^n$ converge weakly to the law with density $\rho_t$.

Since the fixed $t\in(0,1)$ is a continuity point of $\gamma$,
\begin{equation*}
 \widehat\rho(t,x):=e^{\gamma(t)u(t,x)}\widehat Q(t,x)
 \quad\text{and}\quad
 \lim_{n\to\infty}\rho_t^n(x)
 =\lim_{n\to\infty}e^{\gamma_n(t)u_n(t,x)}Q_n(t,x)
 =\widehat\rho(t,x)
\end{equation*}
locally uniformly. Testing against compactly supported continuous functions and comparing with the preceding weak convergence gives $\widehat\rho(t,\cdot)=\rho_t$ and therefore $\widehat Q(t,\cdot)=\rho_te^{-\gamma(t)u(t,\cdot)}$. This proves \eqref{eq:zt-Q-C2-convergence} and \eqref{eq:zt-Stieltjes-bridge} at every continuity point $t$ of $\gamma$.

If $t\in(0,1)$ is a jump point of $\gamma$, repeat Steps 2 and 3 with $(0,t)$ in place of $(0,t]$ and with $\gamma_n(t-)$ and $\gamma(t-)$ in place of $\gamma_n(t)$ and $\gamma(t)$. This proves the asserted local $C^2$ convergence of $Q_n(t-,\cdot)$ and the identity
\begin{equation*}
 Q(t-,x)=c_0p_t(x)\E_{0\to x}^{\mathrm{BB},t}
 \exp\left\{-\int_{(0,t)}u(s,B_s)\,\mathrm d\gamma(s)\right\},
 \qquad \forall x\in\R.
\end{equation*}
Multiplying this identity by $e^{-(\gamma(t)-\gamma(t-))u(t,x)}$ and applying \eqref{eq:zt-Q-jump} proves \eqref{eq:zt-Stieltjes-bridge} at $t$.

\subsection{Proof of Lemma~\ref{lem:zt-tails}}\label{a.pf.lem:zt-tails}

\noindent\emph{Step 1: Gaussian bounds for $1-\mathsf B(t,x)$.}

Fix $t\in I$ and put
\begin{equation*}
 T:=1-t\quad\text{and}\quad A:=\int_t^1\gamma(s)\dd s<\infty.
\end{equation*}
For $x\in\R$, let $Y^{t,x}:=(Y_s^{t,x})_{t\leq s\leq1}$ solve
\begin{equation}
 \dd Y_s^{t,x}=\gamma(s)u_x(s,Y_s^{t,x})\dd s+\dd W_s,
 \qquad t<s\leq1,
 \quad\text{and}\quad Y_t^{t,x}=x.
 \label{eq:zt-restarted-optimal-SDE}
\end{equation}
Thus $Y^{t,x}$ has the same time-dependent transition dynamics as $X$ in \eqref{eq:zt-optimal-SDE}; the process $X$ starts from $(0,0)$, whereas $Y^{t,x}$ starts from $(t,x)$. Since $|u_x|\leq1$,
\begin{equation*}
 \left|\int_t^1\gamma(s)u_x(s,Y_s^{t,x})\dd s\right|
 \leq\int_t^1\gamma(s)\dd s=A.
\end{equation*}
Differentiating \eqref{eq:zt-general-PDE} in $x$ and applying It\^o's formula along \eqref{eq:zt-restarted-optimal-SDE} gives
\begin{equation*}
 \dd u_x(s,Y_s^{t,x})=u_{xx}(s,Y_s^{t,x})\dd W_s,
 \qquad t<s<1.
\end{equation*}
The bounded process $s\mapsto u_x(s,Y_s^{t,x})$ is therefore a martingale. With $\operatorname{sign}(0):=0$, its terminal value is $\operatorname{sign}(Y_1^{t,x})$, and hence
\begin{equation*}
 \mathsf B(t,x)
 \stackref{=}{\eqref{eq:zt-BCD-shorthand}}
 u_x(t,x)=\E\operatorname{sign}(Y_1^{t,x}).
\end{equation*}
Since $1-\operatorname{sign}(z)$ equals $2$, $1$, or $0$ according as $z<0$, $z=0$, or $z>0$, respectively, we obtain
\begin{equation*}
 1-\mathsf B(t,x)
 =\E[1-\operatorname{sign}(Y_1^{t,x})]
 =2\PP(Y_1^{t,x}<0)+\PP(Y_1^{t,x}=0).
\end{equation*}
The deterministic drift bound gives the event inclusions
\begin{equation*}
 \{W_1-W_t<-x-A\}\subseteq\{Y_1^{t,x}<0\}
 \subseteq\{Y_1^{t,x}\leq0\}
 \subseteq\{W_1-W_t\leq-x+A\}.
\end{equation*}
Thus, using the standard Gaussian upper-tail function in \eqref{eq:standard-Gaussian-functions},
\begin{equation}
 2\overline\phi\left(\frac{x+A}{\sqrt T}\right)
 \leq1-\mathsf B(t,x)\leq
 2\overline\phi\left(\frac{x-A}{\sqrt T}\right).
 \label{eq:zt-B-Gaussian-bracket}
\end{equation}

\smallskip
\noindent\emph{Step 2: Proof of \eqref{eq:zt-C-Gaussian-tail}.}

For the remainder of the proof, set $B:=\mathsf B(t,x)$. Then $\mathsf c(t,B)=\mathsf C(t,x)$ by \eqref{eq:zt-slope-curvature}. From \eqref{eq:zt-zero-temp-five} and $\mathsf z_B=m+\mathsf J$,
\begin{equation}
 \partial_x\mathsf C(t,x)
 \stackref{=}{\eqref{eq:zt-slope-curvature},\,\eqref{eq:zt-slope-variables}}
 -2\mathsf C(t,x)\widehat{\mathsf z}(t,x)
 \quad\text{and}\quad
 \partial_x\widehat{\mathsf z}(t,x)
 \stackref{=}{\eqref{eq:zt-slope-variables}}
 \mathsf C(t,x)\bigl(m+\mathsf J(t,B)\bigr)
 \stackref{\geq}{\eqref{eq:zt-zero-temp-five}}0.
 \label{eq:zt-Cz-x-relations}
\end{equation}
Therefore $x\mapsto\mathsf C(t,x)$ is decreasing on $(0,\infty)$ and
\begin{equation*}
 \partial_{xx}\log\mathsf C(t,x)=-2\partial_x\widehat{\mathsf z}(t,x)\leq0.
\end{equation*}
Put
\begin{equation*}
 I_t(x):=1-\mathsf B(t,x)=\int_x^\infty\mathsf C(t,y)\dd y.
\end{equation*}
Choose $h>2A$.  Monotonicity gives
\begin{equation}
 \frac{I_t(x)-I_t(x+h)}h\leq\mathsf C(t,x)
 \leq\frac{I_t(x-h)}h.
 \label{eq:zt-C-from-integrated-tail}
\end{equation}
For every $z>0$, the Gaussian Mills bounds are
\begin{equation}
 \frac{z}{1+z^2}\frac{e^{-z^2/2}}{\sqrt{2\pi}}
 \leq\overline\phi(z)
 \leq\frac1z\frac{e^{-z^2/2}}{\sqrt{2\pi}}.
 \label{eq:zt-Gaussian-Mills-bounds}
\end{equation}
Indeed, the upper bound follows by replacing $1$ with $r/z$ in the integral defining $\overline\phi(z)$, while integration by parts gives
\begin{equation*}
 \overline\phi(z)=\frac{e^{-z^2/2}}{z\sqrt{2\pi}}
 -\frac1{\sqrt{2\pi}}\int_z^\infty\frac{e^{-r^2/2}}{r^2}\dd r
 \geq\frac{e^{-z^2/2}}{z\sqrt{2\pi}}-\frac{\overline\phi(z)}{z^2},
\end{equation*}
which proves the lower bound. For all sufficiently large $x$, the upper bound in \eqref{eq:zt-B-Gaussian-bracket} at $x+h$, the lower bound at $x$, and \eqref{eq:zt-Gaussian-Mills-bounds} give
\begin{equation*}
 \frac{I_t(x+h)}{I_t(x)}
 \leq C_0\,
 \frac{\overline\phi((x+h-A)/\sqrt T)}
      {\overline\phi((x+A)/\sqrt T)}
 \stackref{\leq}{\eqref{eq:zt-Gaussian-Mills-bounds}}
 C_0e^{-x(h-2A)/T}.
\end{equation*}
Thus $h-2A>0$ implies $I_t(x+h)\leq I_t(x)/2$ for all large $x$. Combining this fact with \eqref{eq:zt-C-from-integrated-tail} and applying \eqref{eq:zt-Gaussian-Mills-bounds} once more proves \eqref{eq:zt-C-Gaussian-tail}. The constants are uniform for $t\in I$ after taking $h>2\sup_{t\in I}A(t)$.

\smallskip
\noindent\emph{Step 3: Bounds for $\widehat{\mathsf z}$ and its first two spatial derivatives.}

Let $\ell_t(x):=-\log\mathsf C(t,x)$. Then $\ell_t$ is increasing and convex, $\ell_t'(x)=2\widehat{\mathsf z}(t,x)$, and \eqref{eq:zt-C-Gaussian-tail} gives
\begin{equation*}
 \ell_t(x)=\frac{x^2}{2T}+O(x+\log x).
\end{equation*}
Convexity yields
\begin{equation*}
 0\leq\ell_t'(x)\leq\frac{\ell_t(2x)-\ell_t(x)}x=O(x),
\end{equation*}
so $\widehat{\mathsf z}(t,x)=O(x)$. With $y(x):=\partial_x\widehat{\mathsf z}(t,x)$, differentiation of \eqref{eq:zt-Cz-x-relations} gives
\begin{equation*}
 y'(x)=\partial_{xx}\widehat{\mathsf z}(t,x)
 =-2\widehat{\mathsf z}(t,x)y(x)+\mathsf C(t,x)^2\mathsf J_B(t,B).
\end{equation*}
Since $\mathsf C(t,x)\mathsf J_B(t,B)\leq3\widehat{\mathsf z}(t,x)\mathsf J(t,B)$ and $\mathsf C(t,x)\mathsf J(t,B)\leq\mathsf C(t,x)(m+\mathsf J(t,B))=y(x)$,
\begin{equation}
 -2\widehat{\mathsf z}(t,x)y(x)
 \leq y'(x)
 \stackref{\leq}{\eqref{eq:zt-zero-temp-five}}
 \widehat{\mathsf z}(t,x)y(x).
 \label{eq:zt-y-differential-bracket}
\end{equation}
For $s\in[x-x^{-1},x]$, the already proved bound gives $\widehat{\mathsf z}(t,s)\leq C_0x$. If $y(x)=0$, the desired upper bound is immediate. Otherwise, integrating $y'/y\leq\widehat{\mathsf z}(t,\cdot)$ backward using \eqref{eq:zt-y-differential-bracket} gives $y(s)\geq c\,y(x)$ with a constant independent of large $x$. Hence
\begin{equation*}
 \frac{c}{x}y(x)\leq\int_{x-1/x}^xy(s)\dd s
 =\widehat{\mathsf z}(t,x)-\widehat{\mathsf z}(t,x-1/x)\leq\widehat{\mathsf z}(t,x)\leq C_0x.
\end{equation*}
Thus $\partial_x\widehat{\mathsf z}(t,x)=O(x^2)$, and \eqref{eq:zt-y-differential-bracket} then gives $\partial_{xx}\widehat{\mathsf z}(t,x)=O(x^3)$.

\smallskip
\noindent\emph{Step 4: Proof of \eqref{eq:zt-score-tail}.}

Inside the gap,
\begin{equation*}
 Q(t,\cdot)=P_{t-a}Q(a+,\cdot).
\end{equation*}
The incoming object may be viewed as a finite log-concave measure; if $a=0$, it is a multiple of $\delta_0$.  Set $s:=t-a$ and define the posterior probability measure
\begin{equation*}
 \mu_x(\mathrm d y):=\frac{p_s(x-y)Q(a+,\mathrm d y)}{Q(t,x)}.
\end{equation*}
Differentiating the Gaussian convolution gives the heat-score identities
\begin{equation*}
 \mathfrak r(t,x)=\frac{x-\E_{\mu_x}Y}{s}
 \quad\text{and}\quad
 \mathfrak r_x(t,x)=\frac1s-\frac{\Var_{\mu_x}(Y)}{s^2}.
\end{equation*}
The variance term gives $\mathfrak r_x\leq s^{-1}$, and log-concavity gives $\mathfrak r_x=-(\log Q)_{xx}\geq0$.  Evenness gives $\mathfrak r(t,0)=0$, so $0\leq\mathfrak r(t,x)\leq x/(t-a)$.  Since $H=\mathfrak r^2-\mathfrak r_x$, this proves \eqref{eq:zt-score-tail}.

\smallskip
\noindent\emph{Step 5: Proof of \eqref{eq:zt-rho-Gaussian-tail}.}

The control representation gives $u\geq0$, and one-Lipschitzness gives $u(t,x)\leq u(t,0)+|x|$.  The exponential in \eqref{eq:zt-Stieltjes-bridge} is at most one, so $Q(t,x)\leq c_0p_t(x)$.  Because $\rho_t=e^{mu(t,\cdot)}Q(t,\cdot)$ on the gap,
\begin{equation*}
 \rho_t(x)
 \stackref{\leq}{\eqref{eq:zt-Stieltjes-bridge}}
 C_I\exp\left\{-\frac{x^2}{2t}+m|x|\right\},
\end{equation*}
which proves \eqref{eq:zt-rho-Gaussian-tail} uniformly on $I$.

\smallskip
\noindent\emph{Step 6: Weighted integrability, uniform tails, and the boundary limit.}

For $t\in I$ and $0\leq B<1$, set $x:=x(t,B)$ and define
\begin{align*}
 \kappa(t,B)&:=\mathsf K(t,B)B=\mathsf z(t,B)-mB,\\
 N(t,B)&:=\mathfrak r(t,x)+\kappa(t,B)
 \quad\text{and}\quad w(t,B):=2\rho_t(x)\mathsf c(t,B).
\end{align*}
Uniformly as $B\uparrow1$ and $x=x(t,B)\to+\infty$,
\begin{align*}
 \kappa(t,B),N(t,B)&=O(x)\quad\text{and}\quad \Phi(t,B)=O(x^2),\\
 \mathsf c(t,B)\mathsf J(t,B)&=\partial_x\widehat{\mathsf z}(t,x)-m\mathsf c(t,B)=O(x^2),\\
 \mathsf c(t,B)^2\mathsf J_B(t,B)&=\partial_{xx}\widehat{\mathsf z}(t,x)+2\widehat{\mathsf z}(t,x)\partial_x\widehat{\mathsf z}(t,x)=O(x^3),\\
 \cG(t,B)&=6\widehat{\mathsf z}(t,x)^2\mathsf c(t,B)\mathsf J(t,B)-2\widehat{\mathsf z}(t,x)\mathsf c(t,B)^2\mathsf J_B(t,B)=O(x^4).
\end{align*}
In particular, the quantities
\begin{align*}
 R_1(t,B)&:=\kappa(t,B)\mathfrak r(t,x)+\frac12\kappa(t,B)^2,\\
 R_0(t,B)&:=2\mathsf c(t,B)\mathsf J(t,B)-\frac m2\mathsf c(t,B)-\frac12\mathsf z(t,B)^2,\\
 \text{and}\qquad L(t,B)&:=\mathsf c(t,B)\mathsf J(t,B)-\mathsf z(t,B)\bigl(N(t,B)+\mathsf z(t,B)\bigr).
\end{align*}
are $O(x^2)$. The individual $\mathsf J$ and $\mathsf J_B$ need not have polynomial growth; only the displayed composites are asserted to do so. Since $\dd B=\mathsf c(t,B)\dd x$, the Gaussian estimates for both $\rho$ and $\mathsf C$ give, for every $p\geq0$,
\begin{equation}
 \sup_{t\in I}\int_0^1w(t,B)
   \bigl(1+|x(t,B)|^p\bigr)\dd B
 =\sup_{t\in I}2\int_0^\infty\rho_t(x)\mathsf C(t,x)^2(1+x^p)\dd x<\infty.
 \label{eq:zt-weighted-polynomial-integrability}
\end{equation}
The same Gaussian bounds, together with \eqref{eq:zt-B-Gaussian-bracket}, imply the uniform tail form
\begin{equation}
 \lim_{B_0\uparrow1}\sup_{t\in I}
 \int_{B_0}^1w(t,B)\bigl(1+|x(t,B)|^p\bigr)\dd B=0.
 \label{eq:zt-weighted-uniform-tail}
\end{equation}
Indeed, the inverse points $x(t,B_0)$ tend to $+\infty$ uniformly for $t\in I$ as $B_0\uparrow1$, and the right side of \eqref{eq:zt-weighted-polynomial-integrability} then has a uniform Gaussian tail. Moreover,
\begin{align}
 &\lim_{B\uparrow1}\sup_{t\in I}|w(t,B)\mathsf c(t,B)\mathsf z(t,B)P(x(t,B))| \notag\\
 &\qquad=\lim_{B\uparrow1}\sup_{t\in I}|2\rho_t(x(t,B))\mathsf c(t,B)^2\mathsf z(t,B)P(x(t,B))|=0.
 \label{eq:zt-wCz-boundary}
\end{align}
for every fixed polynomial $P$.  Equations \eqref{eq:zt-weighted-polynomial-integrability}--\eqref{eq:zt-wCz-boundary} are the precise uniform-integrability and boundary statements used in Subsection~\ref{sec:zt-crossing}; they prove the final assertion of the lemma.

\subsection{Proof of Lemma~\ref{lem:zt-polynomial-moment-regularity}}\label{a.pf.lem:zt-polynomial-moment-regularity}

Fix $T'$ with $T<T'<1$. By Lemma~\ref{lem:zt-parabolic-stability}, applied to the constant sequences $g_n=g=|\cdot|$ and $\alpha_n=\alpha=\gamma$, for every $k\geq1$ we have
\begin{equation}
 \sup_{\substack{0\leq t\leq T\\x\in\R}}|\partial_x^ku(t,x)|<\infty,
 \label{eq:zt-polynomial-moment-derivative-bound}
\end{equation}
and $(t,x)\mapsto\partial_x^ku(t,x)$ is continuous on $[0,T]\times\R$. Since the drift in \eqref{eq:zt-optimal-SDE} is bounded,
\begin{equation*}
 \E|X_t-X_s|^2\leq2|t-s|+2\|\gamma\|_{L^\infty(0,T)}^2|t-s|^2,
 \qquad \forall s,t\in[0,T].
\end{equation*}
Consequently, \eqref{eq:zt-polynomial-moment-derivative-bound} and dominated convergence show that $M_P$ is continuous.

Suppose that $\gamma\in C^r([0,T])$. Let $B$ be a standard Brownian motion with $B_0=0$ and define
\begin{equation*}
 \mathcal E_t:=\exp\left\{\int_0^t\gamma(s)u_x(s,B_s)\dd B_s-\frac12\int_0^t\gamma(s)^2u_x(s,B_s)^2\dd s\right\}.
\end{equation*}
The boundedness of $\gamma$ and $u_x$ implies Novikov's condition. Moreover, It\^o's formula, the PDE in \eqref{eq:zt-general-PDE}, and Stieltjes integration by parts give
\begin{equation*}
 \log\mathcal E_t=\int_{[0,t]}\bigl(u(t,B_t)-u(s,B_s)\bigr)\mathrm d\gamma(s).
\end{equation*}
Thus Girsanov's theorem identifies the weighted Brownian expectation in \cite[Proposition~1(ii) and Lemma~2, equations~(37)--(38)]{AuffingerChenProperties} with
\begin{equation*}
 M_P(t)=\E\left[P\bigl(\partial_xu(t,B_t),\ldots,\partial_x^Nu(t,B_t)\bigr)\mathcal E_t\right].
\end{equation*}
With $\xi''\equiv1$ and the distribution function in that paper replaced by $\gamma$, the proof of the cited lemma gives
\begin{equation}
 \frac{\dd}{\dd t}M_P(t)=\E\left[(\mathcal L_{\gamma(t)}P)\bigl(Z_1(t),\ldots,Z_{N+1}(t)\bigr)\right],
 \label{eq:zt-polynomial-moment-derivative}
\end{equation}
where
\begin{align*}
 (\mathcal L_gP)(z)&:=-\frac g2\sum_{k=1}^N\left(\sum_{j=1}^{k-1}\binom{k}{j}z_{j+1}z_{k-j+1}\right)\partial_kP(z)+\frac12\sum_{k,\ell=1}^Nz_{k+1}z_{\ell+1}\partial_{k\ell}P(z).
\end{align*}
The cases $k=1,2$ of the corresponding stochastic identities also appear in \cite[Lemma~3, equations~(14)--(15)]{ChenHandschyLerman}. The proof of the formula in \cite[Lemma~2]{AuffingerChenProperties} is local in time and uses only the PDE and bounded spatial derivatives. It therefore applies to $u^{|\cdot|,\gamma}$ on $[0,T]$: the nonsmooth terminal datum $|\cdot|$ is separated from this interval by $1-T'>0$, and \eqref{eq:zt-polynomial-moment-derivative-bound} supplies the required bounds.

The right-hand side of \eqref{eq:zt-polynomial-moment-derivative} is a finite linear combination of polynomial moments of spatial derivatives, with coefficients depending polynomially on $\gamma(t)$. Repeated application of \eqref{eq:zt-polynomial-moment-derivative}, as in \cite[proof of Theorem~2(ii)]{AuffingerChenProperties}, gives $M_P\in C^{r+1}([0,T])$. Its integrated form gives the asserted one-sided derivatives at zero and $T$.

\section{\texorpdfstring{Relations to \cite{AIxivPreprint,Lopatto2026v2,Lopatto2026}}{Relations to the aiXiv preprint and Lopatto's second and third versions}}
\label{app:development}

Lopatto's second arXiv version~\cite{Lopatto2026v2}, submitted on July~15, 2026, proved that for every $\beta>1$ the support of the positive-temperature Parisi measure contains an interval beginning at zero; see \cite[Theorem~1.1]{Lopatto2026v2}. On July~16, 2026, before Lopatto's third arXiv version \cite{Lopatto2026} became publicly available, an earlier version of the present paper was posted on aiXiv as \cite{AIxivPreprint}. Taking Lopatto's second version \cite{Lopatto2026v2} as an input, \cite{AIxivPreprint} proved that the entire positive-temperature support is an interval and also proved that the zero-temperature Stieltjes measure has full support after closure at one. Thus its two support conclusions went beyond the result stated in Lopatto's second version \cite{Lopatto2026v2}.

Lopatto's third version~\cite{Lopatto2026} subsequently established the complete positive-temperature interval-and-atom characterization stated in Theorem~\ref{thm:finite-temperature-structure}. Its transformed-density analysis and arbitrary-gap crossing argument have mathematical overlap with parts of the positive- and zero-temperature proofs in \cite{AIxivPreprint}. The present manuscript was reorganized in response to this development. The overlapping positive-temperature support proof was removed and replaced by the citation in Theorem~\ref{thm:finite-temperature-structure}; the positive-temperature part now contains only the quantitative endpoint results not supplied by that theorem. The present manuscript also strengthens \cite{AIxivPreprint} by proving smooth absolute continuity of the zero-temperature Stieltjes measure and sharpening the positive-temperature endpoint scale to $\beta^{-2}$.

The roles of Lopatto's second and third versions in the zero-temperature proof should be distinguished. The finite Cole--Hopf inequalities proved in Appendix~\ref{app:finite-KJ} already appear in \cite[Proposition 3.6 and Equation (3.39)]{Lopatto2026v2}. They were used as an input in \cite{AIxivPreprint}; the purpose of Appendix~\ref{app:finite-KJ} is to reproduce their proof locally so that the present zero-temperature argument is self-contained. Proposition~\ref{prop:zt-zero-temp-KJ} scales these inequalities to obtain \eqref{eq:zt-zero-temp-five}. This zero-temperature scaling and the uniform tail estimates now proved in Appendix~\ref{a.pf.lem:zt-tails}, including the estimates for $Q$ and $\rho$, were already present with essentially the same arguments in \cite{AIxivPreprint}. In particular, neither of these parts was obtained from Lopatto's third version. The role of the tail estimates is analogous to the endpoint control in \cite[Lemma 3.8 and the proof of Proposition 4.1]{Lopatto2026v2}, but their zero-temperature formulation is the one given in \cite{AIxivPreprint}.

There is a separate similarity with Lopatto's third version. After the change of variables $B=\mathsf B(t,x)$, the identity \eqref{eq:zt-Gamma-second-I} and the differentiation of its right-hand side parallel \cite[Proposition 9.1]{Lopatto2026}. However, this calculation and the resulting arbitrary-gap crossing argument were already contained in \cite{AIxivPreprint} before Lopatto's third version became publicly available. Thus \cite[Proposition 9.1]{Lopatto2026} is not an input to the zero-temperature proof but a subsequently released parallel positive-temperature argument.

In summary, Theorem~\ref{thm:finite-temperature-structure} is the only input taken from Lopatto's third version~\cite{Lopatto2026}. Apart from this theorem and the well-established results cited in the text, the manuscript is self-contained: its new arguments develop \cite{AIxivPreprint} and require from Lopatto only the material from his second version~\cite{Lopatto2026v2} that is reproved in Appendix~\ref{app:finite-KJ}.

\bigskip

\noindent\textbf{Acknowledgements.} HBC gratefully acknowledges funding from the NYU Shanghai Start-Up Fund and support from the NYU--ECNU Institute of Mathematical Sciences at NYU Shanghai. HBC warmly thanks Zijie Zhuang for very helpful discussions and thanks Patrick Lopatto for encouraging the preparation of this manuscript.

\end{document}